\documentclass[times,sort&compress,3p]{elsarticle}
\journal{Journal of Multivariate Analysis}

\usepackage{amsmath,amsfonts,amssymb,amsthm,graphicx,hyperref,url}
\usepackage[labelfont=bf]{caption}

\usepackage{mathtools}
\mathtoolsset{showonlyrefs=true} 
\usepackage{dsfont} 
\usepackage{mathrsfs} 
\usepackage{subcaption}
\usepackage{geometry}
\theoremstyle{plain}
\newtheorem{theorem}{Theorem}
\newtheorem{proposition}{Proposition}
\newtheorem{lemma}{Lemma}
\newtheorem{corollary}{Corollary}
\newtheorem{remark}{Remark}

\newcommand{\N}{\mathbb{N}}
\newcommand{\R}{\mathbb{R}}

\newcommand{\QQ}{\mathbb{Q}}
\newcommand{\PP}{\mathbb{P}}
\newcommand{\EE}{\mathbb{E}}
\newcommand{\BB}{\mathbb{B}\mathrm{ias}}
\newcommand{\VV}{\mathbb{V}\mathrm{ar}}

\newcommand{\bb}[1]{\boldsymbol{#1}}

\newcommand{\vecp}{\mathrm{vecp}}

\newcommand{\mat}[1]{\mathbb{#1}}

\newcommand{\OO}{\mathcal O}
\newcommand{\oo}{\mathrm{o}}
\newcommand{\leqdef}{\vcentcolon=}
\newcommand{\reqdef}{=\vcentcolon}
\newcommand{\rd}{{\rm d}}
\newcommand{\ind}{\mathds{1}}
\newcommand{\e}{\varepsilon}

\allowdisplaybreaks

\begin{document}

\begin{frontmatter}

    \title{A symmetric matrix-variate normal local approximation for the Wishart distribution and some applications}%

    \author[a1,a2]{Fr\'ed\'eric Ouimet\texorpdfstring{\corref{cor1}\fnref{fn1}}{)}}%

    \address[a1]{California Institute of Technology, Pasadena, CA 91125, USA.}%
    \address[a2]{McGill University, Montreal, QC H3A 0B9, Canada.}%

    \cortext[cor1]{Corresponding author}%
    \ead{frederic.ouimet2@mcgill.ca}%


    \begin{abstract}
        The noncentral Wishart distribution has become more mainstream in statistics as the prevalence of applications involving sample covariances with underlying multivariate Gaussian populations as dramatically increased since the advent of computers.
        Multiple sources in the literature deal with local approximations of the noncentral Wishart distribution with respect to its central counterpart.
        However, no source has yet developed explicit local approximations for the (central) Wishart distribution in terms of a normal analogue, which is important since Gaussian distributions are at the heart of the asymptotic theory for many statistical methods.
        In this paper, we prove a precise asymptotic expansion for the ratio of the Wishart density to the symmetric matrix-variate normal density with the same mean and covariances. The result is then used to derive an upper bound on the total variation between the corresponding probability measures and to find the pointwise variance of a new density estimator on the space of positive definite matrices with a Wishart asymmetric kernel.
        For the sake of completeness, we also find expressions for the pointwise bias of our new estimator, the pointwise variance as we move towards the boundary of its support, the mean squared error, the mean integrated squared error away from the boundary, and we prove its asymptotic normality.
    \end{abstract}

    \begin{keyword} 
        asymmetric kernel, asymptotic statistics, density estimation, expansion, local approximation, matrix-variate normal, multivariate associated kernel, normal approximation, smoothing,  total variation, Wishart distribution
        \MSC[2020]{Primary: 62E20 Secondary: 62H10, 62H12, 62B15, 62G05, 62G07}
    \end{keyword}

\end{frontmatter}

\section{Introduction}\label{sec:intro}

    Let $d\in \N$ be given.
    Define the space of (real) symmetric matrices of size $d\times d$ and the space of (real symmetric) positive definite matrices of size $d\times d$ as follows:
    \begin{align}
        &\mathcal{S}^{\hspace{0.3mm}d} \leqdef \left\{\mat{M}\in \R^{d\times d} : \text{$\mat{M}$ is symmetric}\right\}, \\
        &\mathcal{S}_{++}^{\hspace{0.3mm}d} \leqdef \left\{\mat{M}\in \R^{d\times d} : \text{$\mat{M}$ is symmetric and positive definite}\right\}. \label{eq:def:positive.definite.matrices}
    \end{align}
    For $\nu > d - 1$ and $\mat{S}\in \mathcal{S}_{++}^{\hspace{0.3mm}d}$, the density function of the $\mathrm{Wishart}_d(\nu,\mat{S})$ distribution is defined by
    \begin{equation}\label{eq:Wishart.density}
        \begin{aligned}
            K_{\nu,\mat{S}}(\mat{X})
            &\leqdef \frac{|\mat{S}^{-1} \mat{X}|^{\nu/2 - (d + 1)/2} \exp\left(-\frac{1}{2}\mathrm{tr}(\mat{S}^{-1} \mat{X})\right)}{2^{\nu d / 2} |\mat{S}|^{(d + 1)/2} \pi^{\hspace{0.3mm}d(d-1)/4} \prod_{i=1}^d \Gamma(\frac{1}{2} (\nu - (i + 1)) + 1)}, \quad \mat{X}\in \mathcal{S}_{++}^{\hspace{0.3mm}d},
        \end{aligned}
    \end{equation}
    where $\nu$ is the number of degrees of freedom, $\mat{S}$ is the scale matrix, and
    \begin{equation}
        \Gamma(a) \leqdef \int_0^{\infty} t^{a - 1} e^{-t} \rd t, \quad a > 0,
    \end{equation}
    denotes the Euler gamma function.
    The mean and covariance matrix for the vectorization of $\mat{W}\sim \mathrm{Wishart}_d(\nu,\mat{S})$, namely
    \begin{equation}\label{eq:vectorization}
        \vecp(\mat{W}) \leqdef (\mat{W}_{11}, \mat{W}_{12}, \mat{W}_{22}, \dots, \mat{W}_{1d}, \dots, \mat{W}_{dd})^{\top},
    \end{equation}
    ($\mathrm{vecp}(\cdot)$ is the operator that stacks the columns of the upper triangular portion of a symmetric matrix on top of each other) are well known to be:
    \begin{equation}\label{eq:expectation.explicit.estimate}
        \EE[\vecp(\mat{W})] = \nu \, \vecp(\mat{S}) \quad \text{(alternatively, $\EE[\mat{W}] = \nu \hspace{0.3mm} \mat{S}$)}
    \end{equation}
    and
    \begin{equation}\label{eq:covariance.explicit.estimate}
        \VV(\vecp(\mat{W})) = B_d^{\top} (\hspace{-0.5mm} \sqrt{2\nu} \, \mat{S} \otimes \hspace{-1mm} \sqrt{2\nu} \, \mat{S}) B_d,
    \end{equation}
    where $\mathrm{I}_d$ is the identity matrix of order $d$, $B_d$ is a $d^{\hspace{0.3mm}2} \times \frac{1}{2} d(d + 1)$ transition matrix (see \citet[p.11]{{Gupta_Nagar_1999}} for the precise definition), and $\otimes$ denotes the Kronecker product.

    Multiple sources in the literature deal with local approximations of the noncentral Wishart distribution with respect to the (central) Wishart distribution, see, e.g., \citet{MR326925,MR648656,MR1370289,doi:10.1080/03610919908813585}.
    However, no source has yet developed explicit local approximations for the (central) Wishart distribution in terms of a normal analogue, which is important since Gaussian distributions are at the heart of the asymptotic theory for many statistical methods.

    The main goal of our paper (Theorem~\ref{thm:p.k.expansion}) is to establish an asymptotic expansion for the ratio of the Wishart density \eqref{eq:Wishart.density} to the symmetric matrix-variate normal (SMN) density with the same mean and covariances.
    According to \citet[Eq.(2.5.8)]{Gupta_Nagar_1999}, the density of the $\mathrm{SMN}_{d\times d}(\nu \hspace{0.3mm} \mat{S}, B_d^{\top} (\hspace{-0.5mm} \sqrt{2\nu} \, \mat{S} \otimes \hspace{-1mm} \sqrt{2\nu} \, \mat{S}) B_d)$ distribution is
    \begin{equation}\label{eq:sym.matrix.normal.density}
        g_{\nu,\mat{S}}(\mat{X}) = \frac{\exp\left(-\frac{1}{2} \mathrm{tr}(\Delta_{\nu,\mat{S}}^2)\right)}{\sqrt{(2\pi)^{d(d + 1)/2} \, |B_d^{\top} (\hspace{-0.5mm} \sqrt{2\nu} \, \mat{S} \otimes \hspace{-1mm} \sqrt{2\nu} \, \mat{S}) B_d|}} = \frac{\exp\left(-\frac{1}{2} \mathrm{tr}(\Delta_{\nu,\mat{S}}^2)\right)}{\sqrt{2^d \pi^{\hspace{0.3mm}d(d + 1) / 2} |\hspace{-0.5mm} \sqrt{2\nu} \, \mat{S}|^{d + 1}}}, \quad \mat{X}\in \mathcal{S}^{\hspace{0.3mm}d},
    \end{equation}
    where the last equality follows from \citet[Eq.(1.2.18)]{Gupta_Nagar_1999}, and
    \begin{equation}
        \Delta_{\nu,\mat{S}} \leqdef (\hspace{-0.5mm} \sqrt{2\nu} \, \mat{S})^{-1/2} (\mat{X} - \nu \hspace{0.3mm} \mat{S}) (\hspace{-0.5mm} \sqrt{2\nu} \, \mat{S})^{-1/2}.
    \end{equation}
    Rewritings of the density \eqref{eq:sym.matrix.normal.density} are provided on page 71 of \citet{Gupta_Nagar_1999} using the vectorization operators $\mathrm{vec}(\cdot)$ and $\mathrm{vecp}(\cdot)$.
    For example, we can rewrite $g_{\nu,\mat{S}}(\mat{X})$ in terms of $\mathrm{vecp}(\mat{X})$ as follows:
    \begin{equation}\label{eq:sym.matrix.normal.density.vecp}
        g_{\nu,\mat{S}}(\mat{X}) = \frac{\exp\left(-\frac{1}{2} (\mathrm{vecp}(\mat{X}) - \mathrm{vecp}(\nu \hspace{0.3mm} \mat{S}))^{\top} \big[B_d^{\top} (\hspace{-0.5mm} \sqrt{2\nu} \, \mat{S} \otimes \hspace{-1mm} \sqrt{2\nu} \, \mat{S}) B_d\big]^{-1} (\mathrm{vecp}(\mat{X}) - \mathrm{vecp}(\nu \hspace{0.3mm} \mat{S}))\right)}{\sqrt{(2\pi)^{d(d + 1)/2} \, |B_d^{\top} (\hspace{-0.5mm} \sqrt{2\nu} \, \mat{S} \otimes \hspace{-1mm} \sqrt{2\nu} \, \mat{S}) B_d|}}, \quad \mat{X}\in \mathcal{S}^{\hspace{0.3mm}d}.
    \end{equation}

    To give a bit of practical motivations for the SMN distribution \eqref{eq:sym.matrix.normal.density}, note that noise in the estimate of individual voxels of diffusion tensor magnetic resonance imaging (DT-MRI) data has been shown to be well modeled by the $\mathrm{SMN}_{3\times 3}$ distribution in \cite{Pajevic_Basser_1999,doi:10.1002nbm.783,doi:10.1016/s1090-7807(02)00178-7}. The SMN voxel distributions were combined into a tensor-variate normal distribution in \cite{doi:10.1109/TMI.2003.815059,doi:10.1137/16M1098693}, which could help to predict how the whole image (not just individual voxels) changes when shearing and dilation operations are applied in image wearing and registration problems, see \citet{doi:10.1109/42.963816}. In \cite{MR2485016}, maximum likelihood estimators and likelihood ratio tests are developed for the eigenvalues and eigenvectors of a form of the SMN distribution with an orthogonally invariant covariance structure, both in one-sample problems (for example, in image interpolation) and two-sample problems (when comparing images) and under a broad variety of assumptions.
    This work extended significantly previous results of \citet{MR131312}.
    In \cite{MR2485016}, it is also mentioned that the polarization pattern of cosmic microwave background (CMB) radiation measurements can be represented by $2\times 2$ positive definite matrices, see the primer by \citet{doi:10.1016/S1384-1076(97)00022-5}. In a very recent and interesting paper, \citet{doi:10.1093/mnras/stab368} presented evidence for the Gaussianity of the local extrema of CMB maps.
    We can also mention \cite{doi:10.1016/j.patcog.2018.02.025}, where finite mixtures of skewed SMN distributions were applied to an image recognition problem.

    In general, we know that the Gaussian distribution is an attractor for sums of i.i.d.\ random variables with finite variance, which makes many estimators in statistics asymptotically normal. Similarly, we expect the SMN distribution \eqref{eq:sym.matrix.normal.density} to be an attractor for sums of i.i.d.\ random symmetric matrices with finite variances, thus including many estimators such as sample covariance matrices and score statistics for symmetric matrix parameters.
    In particular, if a given statistic or estimator is a function of the components of a sample covariance matrix for i.i.d.\ observations coming from a multivariate Gaussian population, then we could study its large sample properties (such as its moments) using Theorem~\ref{thm:p.k.expansion} (for example, by turning a Wishart-moments estimation problem into a Gaussian-moments estimation problem).

    In Section~\ref{sec:application}, we use our asymptotic expansion (Theorem~\ref{thm:p.k.expansion}) to find the pointwise variance of a new density estimator on the space of positive definite matrices with a Wishart asymmetric kernel (Section~\ref{sec:Wishart.asymmetric.kernel}), and we derive an upper bound on the total variation between the probability measures on $\mathcal{S}^{\hspace{0.3mm}d}$ induced by \eqref{eq:Wishart.density} and \eqref{eq:sym.matrix.normal.density} (Section~\ref{sec:total.variation}).
    These are two examples of applications, but it is clear that there could be many others under the proper context.

    \begin{remark}[Notation]
        {\normalfont
        Throughout the paper, $a = \OO(b)$ means that $\limsup |a / b| < C$ as $\nu\to \infty$ (or as $b\to 0$ or as $n\to \infty$ in Section~\ref{sec:Wishart.asymmetric.kernel}, depending on the context), where $C > 0$ is a universal constant.
        Whenever $C$ might depend on some parameter, we add a subscript (for example, $a = \OO_d(b)$).
        Similarly, $a = \oo(b)$ means that $\lim |a / b| =~0$, and subscripts indicate which parameters the convergence rate can depend on.
        The notation $\mathrm{tr}(\cdot)$ will denote the trace operator for matrices and $|\cdot|$ their determinant.
        For a matrix $\mat{M}\in \R^{d\times d}$ that is diagonalizable, $\lambda_1(\mat{M}) \leq \dots \leq \lambda_d(\mat{M})$ will denote its eigenvalues, and we let $\bb{\lambda}(\mat{M}) \leqdef (\lambda_1(\mat{M}), \dots, \lambda_d(\mat{M}))^{\top}$.

        In Section~\ref{sec:Wishart.asymmetric.kernel} and the related proofs, the symbol $\mathscr{D}$ over an arrow `$\longrightarrow$' will denote the convergence in distribution (or law). We will also use the shorthand $[d] \leqdef \{1,\dots,d\}$ in several places.
        Finally, the bandwidth parameter $b = b(n)$ will always be implicitly a function of the number of observations, the only exceptions being in Theorem~\ref{thm:bias} and the related proof.
        }
    \end{remark}

\section{Main result}\label{sec:main.result}

    Below, we prove an asymptotic expansion for the ratio of the Wishart density to the symmetric matrix-variate normal (SMN) density with the same mean and covariances.
    This result is (much) stronger than the result found, for example, in \cite[Theorem~3.6.2]{MR1990662} or \cite[Theorem 2.5.1]{MR2640807}, which says that for a sequence of i.i.d.\ multivariate Gaussian observations $\bb{X}_1, \dots, \bb{X}_n\sim \mathcal{N}_d(\bb{\mu},\mat{S})$ with $\bb{\mu}\in \R^d$ and $\mat{S}\in \mathcal{S}_{++}^{\hspace{0.3mm}d}$, the scaled and recentered sample covariance matrix of $\bb{X}_1, \dots, \bb{X}_n$ converges in law to a SMN distribution, specifically,
    \begin{equation}\label{eq:scaled.covariance.convergence}
        n^{-1/2} \left[\sum_{i=1}^n (\bb{X}_i - \bb{\mu}) (\bb{X}_i - \bb{\mu})^{\top} \hspace{-0.5mm} - n \, \mat{S}\right] \stackrel{\mathscr{D}}{\longrightarrow} \mathrm{SMN}_{d\times d}(0_{d\times d}, 2 \hspace{0.3mm} B_d^{\top} (\mat{S} \otimes \mat{S}) \hspace{0.3mm} B_d), \quad n\to \infty.
    \end{equation}
    The result in Theorem~\ref{thm:p.k.expansion} is stronger than \eqref{eq:scaled.covariance.convergence} since it is well known that $\sum_{i=1}^n (\bb{X}_i - \bb{\mu}) (\bb{X}_i - \bb{\mu})^{\top}\sim \mathrm{Wishart}\hspace{0.3mm}(n, \mat{S})$ in this context.

    \begin{theorem}\label{thm:p.k.expansion}
        Let $\nu > d - 1$ and $\mat{S}\in \mathcal{S}_{++}^{\hspace{0.3mm}d}$ be given.
        Pick any $\eta\in (0,1)$ and let
        \begin{equation}\label{eq:thm:p.k.expansion.condition}
            B_{\nu,\mat{S}}(\eta) \leqdef \left\{\mat{X}\in \mathcal{S}_{++}^{\hspace{0.3mm}d} : \max_{1 \leq i \leq d} \left|\sqrt{2 / \nu} \, \lambda_i(\Delta_{\nu,\mat{S}})\right| \leq \eta \, \nu^{-1/3}\right\}
        \end{equation}
        denote the bulk of the Wishart distribution.
        Then, as $\nu\to \infty$ and uniformly for $\mat{X}\in B_{\nu,\mat{S}}(\eta)$, we have
        \begin{equation}\label{eq:LLT.order.2.log}
            \begin{aligned}
                \log \left(\frac{K_{\nu,\mat{S}}(\mat{X})}{g_{\nu,\mat{S}}(\mat{X})}\right)
                &= \nu^{-1/2} \cdot \Bigg\{\frac{\sqrt{2}}{3} \, \mathrm{tr}(\Delta_{\nu,\mat{S}}^3) - \frac{d + 1}{\sqrt{2}} \, \mathrm{tr}(\Delta_{\nu,\mat{S}})\Bigg\} + \nu^{-1} \cdot \left\{- \frac{1}{2} \, \mathrm{tr}(\Delta_{\nu,\mat{S}}^4) + \frac{d + 1}{2} \, \mathrm{tr}(\Delta_{\nu,\mat{S}}^2) - \left(\frac{d \, (2 d^{\hspace{0.3mm}2} + 3 d - 5)}{24} + \frac{d}{6}\right)\right\} \\
                &\quad+ \OO_{d,\eta}\left(\frac{1 + \max_{1 \leq i \leq d} |\lambda_i(\Delta_{\nu,\mat{S}})|^5}{\nu^{3/2}}\right).
            \end{aligned}
        \end{equation}
        Furthermore,
        \begin{equation}\label{eq:LLT.order.2}
            \begin{aligned}
                \frac{K_{\nu,\mat{S}}(\mat{X})}{g_{\nu,\mat{S}}(\mat{X})} = 1
                &+ \nu^{-1/2} \cdot \Bigg\{\frac{\sqrt{2}}{3} \, \mathrm{tr}(\Delta_{\nu,\mat{S}}^3) - \frac{d + 1}{\sqrt{2}} \, \mathrm{tr}(\Delta_{\nu,\mat{S}})\Bigg\} + \nu^{-1} \cdot
                    \left\{\frac{1}{9} \, \left(\mathrm{tr}(\Delta_{\nu,\mat{S}}^3)\right)^2 - \frac{d + 1}{3} \, \mathrm{tr}(\Delta_{\nu,\mat{S}}^3) \, \mathrm{tr}(\Delta_{\nu,\mat{S}}) + \frac{(d + 1)^2}{4} \, \left(\mathrm{tr}(\Delta_{\nu,\mat{S}})\right)^2 \right. \\
                &\quad \left.- \frac{1}{2} \, \mathrm{tr}(\Delta_{\nu,\mat{S}}^4) + \frac{d + 1}{2} \, \mathrm{tr}(\Delta_{\nu,\mat{S}}^2) - \left(\frac{d \, (2 d^{\hspace{0.3mm}2} + 3 d - 5)}{24} + \frac{d}{6}\right)\right\} + \OO_{d,\eta}\left(\frac{1 + \max_{1 \leq i \leq d} |\lambda_i(\Delta_{\nu,\mat{S}})|^9}{\nu^{3/2}}\right).
            \end{aligned}
        \end{equation}
    \end{theorem}

    As a direct consequence of Theorem~\ref{thm:p.k.expansion}, we obtain expansions for the ratio and log-ratio of the density function for a multivariate bijective mapping $\bb{h}(\cdot)$ applied to a Wishart random matrix to the density function of the same mapping applied to the corresponding SMN random matrix. In particular, the corollary below provides an asymptotic expansion for the density of a bijective mapping applied to a sample covariance matrix for i.i.d.\ observations coming from a multivariate Gaussian population.

    \begin{corollary}\label{cor:LLT}
        Let $\nu > d - 1$ and $\mat{S}\in \mathcal{S}_{++}^{\hspace{0.3mm}d}$ be given, and let $\mat{W}\sim \mathrm{Wishart}_d(\nu,\mat{S})$ and $\mat{N}\sim \mathrm{SMN}_{d\times d}(\nu \hspace{0.3mm} \mat{S}, B_d^{\top} (\hspace{-0.5mm} \sqrt{2 \nu} \, \mat{S} \otimes \hspace{-1mm} \sqrt{2 \nu} \, \mat{S}) B_d)$.
        Let $\bb{h}(\cdot)$ be a one-to-one mapping from an open subset $\mathcal{D}$ of $\mathcal{S}_{++}^{\hspace{0.3mm}d}$ onto a subset $\mathcal{R}$ of $\R^{d (d + 1) / 2}$. Assume further that $\bb{h}$ has continuous partial derivatives on $\mathcal{D}$ and its Jacobian determinant $\big|\frac{\rd}{\rd \, \mathrm{vecp}(\mat{X})} \bb{h}(\mat{X})\big|$ is non-zero for all $\mat{X}\in \mathcal{D}$.
        Define
        \begin{equation}
            \widetilde{\Delta}_{\nu,\mat{S}} = (\hspace{-0.5mm} \sqrt{2\nu} \, \mat{S})^{-1/2} (\bb{h}^{-1}(\bb{y}) - \nu \hspace{0.3mm} \mat{S}) (\hspace{-0.5mm} \sqrt{2\nu} \, \mat{S})^{-1/2}, \quad \bb{y}\in \mathcal{R},
        \end{equation}
        and denote by $f_{\bb{h}(\mat{W})}$ and $f_{\bb{h}(\mat{N})}$ the density functions of $\bb{h}(\mat{W})$ and $\bb{h}(\mat{N})$, respectively.
        Fix any $\eta\in (0,1)$, then we have, as $\nu\to \infty$, and uniformly for $\bb{y}\in \mathcal{R}$ such that $\bb{h}^{-1}(\bb{y})\in B_{\nu,\mat{S}}(\eta)$,
        \begin{equation}\label{eq:LLT.order.2.log.function.corollary}
            \begin{aligned}
                \log \left(\frac{f_{\bb{h}(\mat{W})}(\bb{y})}{f_{\bb{h}(\mat{N})}(\bb{y})}\right)
                &= \nu^{-1/2} \cdot \Bigg\{\frac{\sqrt{2}}{3} \, \mathrm{tr}(\widetilde{\Delta}_{\nu,\mat{S}}^3) - \frac{d + 1}{\sqrt{2}} \, \mathrm{tr}(\widetilde{\Delta}_{\nu,\mat{S}})\Bigg\} + \nu^{-1} \cdot \left\{- \frac{1}{2} \, \mathrm{tr}(\widetilde{\Delta}_{\nu,\mat{S}}^4) + \frac{d + 1}{2} \, \mathrm{tr}(\widetilde{\Delta}_{\nu,\mat{S}}^2) - \left(\frac{d \, (2 d^{\hspace{0.3mm}2} + 3 d - 5)}{24} + \frac{d}{6}\right)\right\} \\
                &\quad+ \OO_{d,\eta}\left(\frac{1 + \max_{1 \leq i \leq d} |\lambda_i(\widetilde{\Delta}_{\nu,\mat{S}})|^5}{\nu^{3/2}}\right),
            \end{aligned}
        \end{equation}
        and
        \begin{equation}\label{eq:LLT.order.2.function.corollary}
            \begin{aligned}
                \frac{f_{\bb{h}(\mat{W})}(\bb{y})}{f_{\bb{h}(\mat{N})}(\bb{y})} = 1
                &+ \nu^{-1/2} \cdot \Bigg\{\frac{\sqrt{2}}{3} \, \mathrm{tr}(\widetilde{\Delta}_{\nu,\mat{S}}^3) - \frac{d + 1}{\sqrt{2}} \, \mathrm{tr}(\widetilde{\Delta}_{\nu,\mat{S}})\Bigg\} + \nu^{-1} \cdot
                    \left\{\frac{1}{9} \, \left(\mathrm{tr}(\widetilde{\Delta}_{\nu,\mat{S}}^3)\right)^2 - \frac{d + 1}{3} \, \mathrm{tr}(\widetilde{\Delta}_{\nu,\mat{S}}^3) \, \mathrm{tr}(\widetilde{\Delta}_{\nu,\mat{S}}) + \frac{(d + 1)^2}{4} \, \left(\mathrm{tr}(\widetilde{\Delta}_{\nu,\mat{S}})\right)^2 \right. \\
                &\quad \left.- \frac{1}{2} \, \mathrm{tr}(\widetilde{\Delta}_{\nu,\mat{S}}^4) + \frac{d + 1}{2} \, \mathrm{tr}(\widetilde{\Delta}_{\nu,\mat{S}}^2) - \left(\frac{d \, (2 d^{\hspace{0.3mm}2} + 3 d - 5)}{24} + \frac{d}{6}\right)\right\} + \OO_{d,\eta}\left(\frac{1 + \max_{1 \leq i \leq d} |\lambda_i(\widetilde{\Delta}_{\nu,\mat{S}})|^9}{\nu^{3/2}}\right).
            \end{aligned}
        \end{equation}
    \end{corollary}

    Under the conditions of Corollary~\ref{cor:LLT}, we know from the multivariate delta method (see, e.g., \cite[Theorem~2.5.2]{MR2640807}) that the random vectors $\nu^{-1/2} \,(\bb{h}(\mat{W}) - \bb{h}(\nu \hspace{0.3mm} \mat{S}))$ and $\nu^{-1/2} \, (\bb{h}(\mat{N}) - \bb{h}(\nu \hspace{0.3mm} \mat{S}))$ both converge in distribution, as $\nu\to \infty$, to
    \begin{equation*}
        \bb{Y}\sim \mathcal{N}_{d(d + 1) / 2}\left(\bb{0}, \left(\left.\frac{\rd}{\rd \, \mathrm{vecp}(\mat{X})} \bb{h}(\mat{X})\right|_{\mat{X} = \nu \hspace{0.3mm} \mat{S}}\right) 2 \hspace{0.3mm} B_d^{\top} (\mat{S} \otimes \mat{S}) \hspace{0.3mm} B_d \left(\left.\frac{\rd}{\rd \, \mathrm{vecp}(\mat{X})} \bb{h}(\mat{X})\right|_{\mat{X} = \nu \hspace{0.3mm} \mat{S}}\right)^{\top}\right),
    \end{equation*}
    where $\bb{h}(\mat{X}) = (h_1(\mat{X}), \dots, h_{d(d + 1)/2}(\mat{X}))^{\top}$ and
    \begin{equation*}
        \left.\frac{\rd}{\rd \, \mathrm{vecp}(\mat{X})} \bb{h}(\mat{X})\right|_{\mat{X} = \nu \hspace{0.3mm} \mat{S}} = \left[\left.\frac{\rd}{\rd \, \mathrm{vecp}(\mat{X})} h_1(\mat{X})\right|_{\mat{X} = \nu \hspace{0.3mm} \mat{S}} ~~ \left.\frac{\rd}{\rd \, \mathrm{vecp}(\mat{X})} h_2(\mat{X})\right|_{\mat{X} = \nu \hspace{0.3mm} \mat{S}} ~~\dots~~ \left.\frac{\rd}{\rd \, \mathrm{vecp}(\mat{X})} h_{d(d + 1)/2}(\mat{X})\right|_{\mat{X} = \nu \hspace{0.3mm} \mat{S}}\right]^{\top}.
    \end{equation*}
    Therefore, it would have been neat to extend Corollary~\ref{cor:LLT} by expanding the log-ratio $\log (f_{\nu^{-1/2} \, (\bb{h}(\mat{W}) - \bb{h}(\nu \hspace{0.3mm} \mat{S}))}(\bb{y}) / f_{\bb{Y}}(\bb{y}))$.
    However, this would most likely require an expansion for the log-ratio $\log (f_{\nu^{-1/2} \, (\bb{h}(\mat{N}) - \bb{h}(\nu \hspace{0.3mm} \mat{S}))}(\bb{y}) / f_{\bb{Y}}(\bb{y}))$, and it is unclear which restrictions we should impose on $\bb{h}$ to progress in that direction.
    This question is left open for future research.

    Below, we provide numerical evidence (displayed graphically) for the validity of the expansion in Theorem~\ref{thm:p.k.expansion} when $d = 2$.
    We compare three levels of approximation for various choices of $\mat{S}$.
    For any given $\mat{S}\in \mathcal{S}_{++}^{\hspace{0.3mm}d}$, define
    \begin{align}
        E_0
        &\leqdef \sup_{\mat{X}\in B_{\nu,\mat{S}}(2^{-1/2} \nu^{-1/6})} \left|\log \left(\frac{K_{\nu,\mat{S}}(\mat{X})}{g_{\nu,\mat{S}}(\mat{X})}\right)\right|, \label{eq:E.0} \\[0.5mm]
        E_1
        &\leqdef \sup_{\mat{X}\in B_{\nu,\mat{S}}(2^{-1/2} \nu^{-1/6})} \left|\log \left(\frac{K_{\nu,\mat{S}}(\mat{X})}{g_{\nu,\mat{S}}(\mat{X})}\right) - \nu^{-1/2} \cdot \Bigg\{\frac{\sqrt{2}}{3} \, \mathrm{tr}(\Delta_{\nu,\mat{S}}^3) - \frac{d + 1}{\sqrt{2}} \, \mathrm{tr}(\Delta_{\nu,\mat{S}})\Bigg\}\right|, \label{eq:E.1} \\[1mm]
        E_2
        &\leqdef \sup_{\mat{X}\in B_{\nu,\mat{S}}(2^{-1/2} \nu^{-1/6})} \left|\log \left(\frac{K_{\nu,\mat{S}}(\mat{X})}{g_{\nu,\mat{S}}(\mat{X})}\right) - \nu^{-1/2} \cdot \Bigg\{\frac{\sqrt{2}}{3} \, \mathrm{tr}(\Delta_{\nu,\mat{S}}^3) - \frac{d + 1}{\sqrt{2}} \, \mathrm{tr}(\Delta_{\nu,\mat{S}})\Bigg\}\right. \notag \\
        &\quad\left.\hspace{25mm}- \nu^{-1} \cdot \left\{- \frac{1}{2} \, \mathrm{tr}(\Delta_{\nu,\mat{S}}^4) + \frac{d + 1}{2} \, \mathrm{tr}(\Delta_{\nu,\mat{S}}^2) - \left(\frac{d \, (2 d^{\hspace{0.3mm}2} + 3 d - 5)}{24} + \frac{d}{6}\right)\right\}\right|. \label{eq:E.2}
    \end{align}
    In order to avoid numerical errors in part due to the gamma functions in $K_{\nu,\mat{S}}(\mat{X})$, we have to work a bit to get an expression for $\log \left(K_{\nu,\mat{S}}(\mat{X}) / g_{\nu,\mat{S}}(\mat{X})\right)$ which is numerically more stable.
    By taking the expression for $K_{\nu,\mat{S}}(\mat{X})$ in \eqref{eq:Wishart.density.rewrite} and dividing by the expression for $g_{\nu,\mat{S}}(\mat{X})$ on the right-hand side of \eqref{eq:sym.matrix.normal.density}, we get
    \begin{equation}
        \frac{K_{\nu,\mat{S}}(\mat{X})}{g_{\nu,\mat{S}}(\mat{X})} = |\mathrm{I}_d + \sqrt{2 / \nu} \, \Delta_{\nu,\mat{S}}|^{\nu/2 - (d + 1)/2} \cdot \frac{\exp\left(-\sqrt{\frac{\nu}{2}} \mathrm{tr}(\Delta_{\nu,\mat{S}}) + \frac{1}{2} \mathrm{tr}(\Delta_{\nu,\mat{S}}^2)\right)}{\prod_{i=1}^d \frac{(\nu - (i + 1))^{(\nu - i)/2}}{e^{-(i + 1)/2} \, \nu^{(\nu - i)/2}} \cdot \prod_{i=1}^d \frac{\Gamma(\frac{1}{2} (\nu - (i + 1)) + 1)}{\sqrt{2\pi} e^{-\frac{1}{2} (\nu - (i + 1))} [\frac{1}{2} (\nu - (i + 1))]^{(\nu - i)/2}}},
    \end{equation}
    so that
    \begin{equation}
        \begin{aligned}
            \log\left(\frac{K_{\nu,\mat{S}}(\mat{X})}{g_{\nu,\mat{S}}(\mat{X})}\right)
            &= \frac{\nu - (d + 1)}{2} \sum_{i=1}^d \log\left(1 + \sqrt{\frac{2}{\nu}} \lambda_i(\Delta_{\nu,\mat{S}})\right) - \sqrt{\frac{\nu}{2}} \mathrm{tr}(\Delta_{\nu,\mat{S}}) + \frac{1}{2} \mathrm{tr}(\Delta_{\nu,\mat{S}}^2) - \sum_{i=1}^d \left[\left(\frac{\nu - i}{2}\right) \log \left(1 - \frac{i + 1}{\nu}\right) + \frac{i + 1}{2}\right] \\
            &\quad- \sum_{i=1}^d \left[\log \Gamma\left(\frac{1}{2} (\nu - (i + 1)) + 1\right) - \frac{1}{2} \log(2\pi) + \frac{1}{2} (\nu - (i + 1)) - \frac{\nu - i}{2} \log\left(\frac{1}{2} (\nu - (i + 1))\right)\right].
        \end{aligned}
    \end{equation}
    In \texttt{R}, we used this last equation to evaluate the log-ratios inside $E_0$, $E_1$ and $E_2$.

    \newpage
    Note that $\mat{X}\in B_{\nu,\mat{S}}(2^{-1/2} \nu^{-1/6})$ implies $|\mathrm{tr}(\Delta_{\nu,\mat{S}}^k)| \leq d \, 2^{-k}$ for all $k\in \N$, so we expect from Theorem~\ref{thm:p.k.expansion} that the maximum errors above ($E_0$, $E_1$ and $E_2$) will have the asymptotic behavior
    \begin{align}
        E_i = \OO_d(\nu^{-(1 + i)/2}), \quad \text{for all } i\in \{0,1,2\},
    \end{align}
    or equivalently,
    \begin{align}\label{eq:liminf.exponent.bound}
        \liminf_{\nu\to \infty} \frac{\log E_i}{\log (\nu^{-1})} \geq \frac{1 + i}{2}, \quad \text{for all } i\in \{0,1,2\}.
    \end{align}
    The property \eqref{eq:liminf.exponent.bound} is verified in Fig.~\ref{fig:error.exponents.plots} below, for various choices of $\mat{S}$.
    Similarly, the corresponding the log-log plots of the errors as a function of $\nu$ are displayed in Fig.~\ref{fig:loglog.errors.plots}.
    The simulations are limited to the range $5 \leq \nu \leq 205$.
    The \texttt{R} code that generated Fig.~\ref{fig:loglog.errors.plots} and Fig.~\ref{fig:error.exponents.plots} can be found in \ref{sec:R.code}.

    \vspace{6mm}
    \begin{figure}[ht]
        \captionsetup[subfigure]{labelformat=empty}
        \vspace{-0.5cm}
        \centering
        \begin{subfigure}[b]{0.22\textwidth}
            \centering
            \includegraphics[width=\textwidth, height=0.85\textwidth]{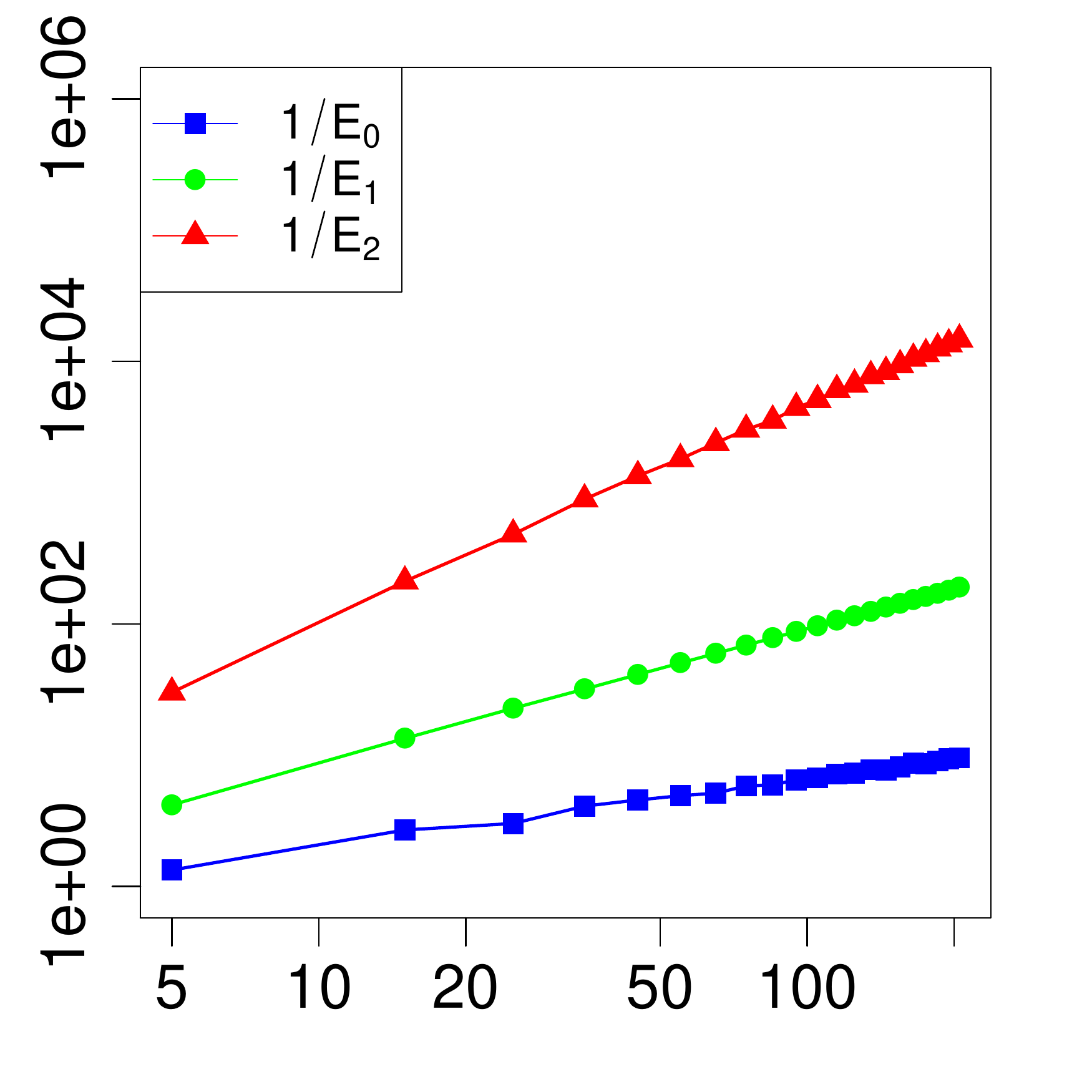}
            \vspace{-0.8cm}
            \caption{$\mat{S} = \begin{pmatrix} 2 & 1 \\ 1 & 2\end{pmatrix}$}
        \end{subfigure}
        \quad
        \begin{subfigure}[b]{0.22\textwidth}
            \centering
            \includegraphics[width=\textwidth, height=0.85\textwidth]{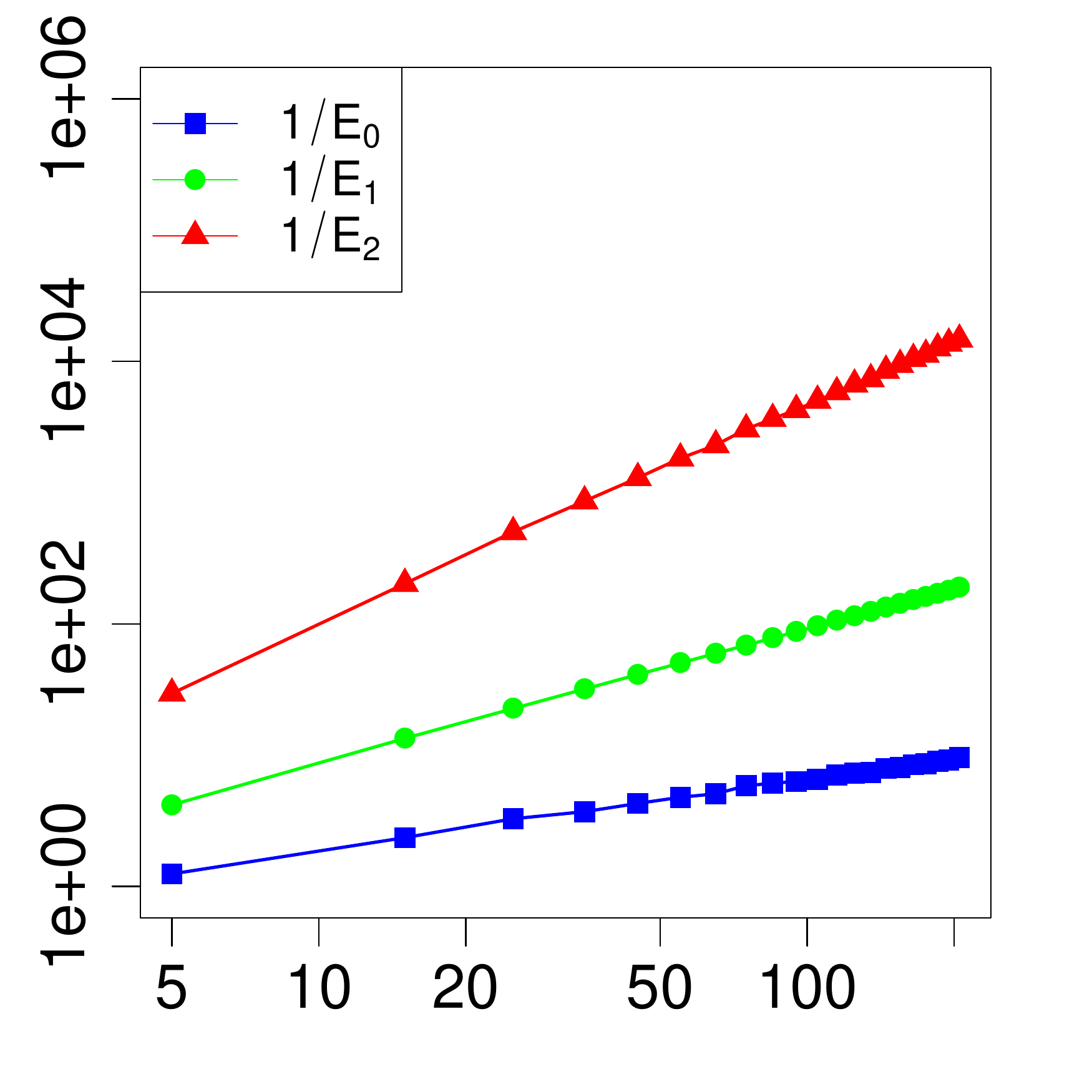}
            \vspace{-0.8cm}
            \caption{$\mat{S} = \begin{pmatrix} 2 & 1 \\ 1 & 3\end{pmatrix}$}
        \end{subfigure}
        \quad
        \begin{subfigure}[b]{0.22\textwidth}
            \centering
            \includegraphics[width=\textwidth, height=0.85\textwidth]{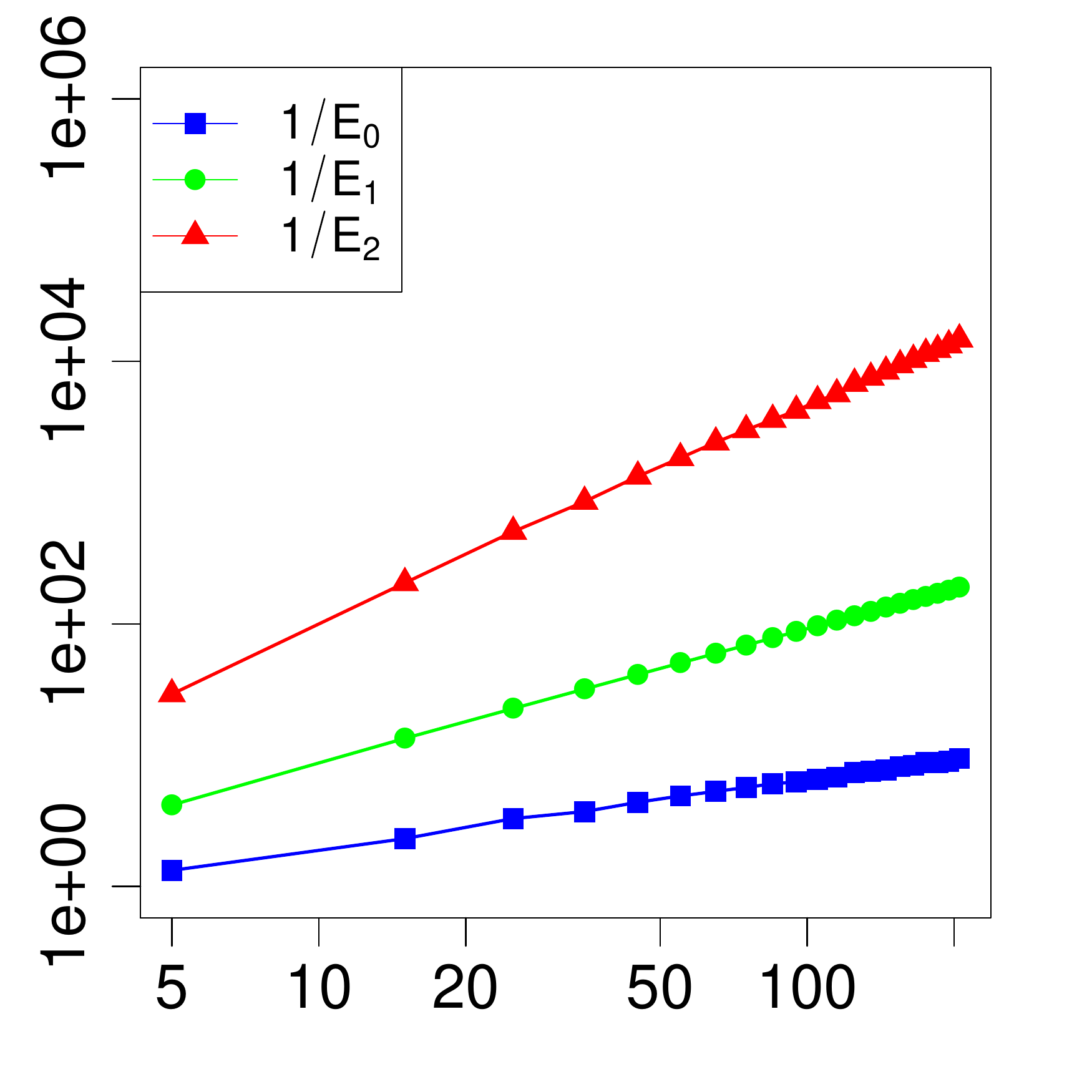}
            \vspace{-0.8cm}
            \caption{$\mat{S} = \begin{pmatrix} 2 & 1 \\ 1 & 4\end{pmatrix}$}
        \end{subfigure}
        \quad
        \begin{subfigure}[b]{0.22\textwidth}
            \centering
            \includegraphics[width=\textwidth, height=0.85\textwidth]{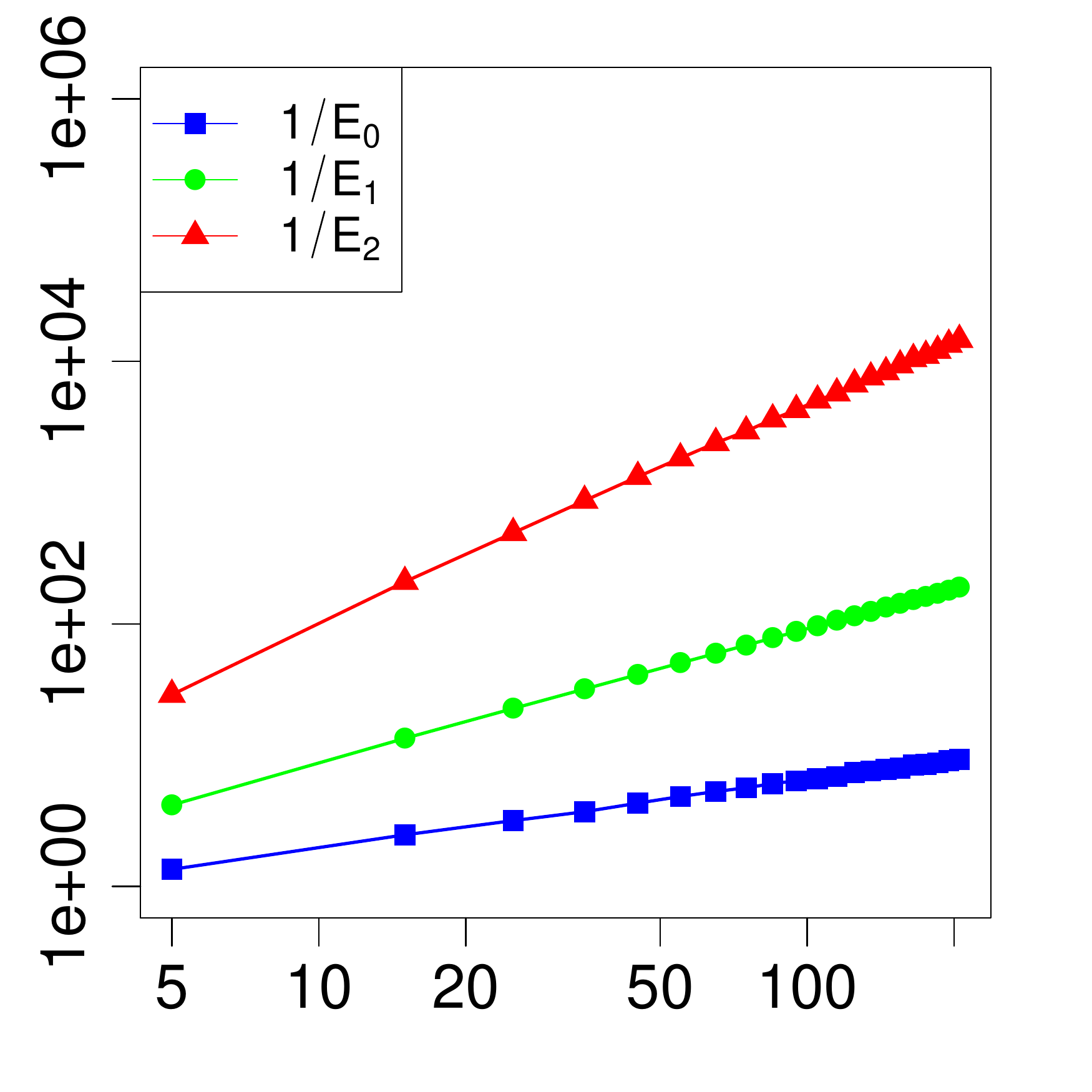}
            \vspace{-0.8cm}
            \caption{$\mat{S} = \begin{pmatrix} 2 & 1 \\ 1 & 5\end{pmatrix}$}
        \end{subfigure}
        \begin{subfigure}[b]{0.22\textwidth}
            \centering
            \includegraphics[width=\textwidth, height=0.85\textwidth]{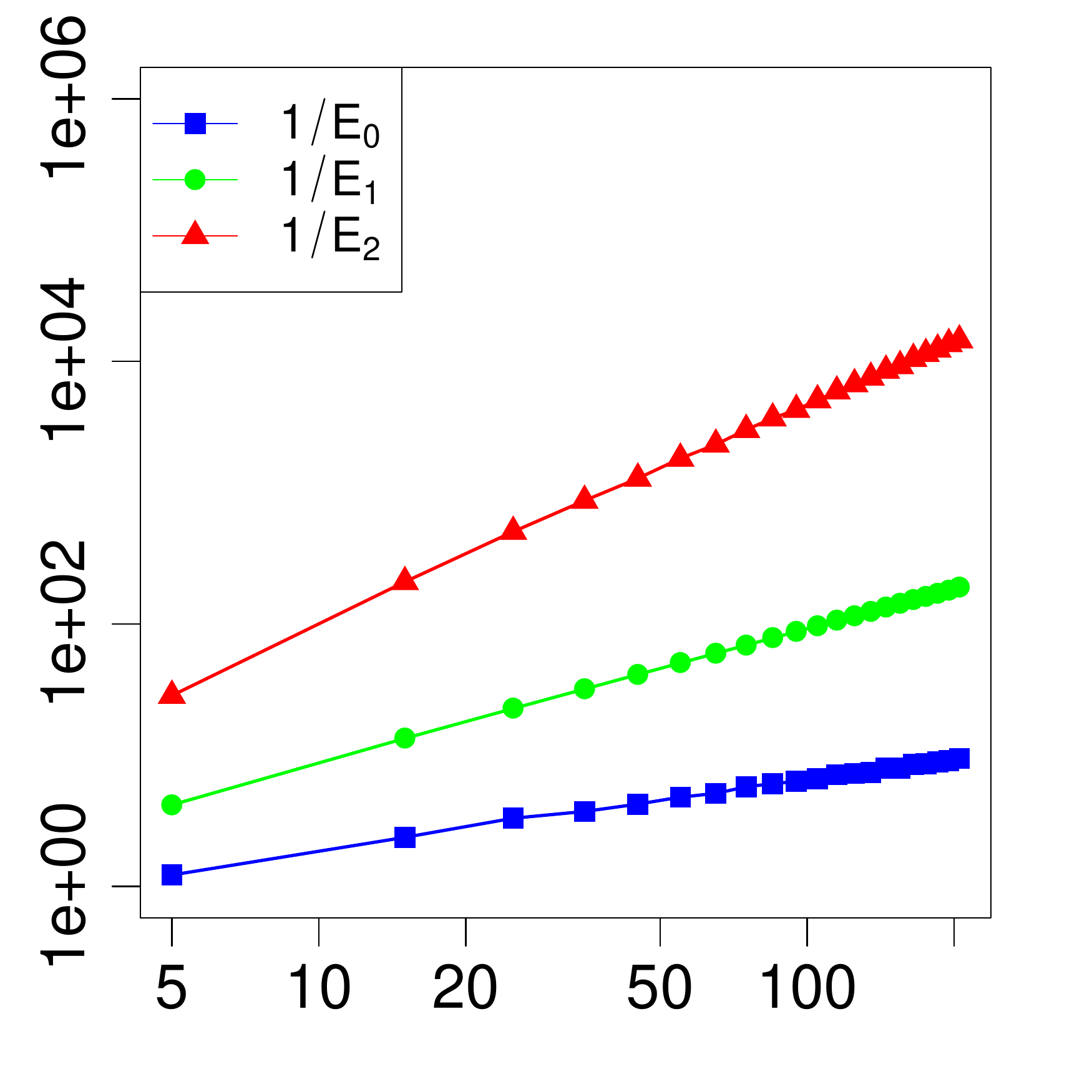}
            \vspace{-0.8cm}
            \caption{$\mat{S} = \begin{pmatrix} 3 & 1 \\ 1 & 2\end{pmatrix}$}
        \end{subfigure}
        \quad
        \begin{subfigure}[b]{0.22\textwidth}
            \centering
            \includegraphics[width=\textwidth, height=0.85\textwidth]{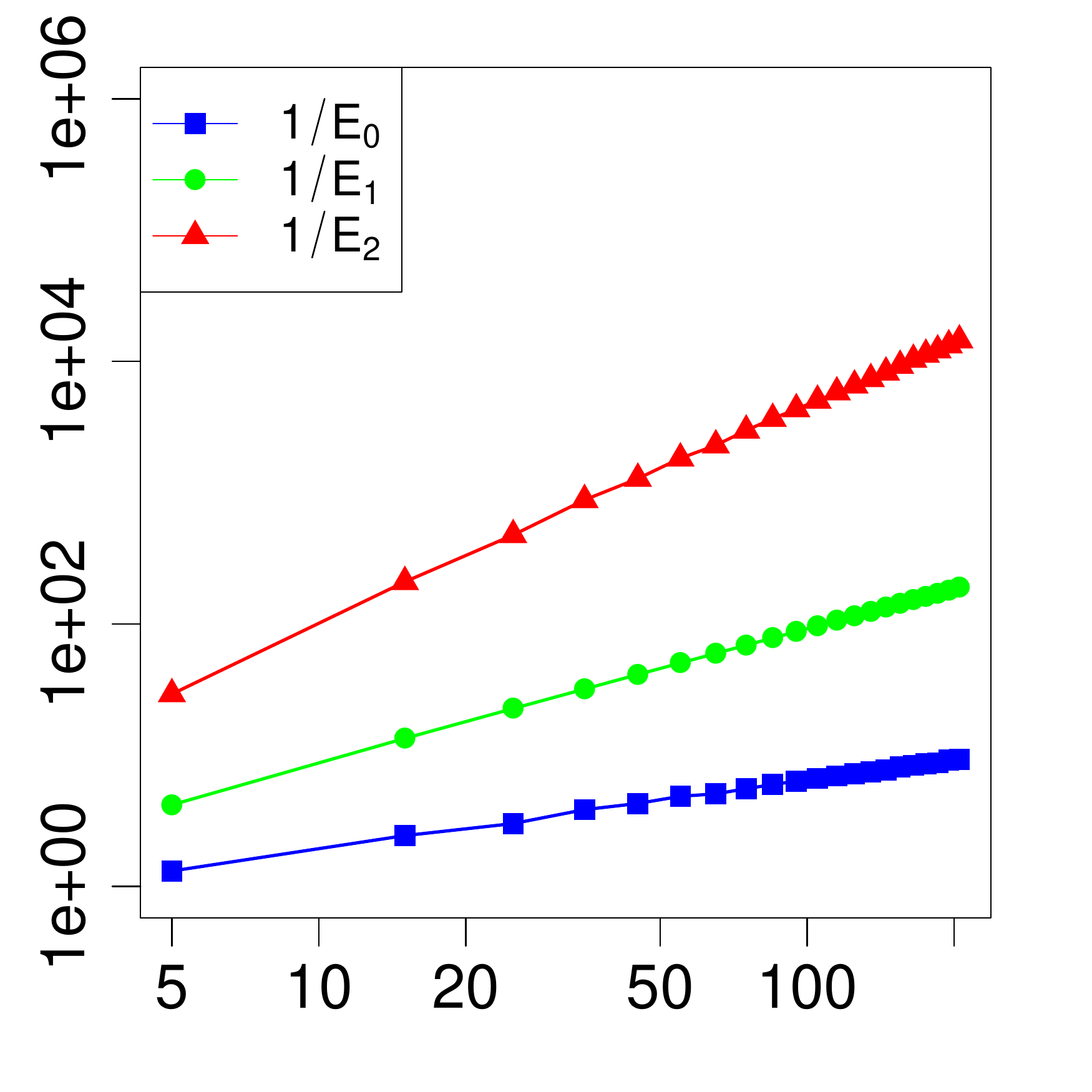}
            \vspace{-0.8cm}
            \caption{$\mat{S} = \begin{pmatrix} 3 & 1 \\ 1 & 3\end{pmatrix}$}
        \end{subfigure}
        \quad
        \begin{subfigure}[b]{0.22\textwidth}
            \centering
            \includegraphics[width=\textwidth, height=0.85\textwidth]{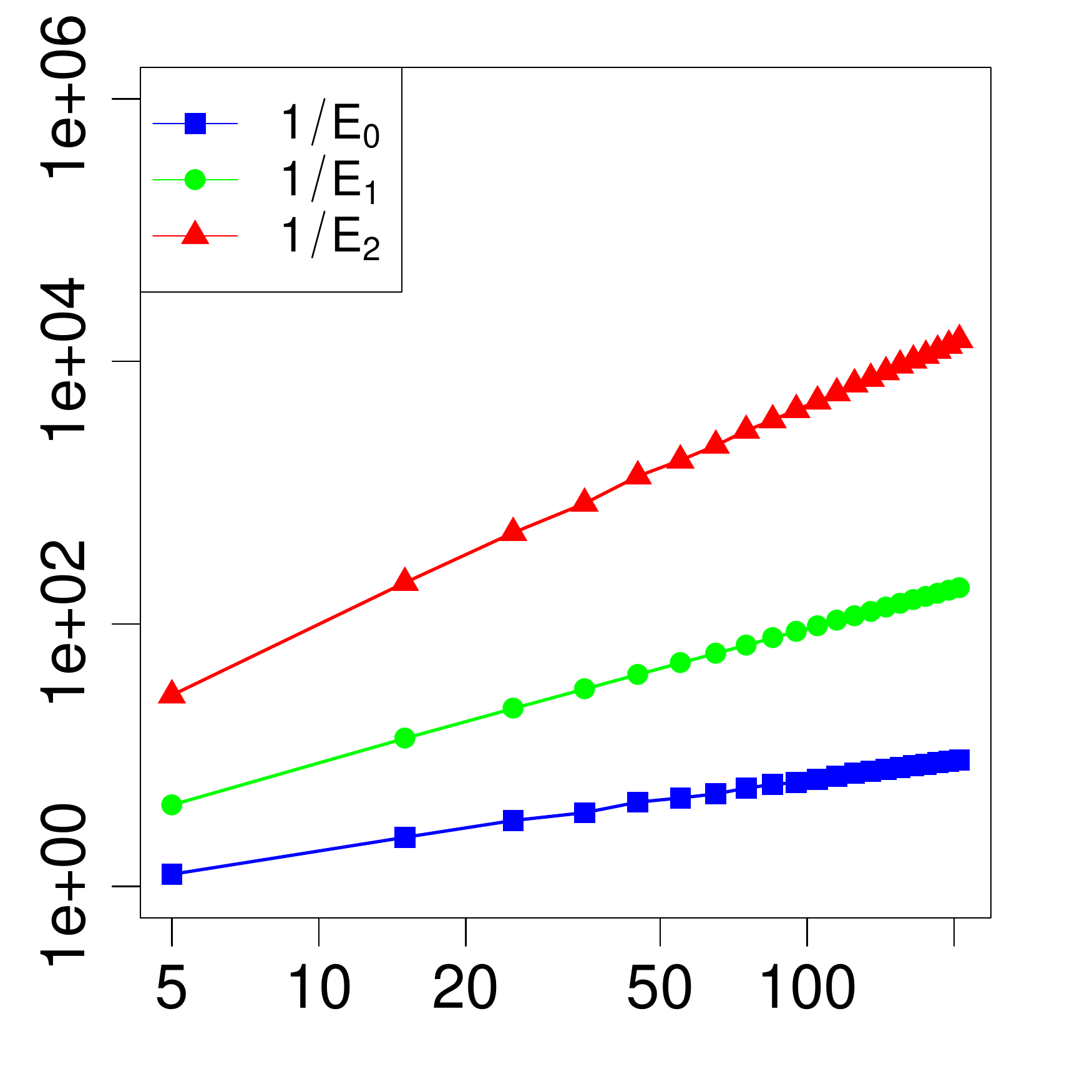}
            \vspace{-0.8cm}
            \caption{$\mat{S} = \begin{pmatrix} 3 & 1 \\ 1 & 4\end{pmatrix}$}
        \end{subfigure}
        \quad
        \begin{subfigure}[b]{0.22\textwidth}
            \centering
            \includegraphics[width=\textwidth, height=0.85\textwidth]{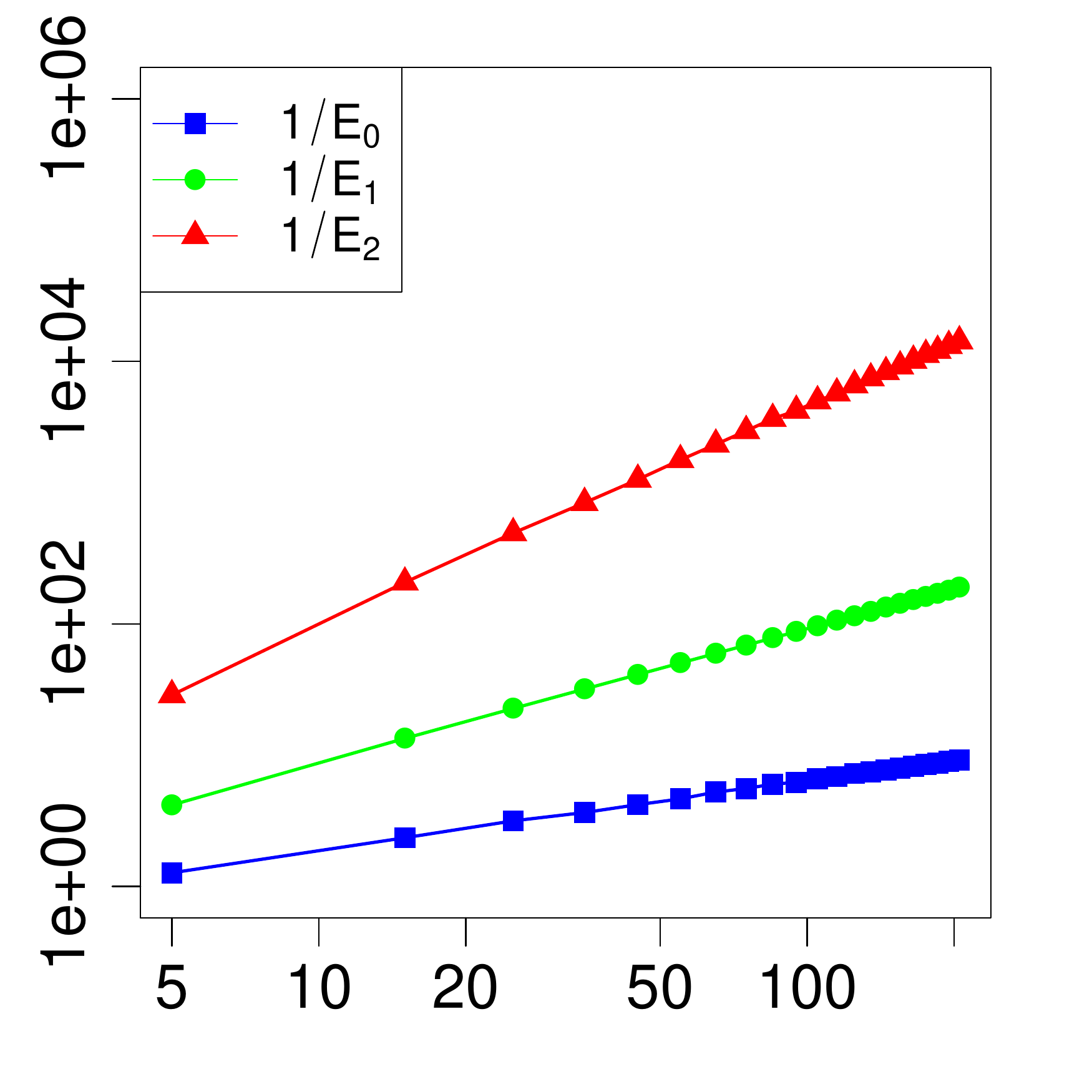}
            \vspace{-0.8cm}
            \caption{$\mat{S} = \begin{pmatrix} 3 & 1 \\ 1 & 5\end{pmatrix}$}
        \end{subfigure}
        \begin{subfigure}[b]{0.22\textwidth}
            \centering
            \includegraphics[width=\textwidth, height=0.85\textwidth]{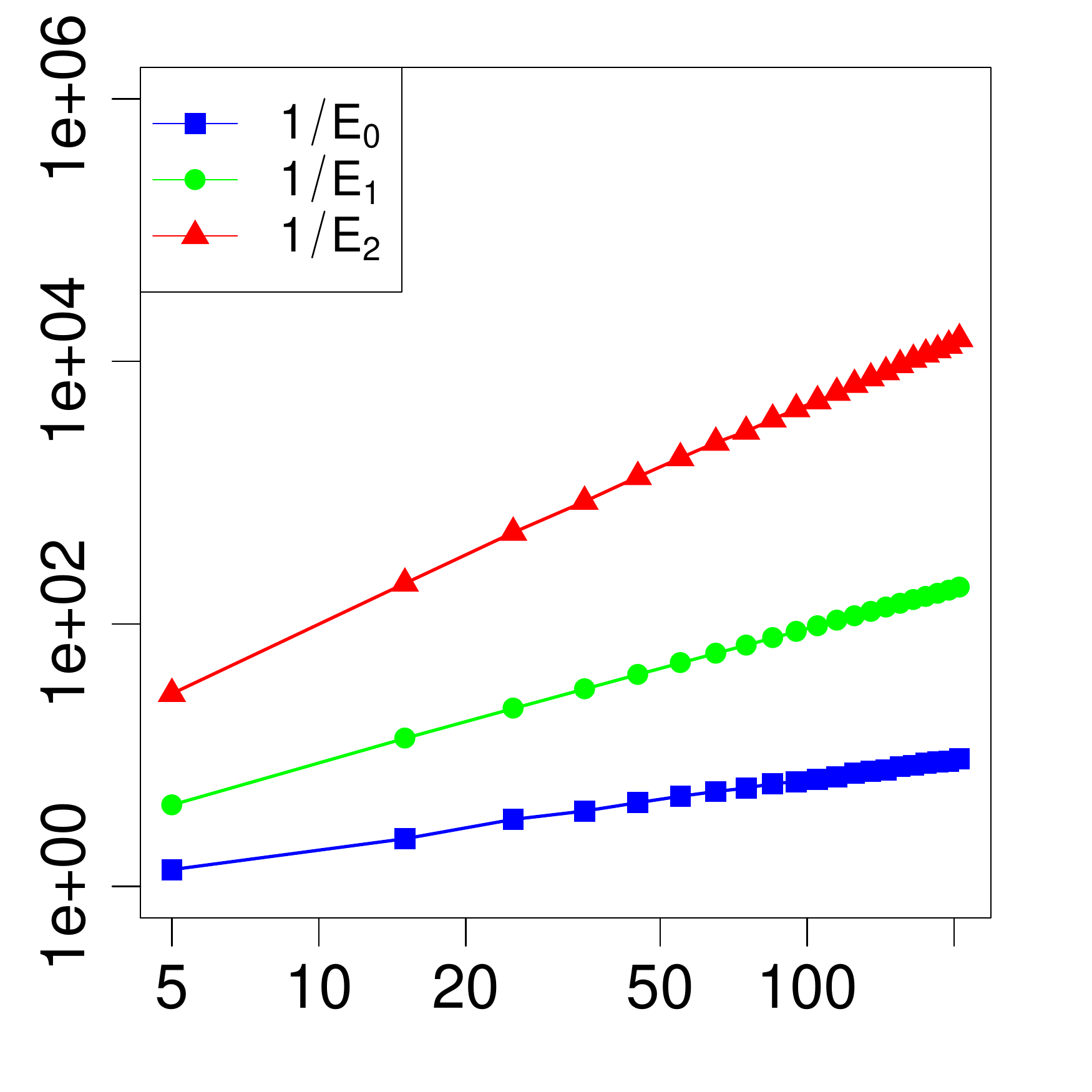}
            \vspace{-0.8cm}
            \caption{$\mat{S} = \begin{pmatrix} 4 & 1 \\ 1 & 2\end{pmatrix}$}
        \end{subfigure}
        \quad
        \begin{subfigure}[b]{0.22\textwidth}
            \centering
            \includegraphics[width=\textwidth, height=0.85\textwidth]{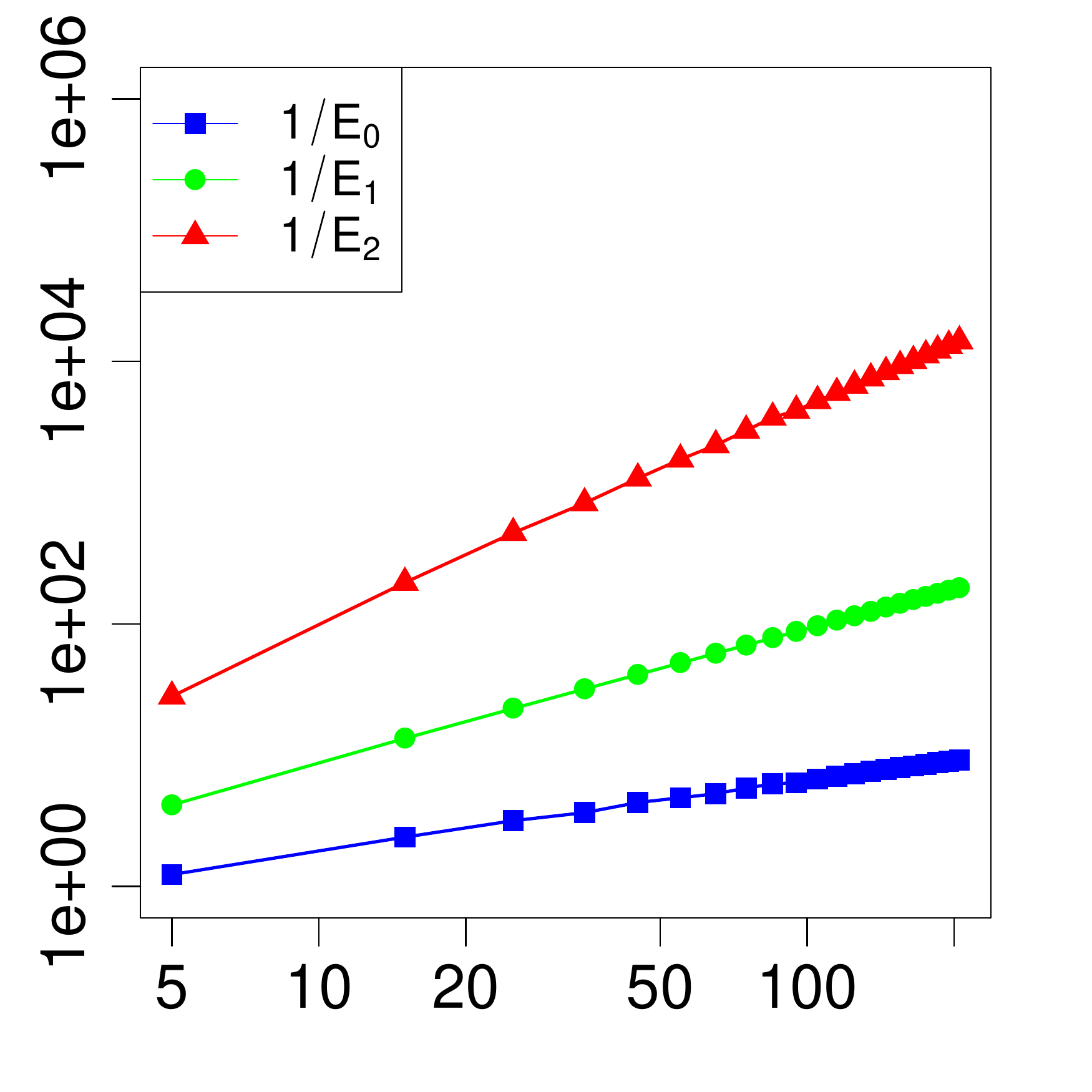}
            \vspace{-0.8cm}
            \caption{$\mat{S} = \begin{pmatrix} 4 & 1 \\ 1 & 3\end{pmatrix}$}
        \end{subfigure}
        \quad
        \begin{subfigure}[b]{0.22\textwidth}
            \centering
            \includegraphics[width=\textwidth, height=0.85\textwidth]{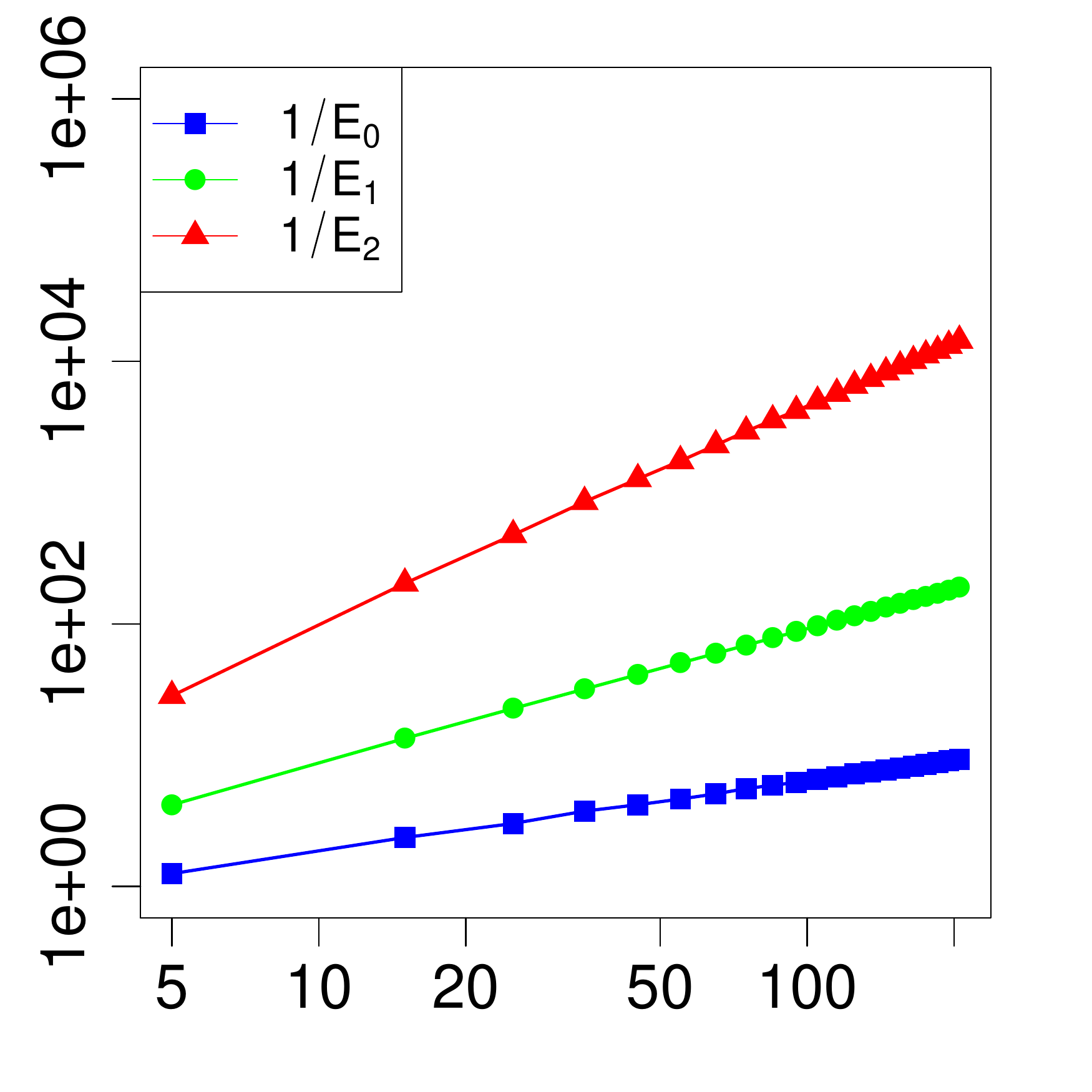}
            \vspace{-0.8cm}
            \caption{$\mat{S} = \begin{pmatrix} 4 & 1 \\ 1 & 4\end{pmatrix}$}
        \end{subfigure}
        \quad
        \begin{subfigure}[b]{0.22\textwidth}
            \centering
            \includegraphics[width=\textwidth, height=0.85\textwidth]{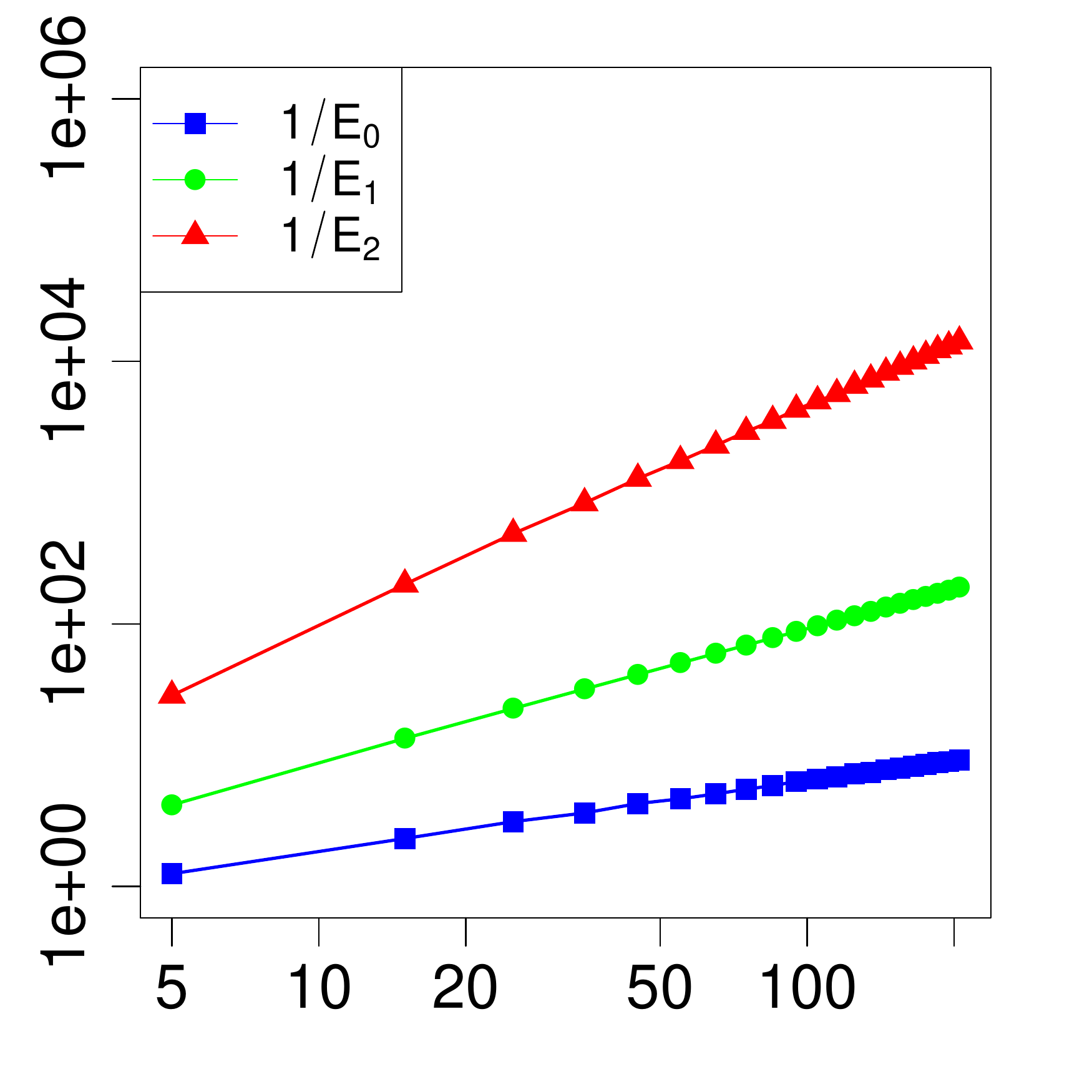}
            \vspace{-0.8cm}
            \caption{$\mat{S} = \begin{pmatrix} 4 & 1 \\ 1 & 5\end{pmatrix}$}
        \end{subfigure}
        \begin{subfigure}[b]{0.22\textwidth}
            \centering
            \includegraphics[width=\textwidth, height=0.85\textwidth]{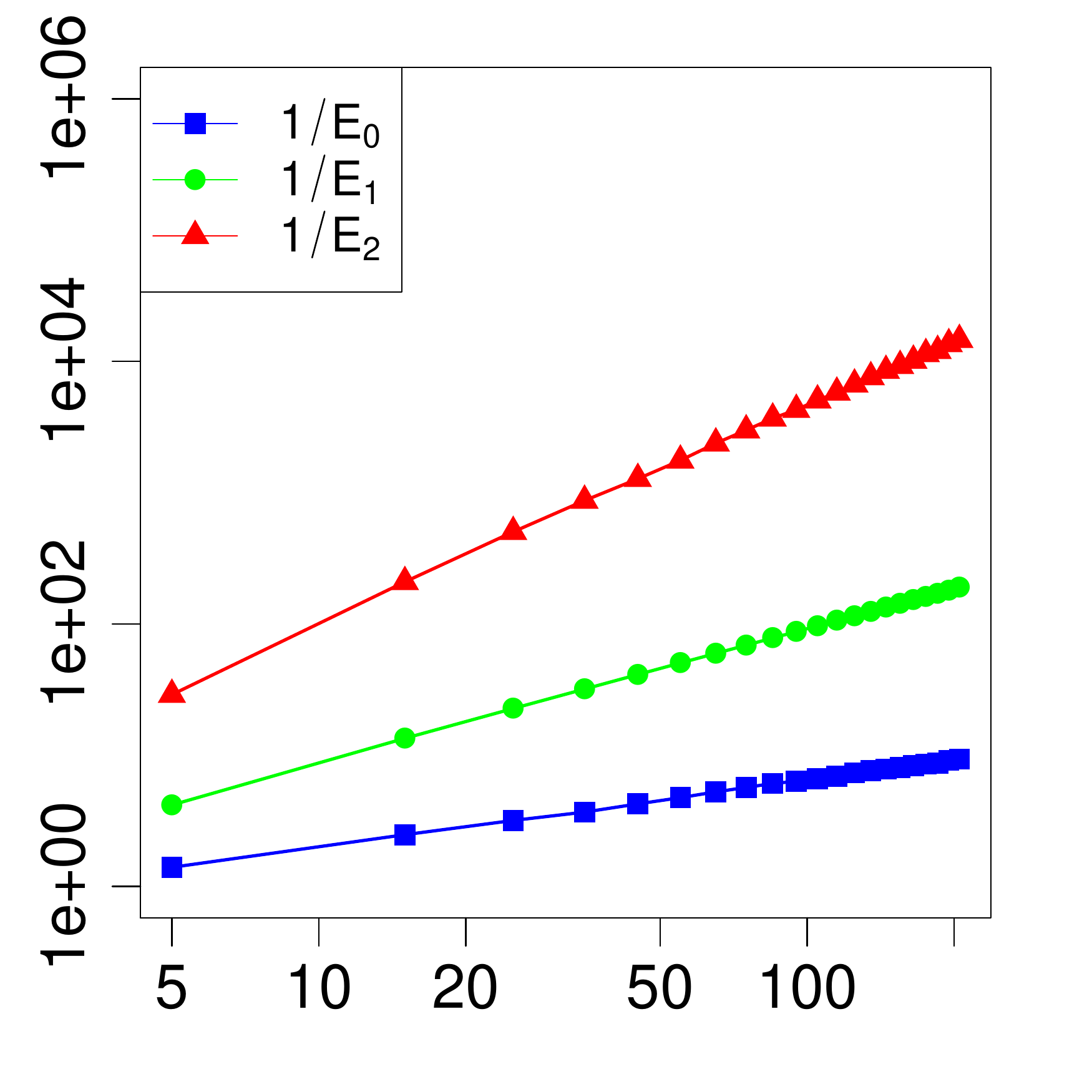}
            \vspace{-0.8cm}
            \caption{$\mat{S} = \begin{pmatrix} 5 & 1 \\ 1 & 2\end{pmatrix}$}
        \end{subfigure}
        \quad
        \begin{subfigure}[b]{0.22\textwidth}
            \centering
            \includegraphics[width=\textwidth, height=0.85\textwidth]{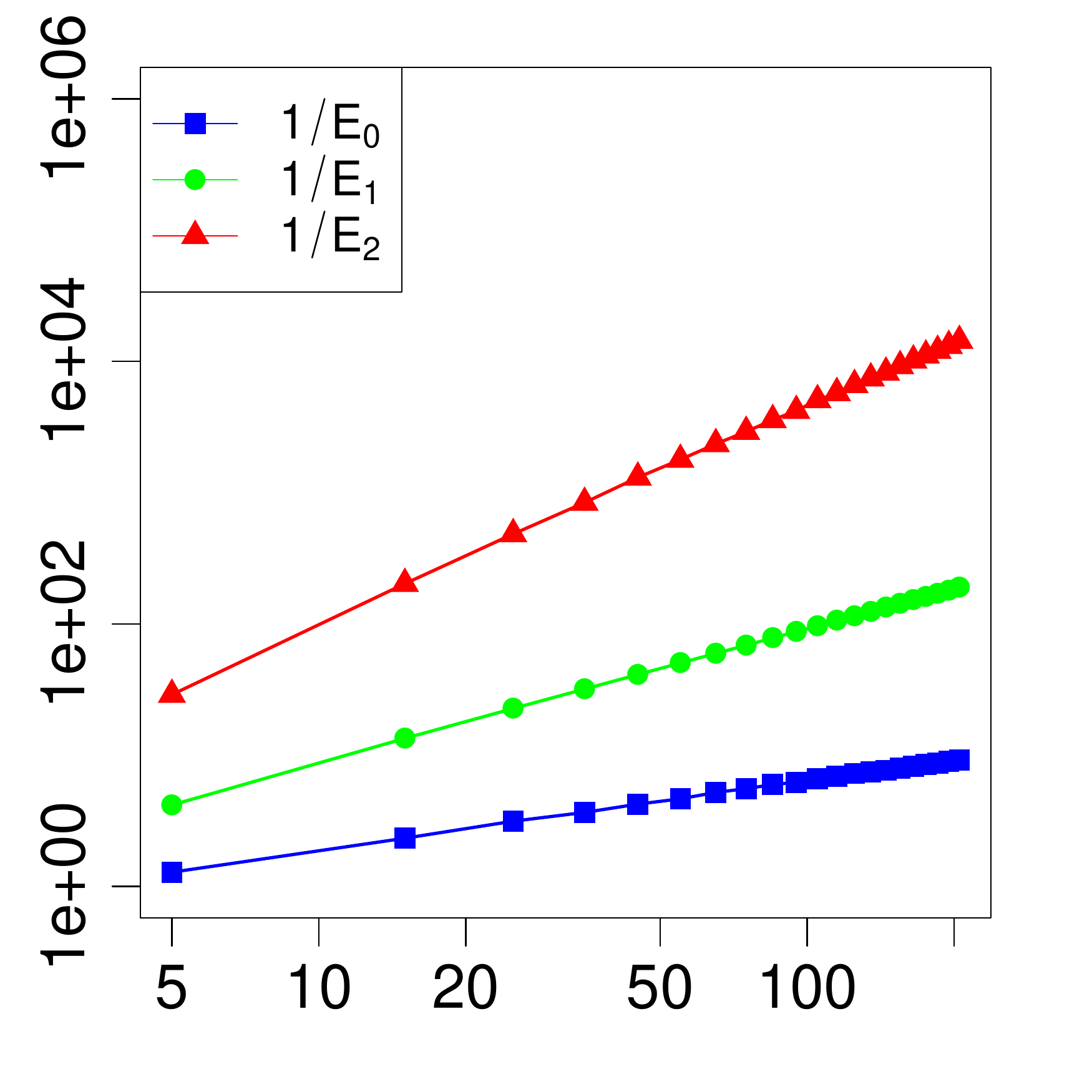}
            \vspace{-0.8cm}
            \caption{$\mat{S} = \begin{pmatrix} 5 & 1 \\ 1 & 3\end{pmatrix}$}
        \end{subfigure}
        \quad
        \begin{subfigure}[b]{0.22\textwidth}
            \centering
            \includegraphics[width=\textwidth, height=0.85\textwidth]{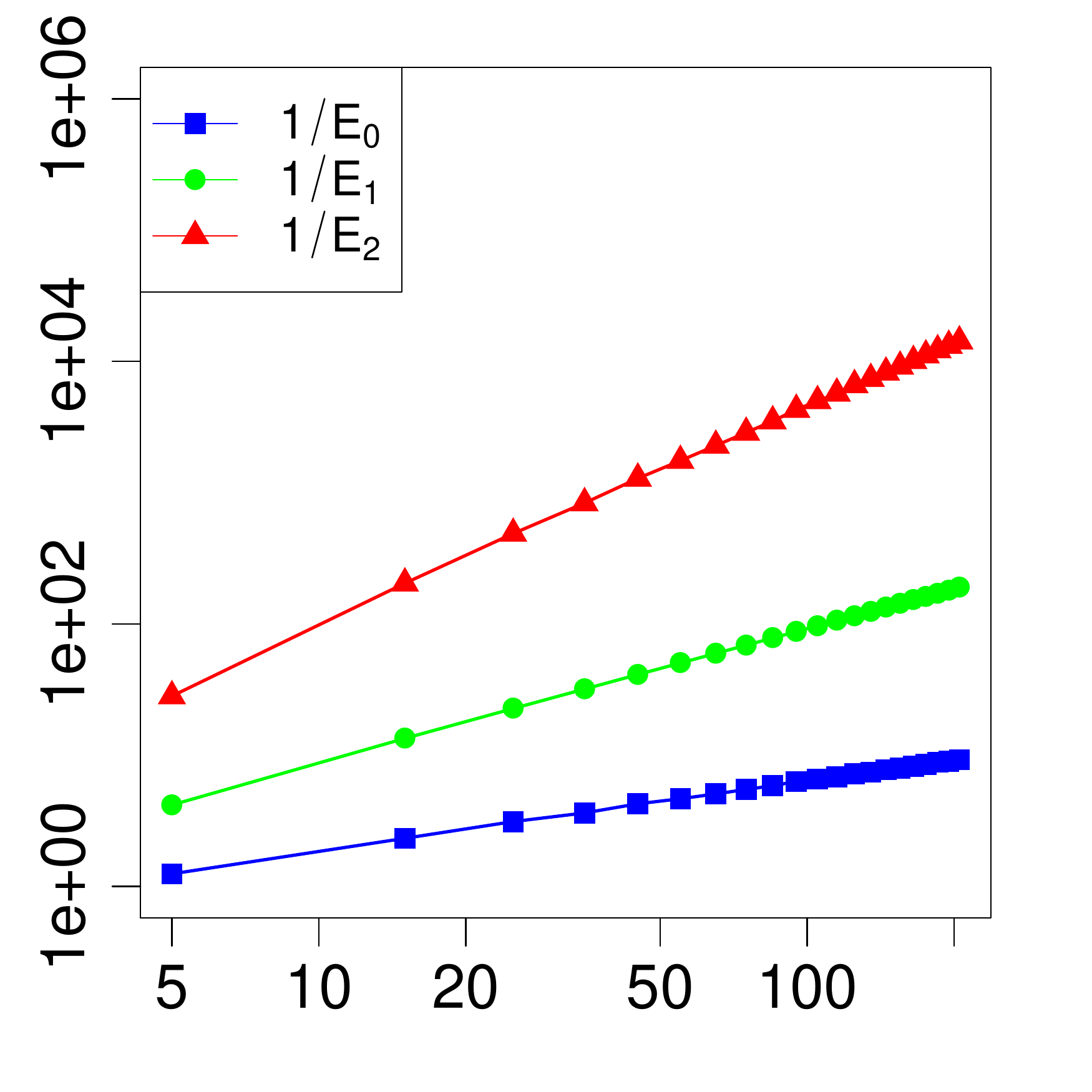}
            \vspace{-0.8cm}
            \caption{$\mat{S} = \begin{pmatrix} 5 & 1 \\ 1 & 4\end{pmatrix}$}
        \end{subfigure}
        \quad
        \begin{subfigure}[b]{0.22\textwidth}
            \centering
            \includegraphics[width=\textwidth, height=0.85\textwidth]{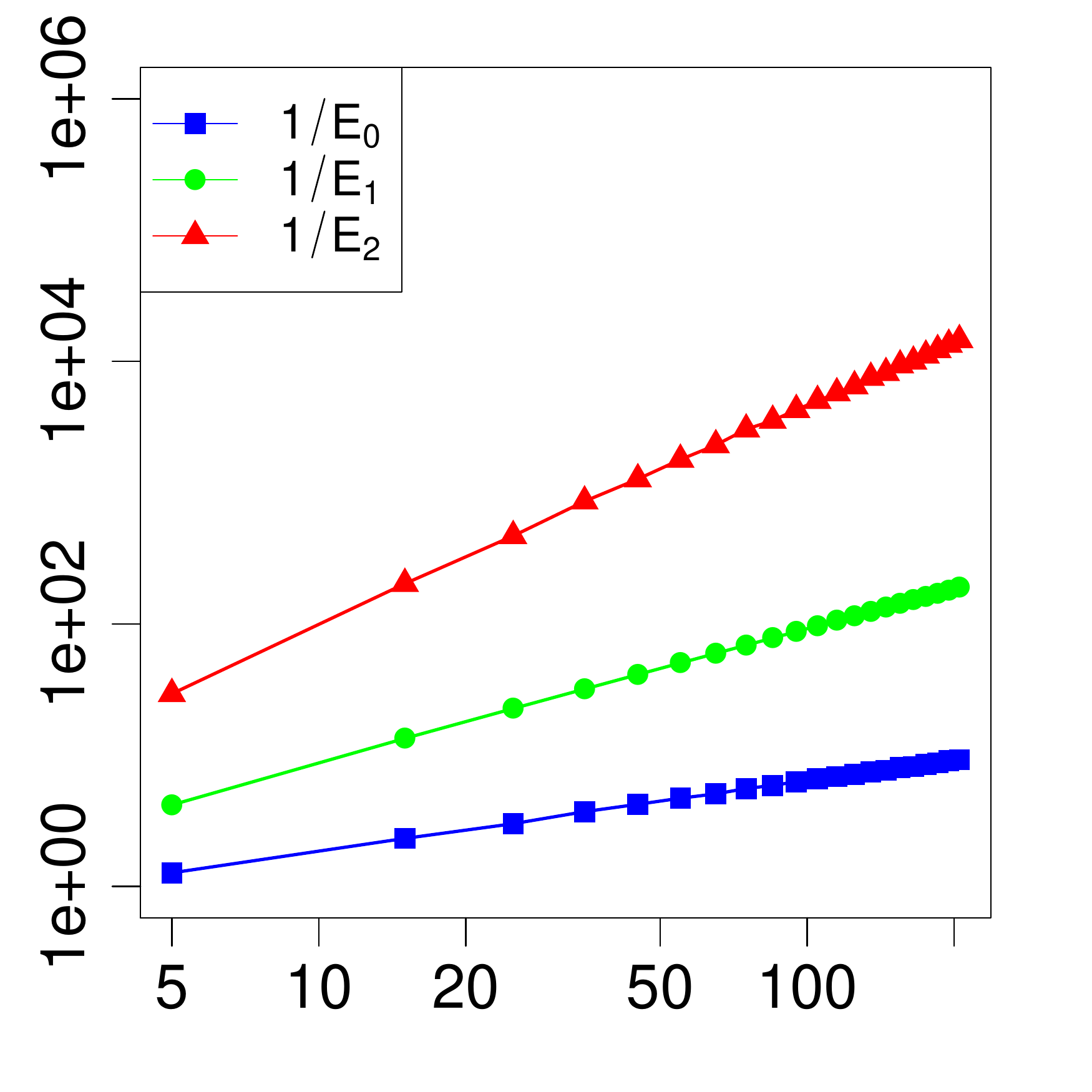}
            \vspace{-0.8cm}
            \caption{$\mat{S} = \begin{pmatrix} 5 & 1 \\ 1 & 5\end{pmatrix}$}
        \end{subfigure}
        \caption{Plots of $1 / E_i$ as a function of $\nu$, for various choices of $\mat{S}$. Both the horizontal and vertical axes are on a logarithmic scale. The plots clearly illustrate how the addition of correction terms from Theorem~\ref{thm:p.k.expansion} to the base approximation \eqref{eq:E.0} improves it.}
        \label{fig:loglog.errors.plots}
    \end{figure}
    \begin{figure}[ht]
        \captionsetup[subfigure]{labelformat=empty}
        \vspace{-0.5cm}
        \centering
        \begin{subfigure}[b]{0.22\textwidth}
            \centering
            \includegraphics[width=\textwidth, height=0.85\textwidth]{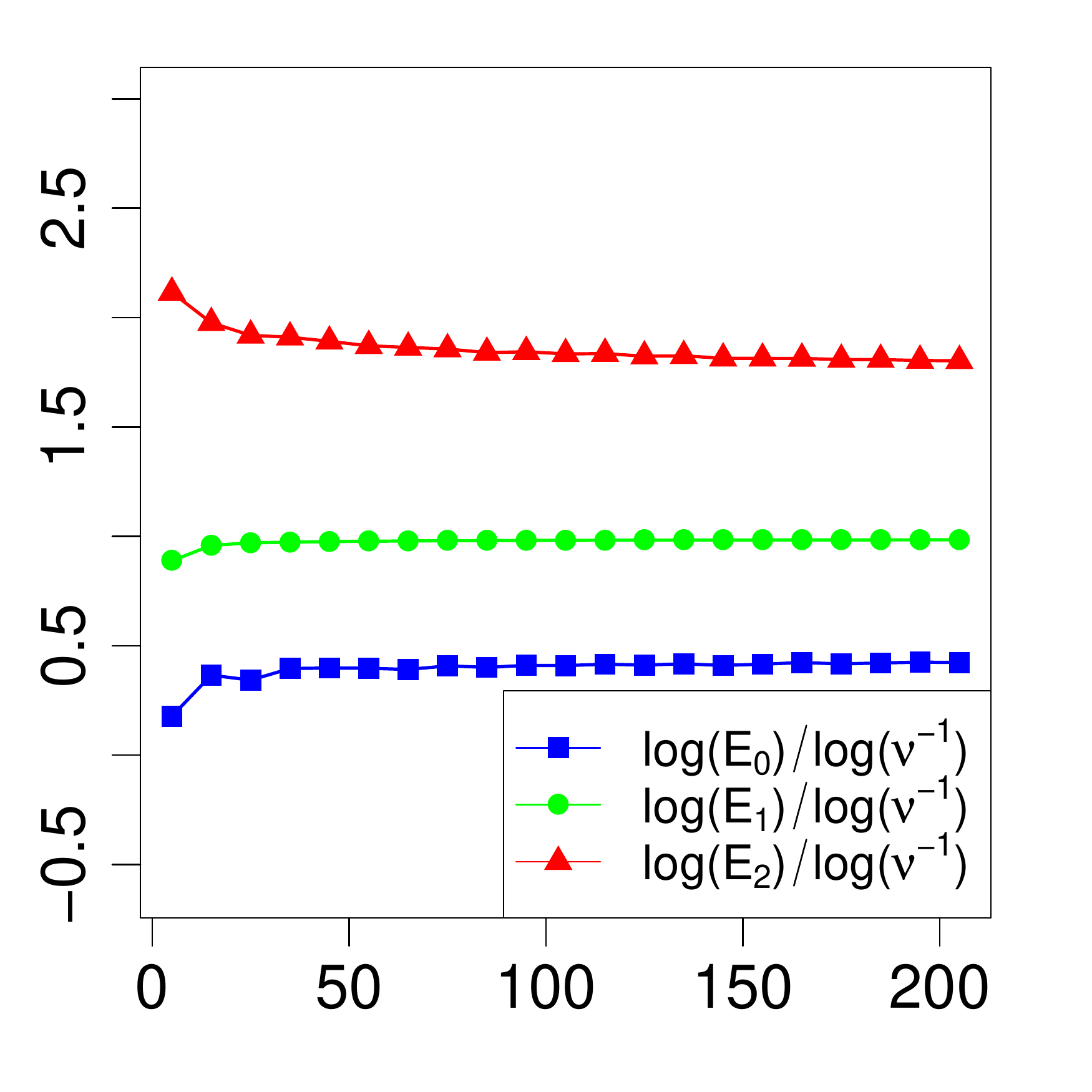}
            \vspace{-0.8cm}
            \caption{$\mat{S} = \begin{pmatrix} 2 & 1 \\ 1 & 2\end{pmatrix}$}
        \end{subfigure}
        \quad
        \begin{subfigure}[b]{0.22\textwidth}
            \centering
            \includegraphics[width=\textwidth, height=0.85\textwidth]{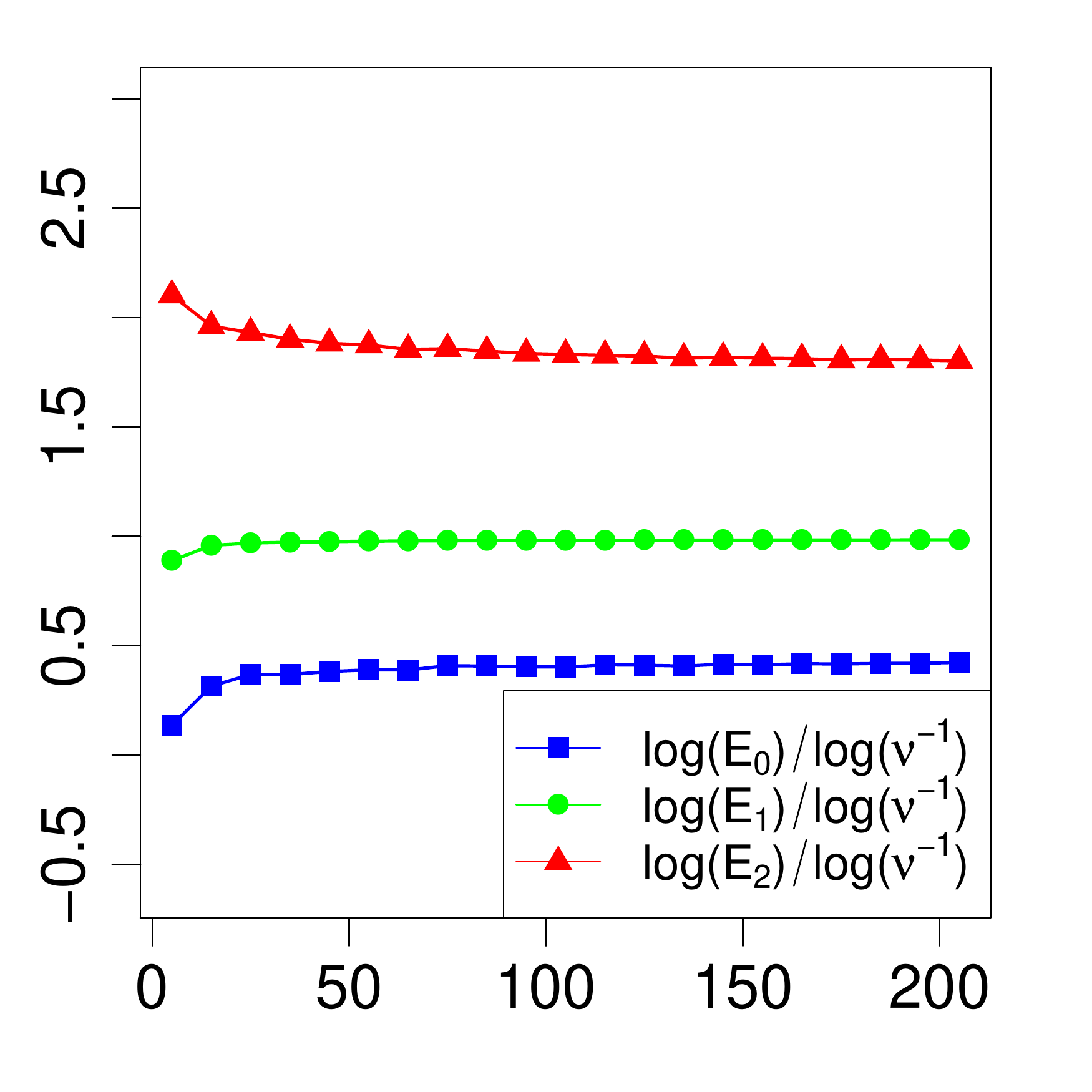}
            \vspace{-0.8cm}
            \caption{$\mat{S} = \begin{pmatrix} 2 & 1 \\ 1 & 3\end{pmatrix}$}
        \end{subfigure}
        \quad
        \begin{subfigure}[b]{0.22\textwidth}
            \centering
            \includegraphics[width=\textwidth, height=0.85\textwidth]{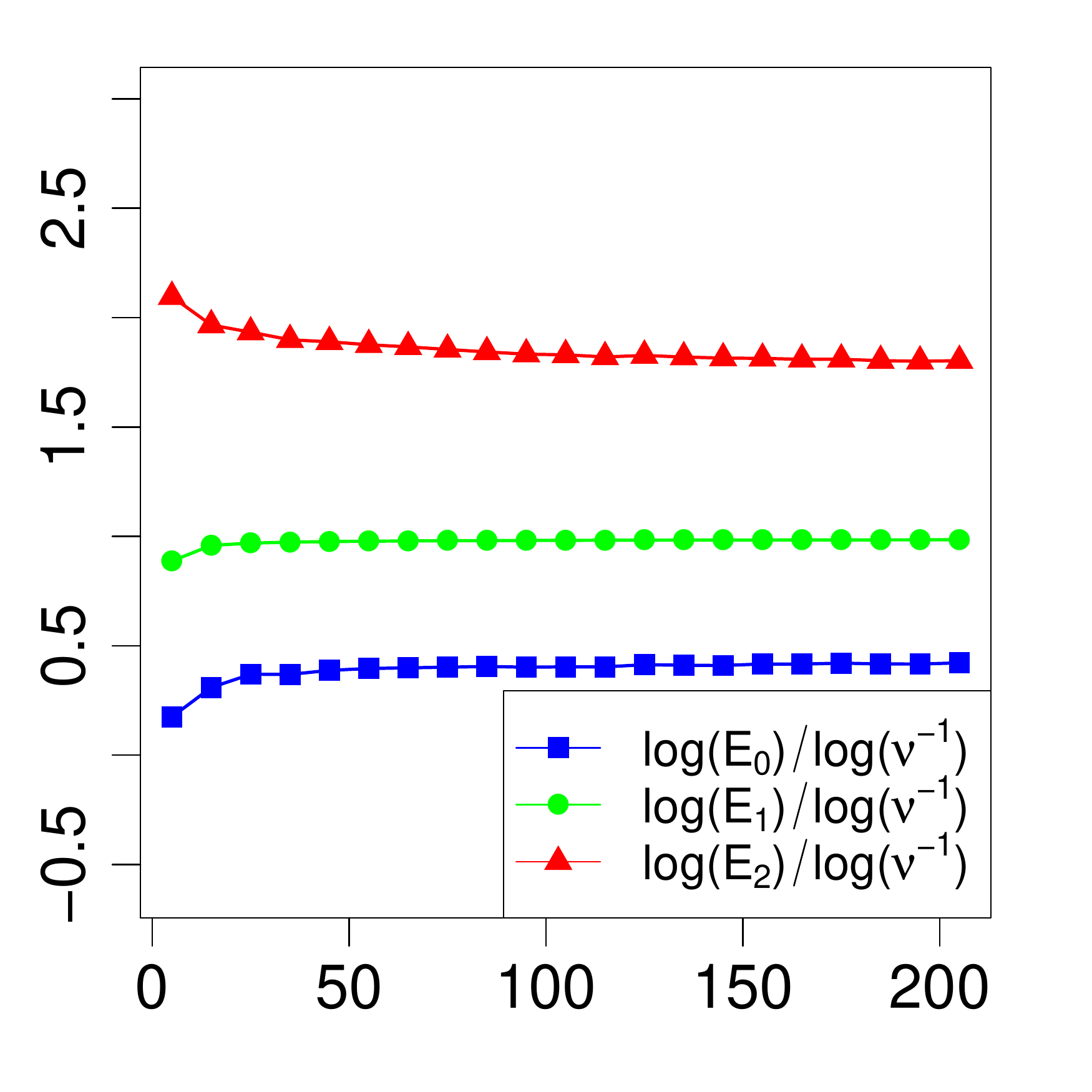}
            \vspace{-0.8cm}
            \caption{$\mat{S} = \begin{pmatrix} 2 & 1 \\ 1 & 4\end{pmatrix}$}
        \end{subfigure}
        \quad
        \begin{subfigure}[b]{0.22\textwidth}
            \centering
            \includegraphics[width=\textwidth, height=0.85\textwidth]{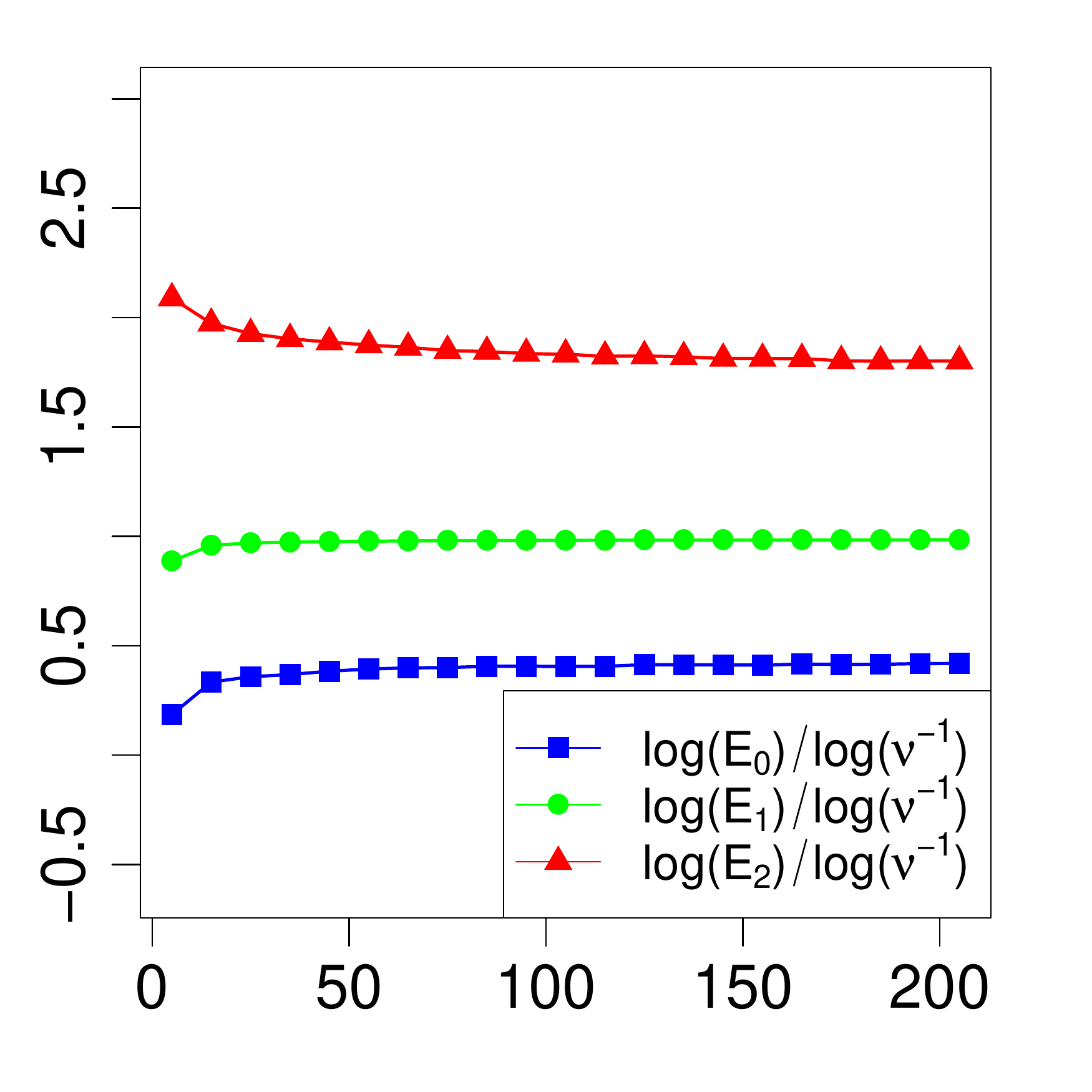}
            \vspace{-0.8cm}
            \caption{$\mat{S} = \begin{pmatrix} 2 & 1 \\ 1 & 5\end{pmatrix}$}
        \end{subfigure}
        \begin{subfigure}[b]{0.22\textwidth}
            \centering
            \includegraphics[width=\textwidth, height=0.85\textwidth]{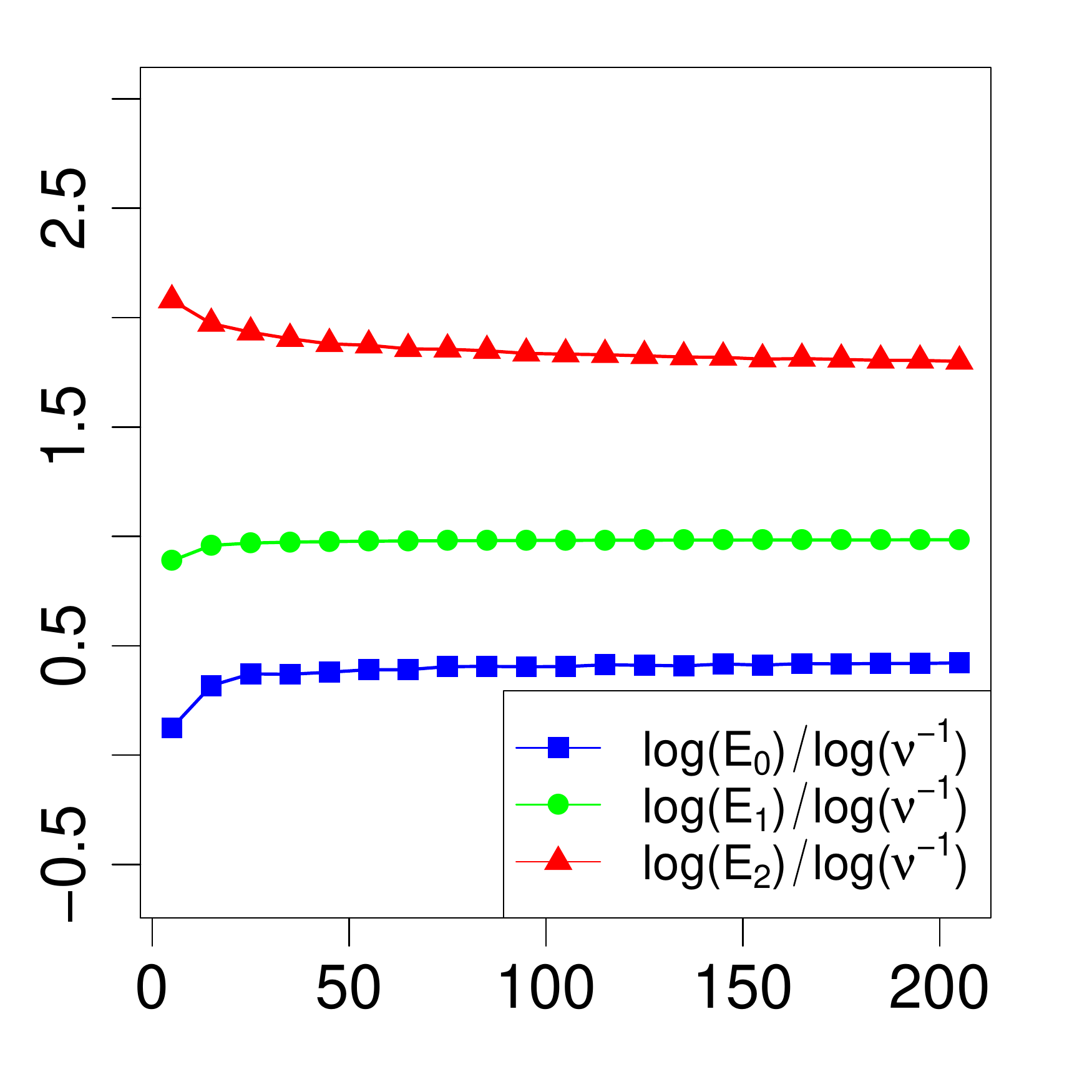}
            \vspace{-0.8cm}
            \caption{$\mat{S} = \begin{pmatrix} 3 & 1 \\ 1 & 2\end{pmatrix}$}
        \end{subfigure}
        \quad
        \begin{subfigure}[b]{0.22\textwidth}
            \centering
            \includegraphics[width=\textwidth, height=0.85\textwidth]{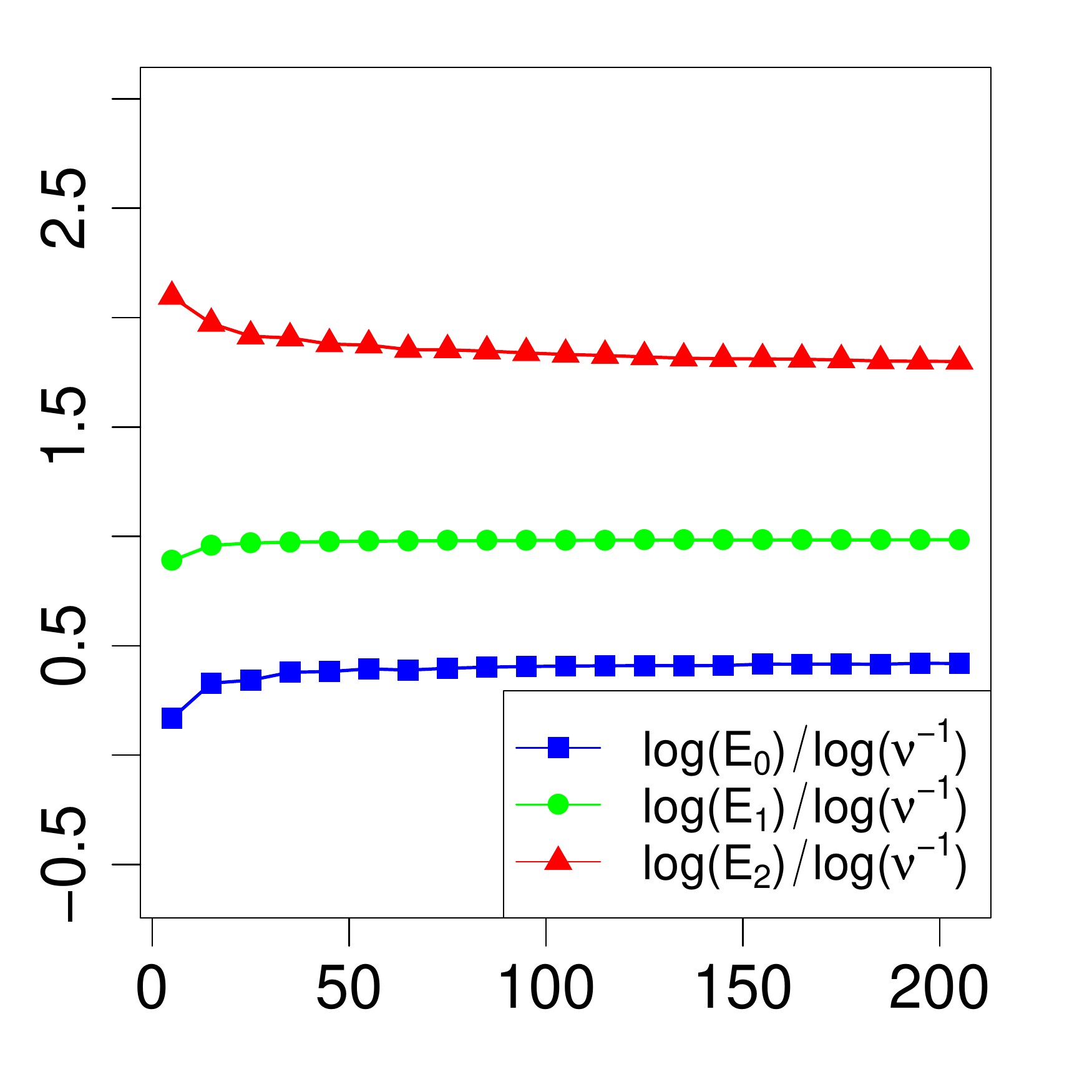}
            \vspace{-0.8cm}
            \caption{$\mat{S} = \begin{pmatrix} 3 & 1 \\ 1 & 3\end{pmatrix}$}
        \end{subfigure}
        \quad
        \begin{subfigure}[b]{0.22\textwidth}
            \centering
            \includegraphics[width=\textwidth, height=0.85\textwidth]{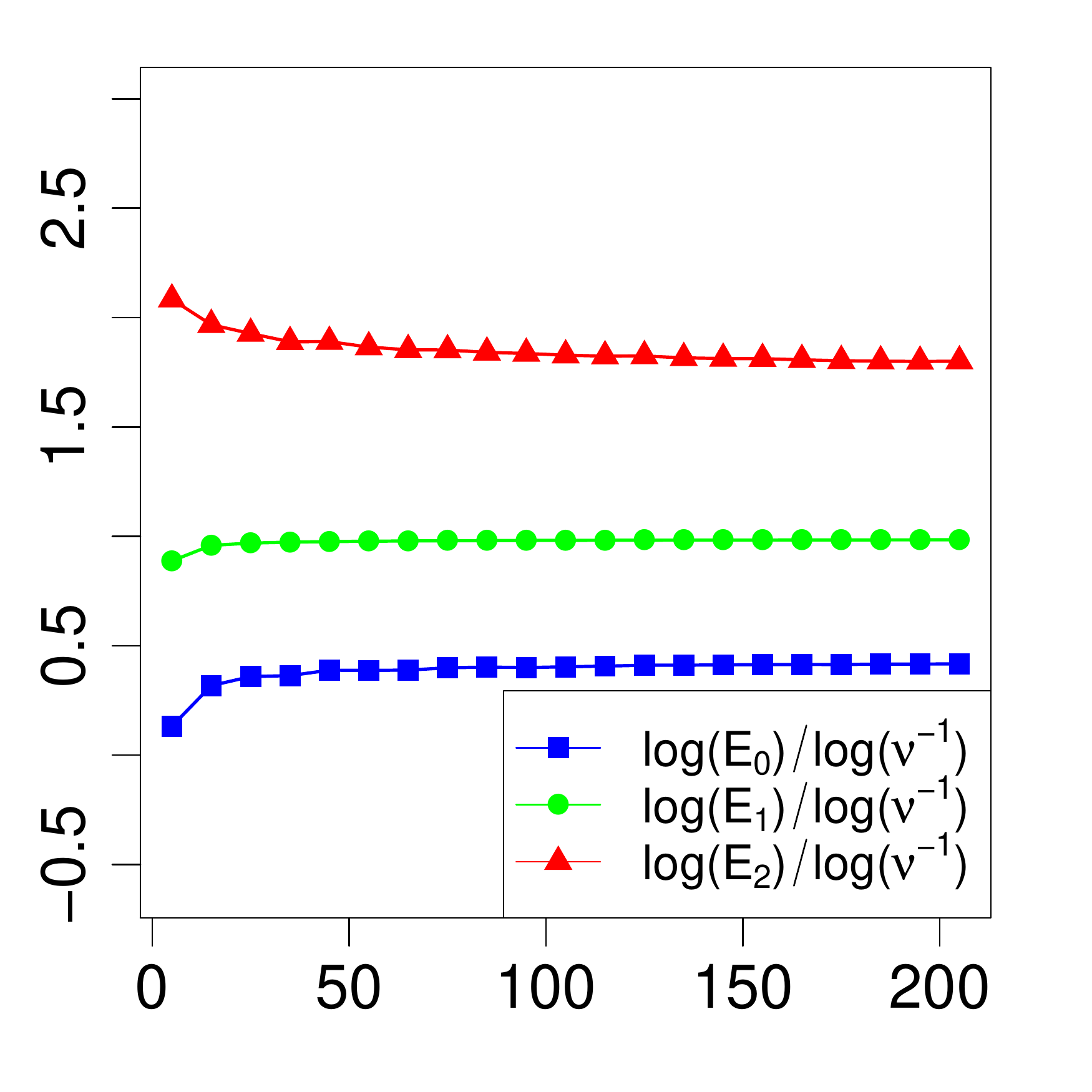}
            \vspace{-0.8cm}
            \caption{$\mat{S} = \begin{pmatrix} 3 & 1 \\ 1 & 4\end{pmatrix}$}
        \end{subfigure}
        \quad
        \begin{subfigure}[b]{0.22\textwidth}
            \centering
            \includegraphics[width=\textwidth, height=0.85\textwidth]{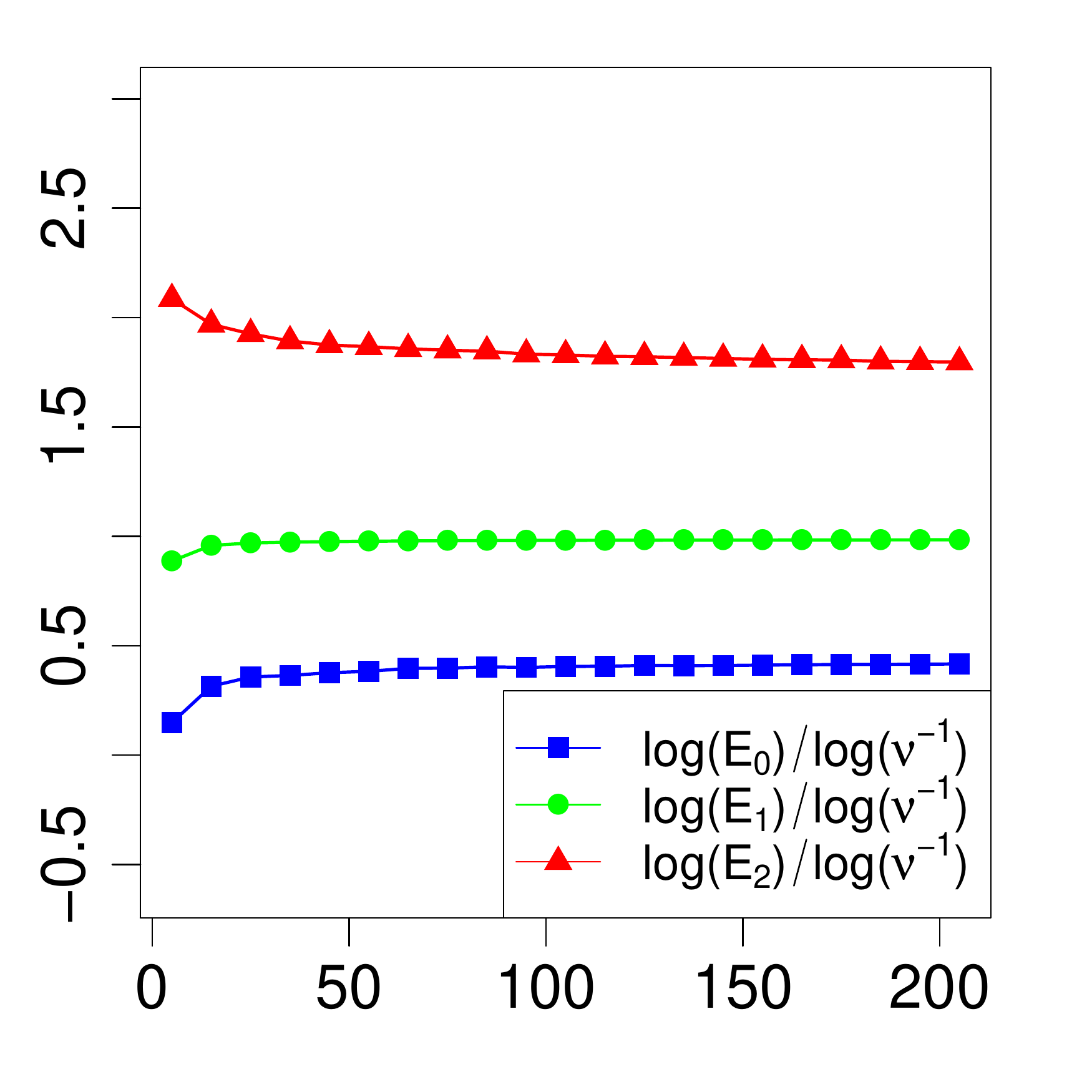}
            \vspace{-0.8cm}
            \caption{$\mat{S} = \begin{pmatrix} 3 & 1 \\ 1 & 5\end{pmatrix}$}
        \end{subfigure}
        \begin{subfigure}[b]{0.22\textwidth}
            \centering
            \includegraphics[width=\textwidth, height=0.85\textwidth]{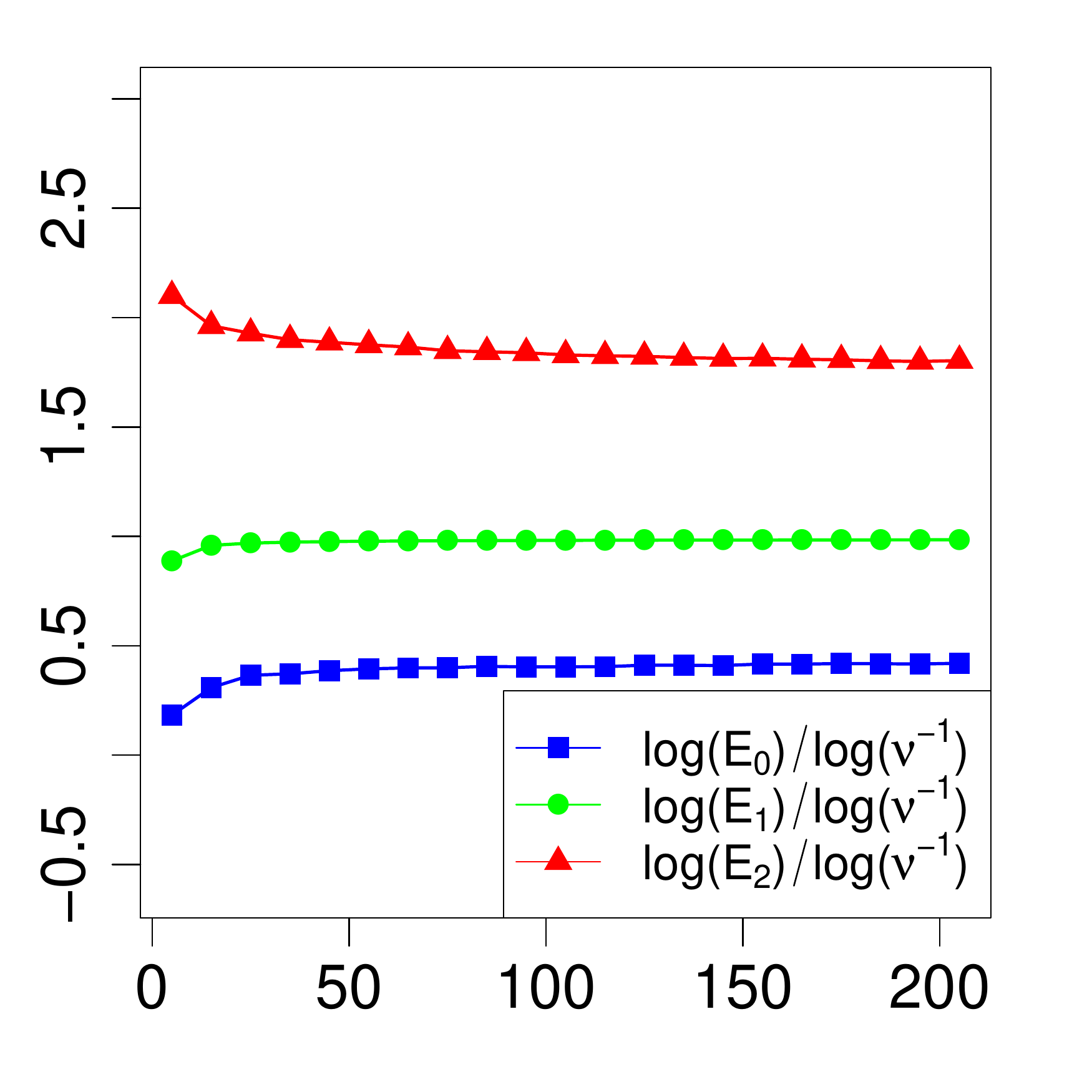}
            \vspace{-0.8cm}
            \caption{$\mat{S} = \begin{pmatrix} 4 & 1 \\ 1 & 2\end{pmatrix}$}
        \end{subfigure}
        \quad
        \begin{subfigure}[b]{0.22\textwidth}
            \centering
            \includegraphics[width=\textwidth, height=0.85\textwidth]{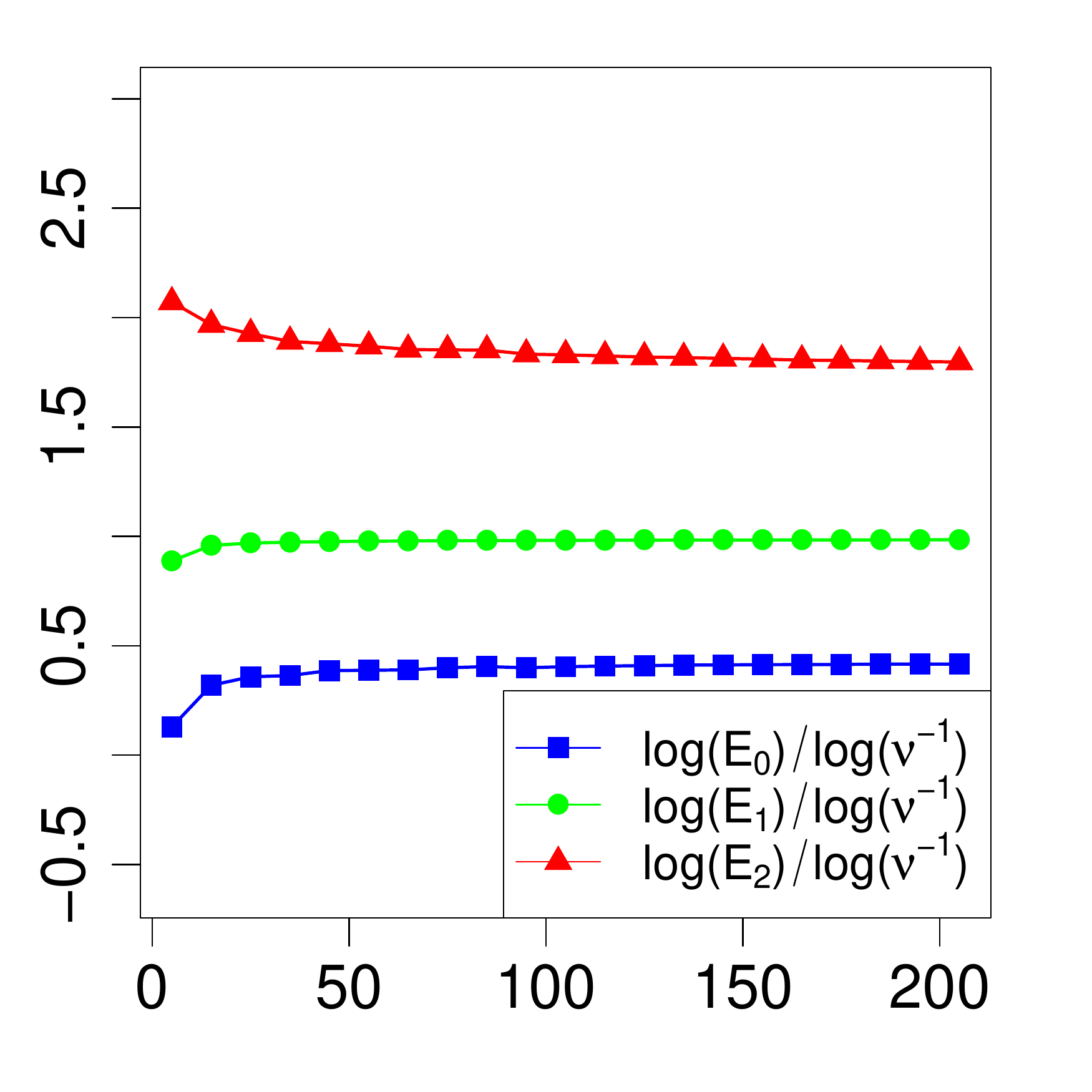}
            \vspace{-0.8cm}
            \caption{$\mat{S} = \begin{pmatrix} 4 & 1 \\ 1 & 3\end{pmatrix}$}
        \end{subfigure}
        \quad
        \begin{subfigure}[b]{0.22\textwidth}
            \centering
            \includegraphics[width=\textwidth, height=0.85\textwidth]{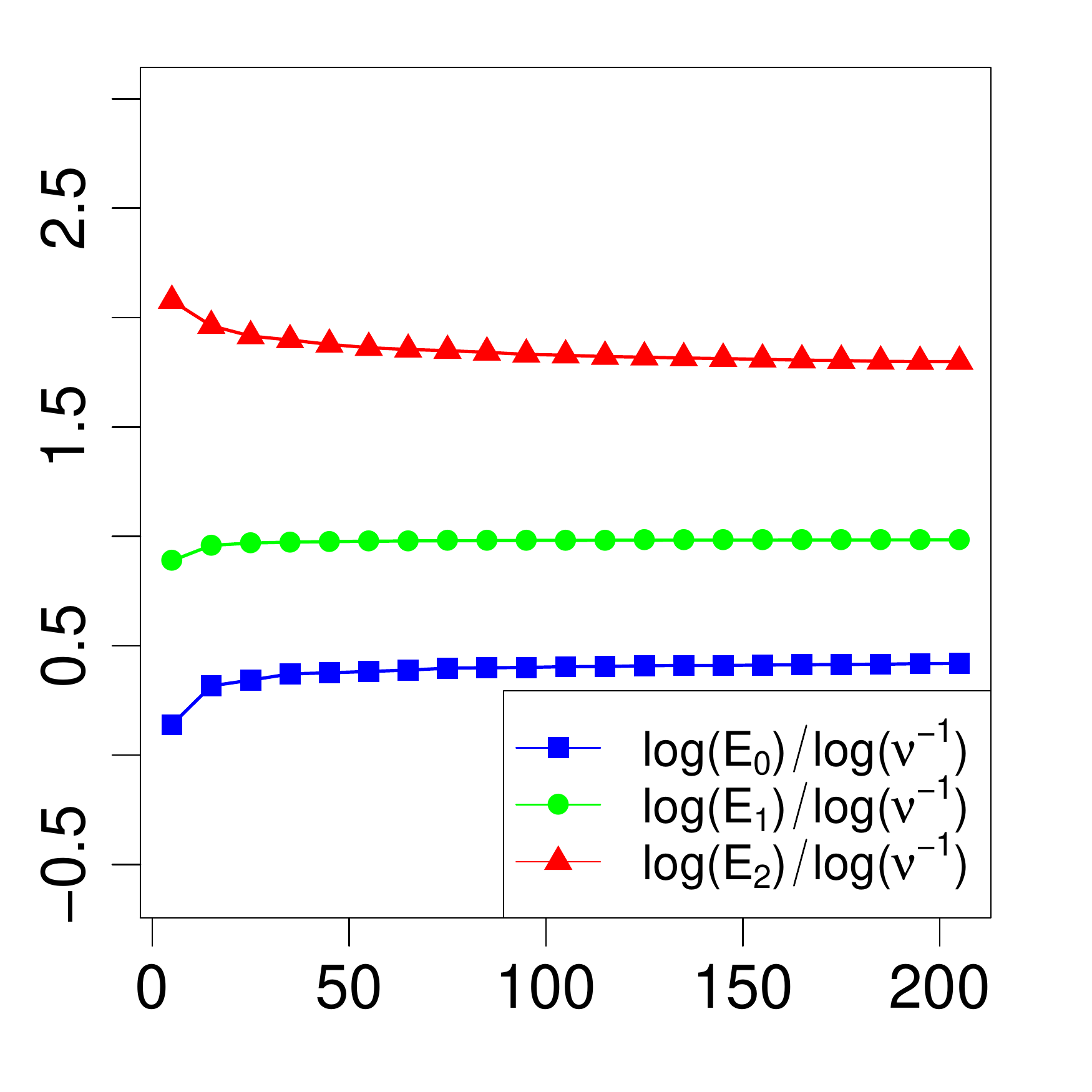}
            \vspace{-0.8cm}
            \caption{$\mat{S} = \begin{pmatrix} 4 & 1 \\ 1 & 4\end{pmatrix}$}
        \end{subfigure}
        \quad
        \begin{subfigure}[b]{0.22\textwidth}
            \centering
            \includegraphics[width=\textwidth, height=0.85\textwidth]{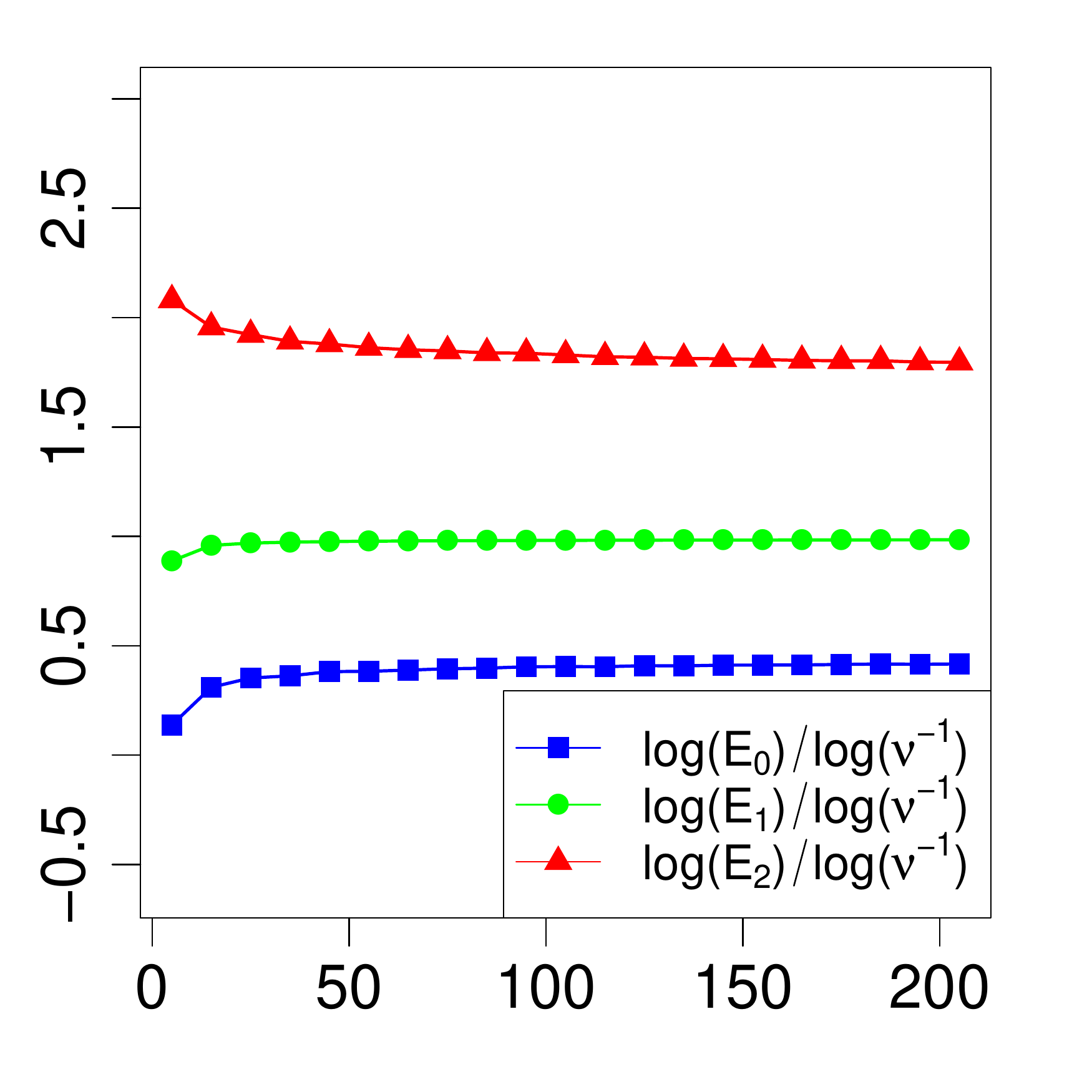}
            \vspace{-0.8cm}
            \caption{$\mat{S} = \begin{pmatrix} 4 & 1 \\ 1 & 5\end{pmatrix}$}
        \end{subfigure}
        \begin{subfigure}[b]{0.22\textwidth}
            \centering
            \includegraphics[width=\textwidth, height=0.85\textwidth]{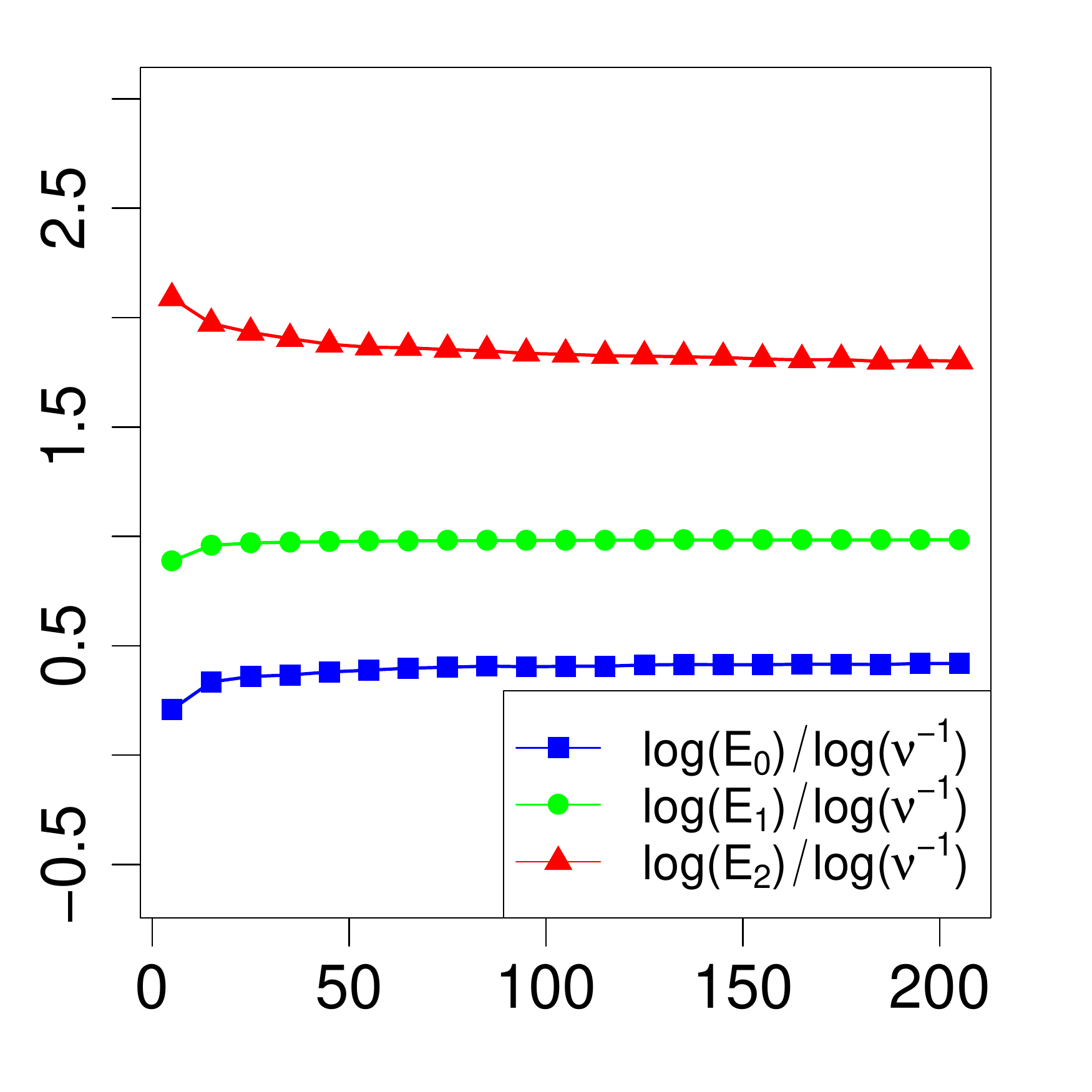}
            \vspace{-0.8cm}
            \caption{$\mat{S} = \begin{pmatrix} 5 & 1 \\ 1 & 2\end{pmatrix}$}
        \end{subfigure}
        \quad
        \begin{subfigure}[b]{0.22\textwidth}
            \centering
            \includegraphics[width=\textwidth, height=0.85\textwidth]{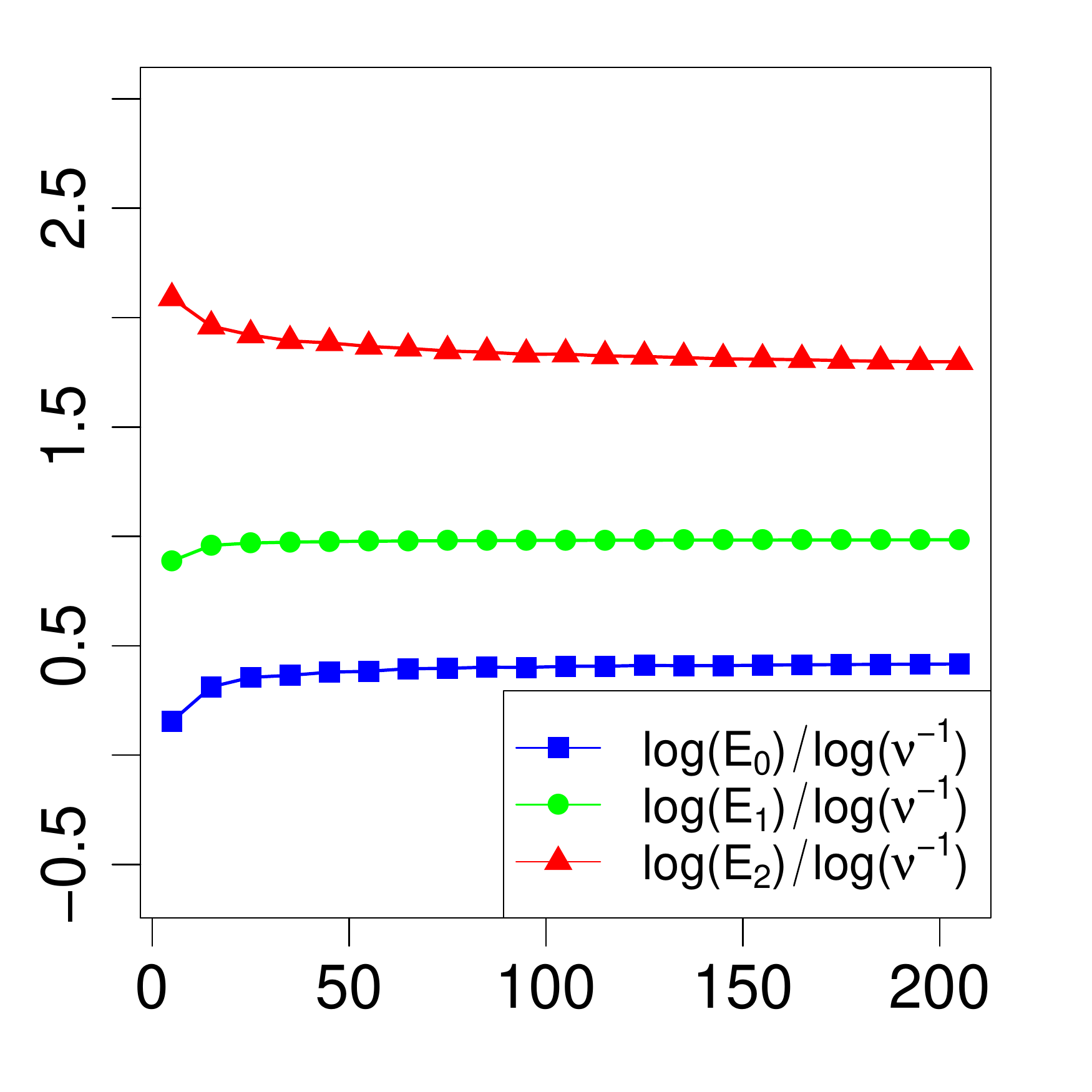}
            \vspace{-0.8cm}
            \caption{$\mat{S} = \begin{pmatrix} 5 & 1 \\ 1 & 3\end{pmatrix}$}
        \end{subfigure}
        \quad
        \begin{subfigure}[b]{0.22\textwidth}
            \centering
            \includegraphics[width=\textwidth, height=0.85\textwidth]{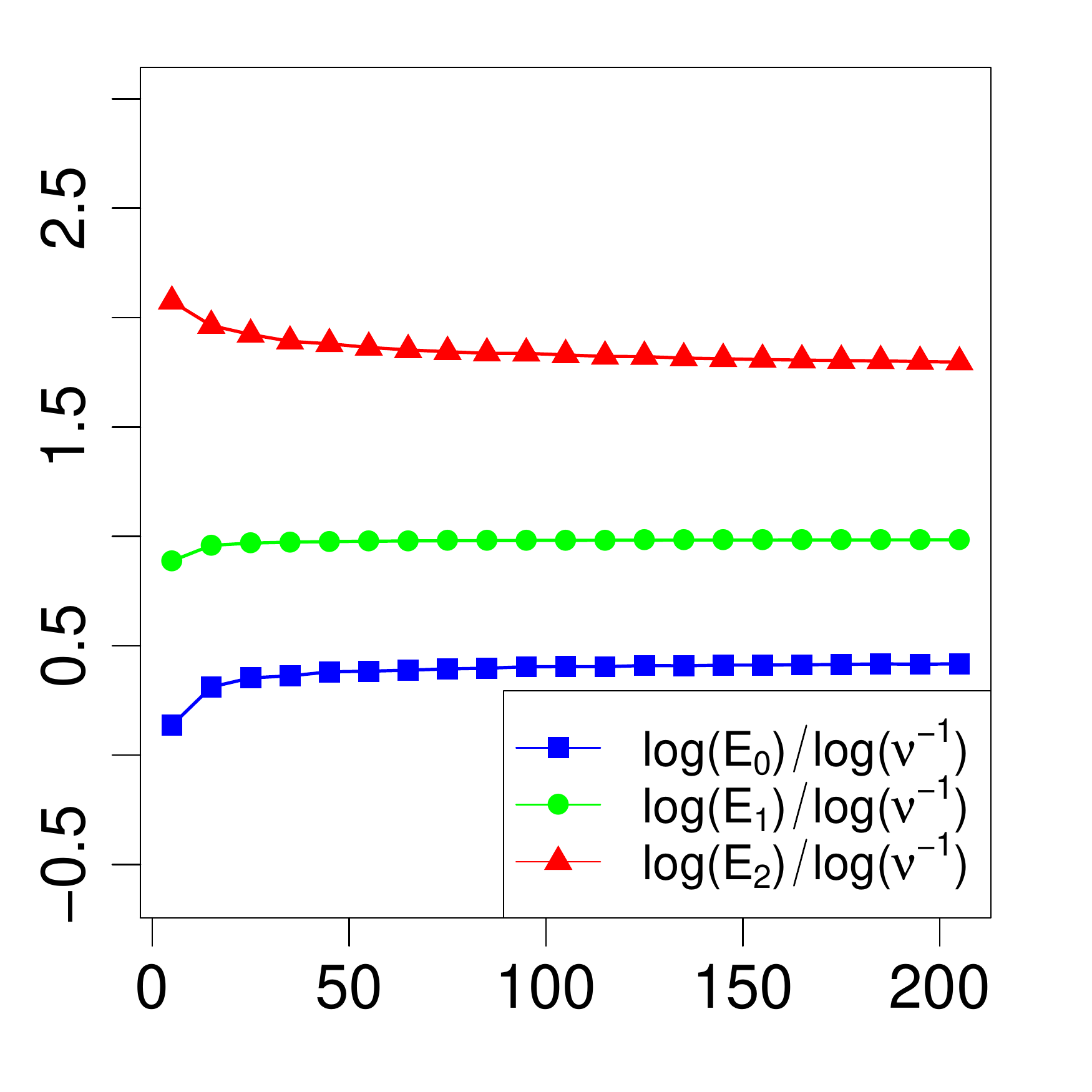}
            \vspace{-0.8cm}
            \caption{$\mat{S} = \begin{pmatrix} 5 & 1 \\ 1 & 4\end{pmatrix}$}
        \end{subfigure}
        \quad
        \begin{subfigure}[b]{0.22\textwidth}
            \centering
            \includegraphics[width=\textwidth, height=0.85\textwidth]{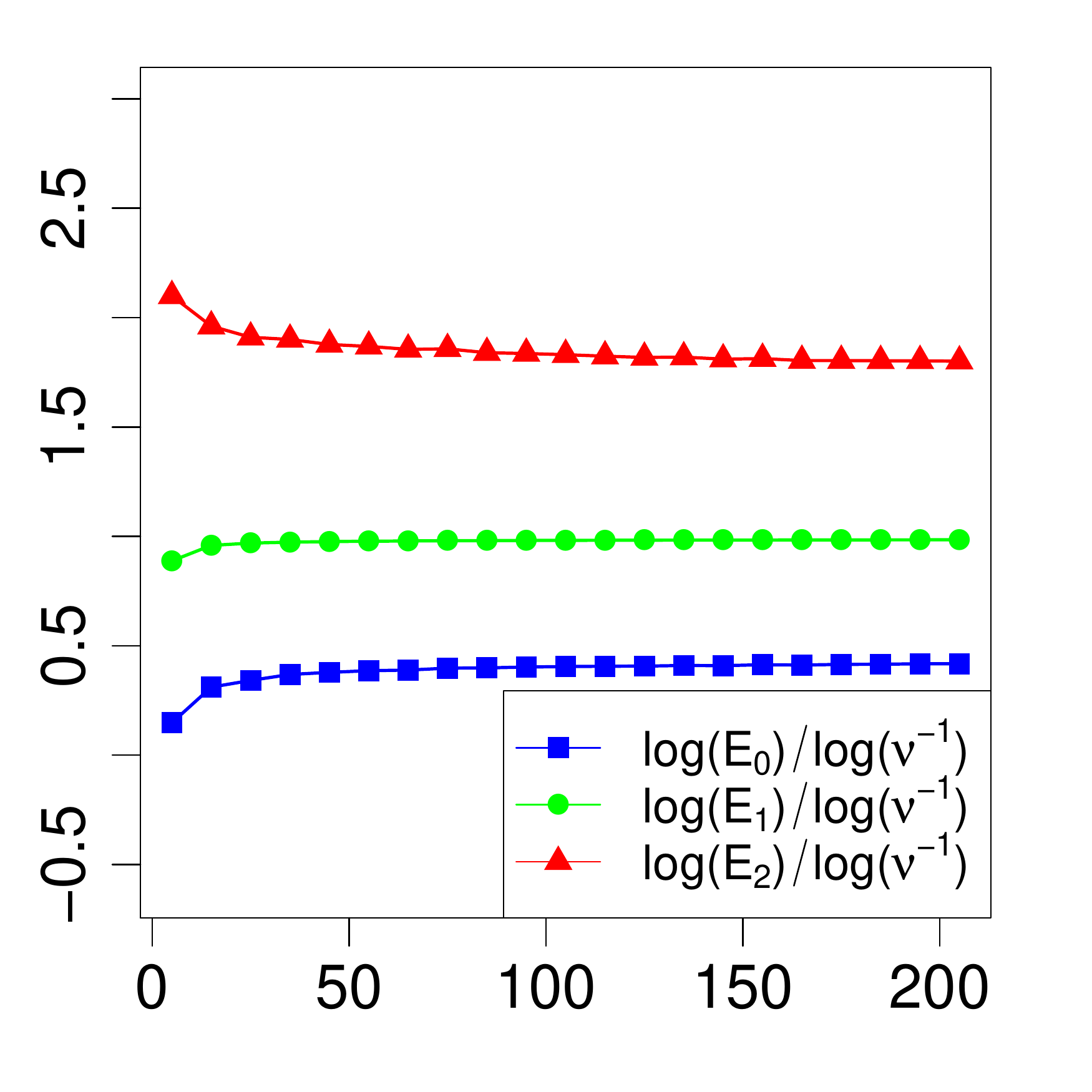}
            \vspace{-0.8cm}
            \caption{$\mat{S} = \begin{pmatrix} 5 & 1 \\ 1 & 5\end{pmatrix}$}
        \end{subfigure}
        \caption{Plots of $\log E_i / \log (\nu^{-1})$ as a function of $\nu$, for various choices of $\mat{S}$. The plots confirm \eqref{eq:liminf.exponent.bound} for our choices of $\mat{S}$ and bring strong evidence for the validity of Theorem~\ref{thm:p.k.expansion}.}
        \label{fig:error.exponents.plots}
    \end{figure}

\section{Applications}\label{sec:application}

    \subsection{Asymptotic properties of Wishart asymmetric kernel estimators}\label{sec:Wishart.asymmetric.kernel}

        Symmetric positive definite (SPD) matrix data are prevalent in modern statistical applications.
        As pointed out by \citet{Hadjicosta2019phd} and \citet{MR4169380}, where goodness-of-fit tests for the Wishart distribution were developed based on integral transforms, factor analysis, diffusion tensor imaging, CMB radiation measurements, volatility models in finance, wireless communication systems and polarimetric radar imaging are just a few areas where SPD matrix data might be observed.
        Some articles have dealt with methods of density estimation on this space but the literature remains relatively scarce.
        \citet{doi:10.1109/ICPR.2014.165} show how truncated Fourier series can be used for various applications, \citet{Kim_Richards_2008_worshop,MR2838725} and \citet{MR3012414} explore the deconvolution of Wishart mixtures, \citet{MR3606419} adapt the kernel estimator on compact Riemannian manifolds introduced by \citet{MR2179289} to compact subsets of the space of multivariate Gaussian distributions under the Fisher information metric and the Wasserstein metric, and \citet{MR4172886} defines a kernel density estimator on symmetric spaces of non-compact type (similar to Pelletier's but for which Helgason–Fourier transforms are defined) and proves an upper bound on the convergence rate that is analogous to the minimax rate of classical kernel estimators on Euclidean spaces.

        In a recent preprint, \citet{arXiv:2009.01983} consider log-Gaussian kernel estimators (based on the logarithm map for SPD matrices) and a variant of the Wishart asymmetric kernel estimator that is slightly different from our definition below in \eqref{eq:Wishart.estimator}. They prove various asymptotic properties for the former and a simulation study compares them both.
        If we were to apply traditional multivariate kernel estimators to the vectorization (in $\R^{d(d + 1)/2}$, recall \eqref{eq:vectorization}) of a sequence of i.i.d.\ random SPD matrices, then these estimators would misbehave near the boundary because of the condition on the eigenvalues (i.e., that they remain positive), and the usual boundary kernel modifications would not be appropriate either since positive definiteness is not a condition that can be translated to individual bounds on the entries of a matrix.
        To the best of our knowledge, \cite{arXiv:2009.01983} is the only paper that presents estimators on the space of SPD matrices that address (implicitly) the spill over problem of traditional multivariate kernel estimators caused by the boundary condition on the eigenvalues.
        Similarly to \citet{arXiv:2009.01983}, we can construct a new density estimator with a Wishart asymmetric kernel that creates a variable smoothing in our space and has a uniformly negligible bias near the boundary of $\mathcal{S}_{++}^{\hspace{0.3mm}d}$.

        In terms of applications, our new density estimation method on $\mathcal{S}_{++}^{\hspace{0.3mm}d}$ could be used, apart from visualization purposes, for nonparametric alternatives to regression and classification (both supervised and unsupervised) in any of the fields mentioned at the beginning of the first paragraph in this section.
        The favorable boundary properties of our estimator (the proof of Theorem~\ref{thm:bias} below shows that the pointwise bias is asymptotically uniformly negligible near the boundary) means that it could be particularly useful for scarce data sets and/or data sets with clusters of observations near the boundary of $\mathcal{S}_{++}^{\hspace{0.3mm}d}$.

        Here is the definition of our estimator.
        Assume that we have a sequence of observations $\mat{X}_1, \dots, \mat{X}_n\in \mathcal{S}_{++}^{\hspace{0.3mm}d}$ that are independent and $F$ distributed ($F$ is unknown), with density $f$ supported on $\mathcal{S}_{++}^{\hspace{0.3mm}d}$ for some $d\in \N$.
        Then, for a given bandwidth parameter $b > 0$, let
        \begin{equation}\label{eq:Wishart.estimator}
            \hat{f}_{n,b}(\mat{S}) \leqdef \frac{1}{n} \sum_{i=1}^n K_{1/b, b \hspace{0.3mm} \mat{S}}(\mat{X}_i), \quad \mat{S}\in \mathcal{S}_{++}^{\hspace{0.3mm}d},
        \end{equation}
        be the (or a) Wishart asymmetric kernel estimator for the density function $f$, where $K_{1/b, b \hspace{0.3mm} \mat{S}}(\cdot)$ is defined in \eqref{eq:Wishart.density}.
        The estimator $\hat{f}_{n,b}$ can be seen as a continuous example in the broader class of multivariate associated kernel estimators introduced by \citet{MR3760293,Kokonendji_Some_2021}.
        It is also a natural generalization of a slight variant of the (unmodified) Gamma kernel estimator introduced by \citet{MR1794247} because the Wishart distribution (recall \eqref{eq:Wishart.density}) is a matrix-variate analogue of the Gamma distribution.
        In \cite{MR1794247,MR2179543,MR2206532,MR2454617,MR2568128,MR2756423,MR2595129,MR2801351,MR3304359}, many asymptotic properties for Gamma kernel estimators of density functions supported on the half-line $[0,\infty)$ were studied, among other things: pointwise bias, pointwise variance, mean squared error, mean integrated squared error, asymptotic normality and uniform strong consistency.
        Also, bias reduction techniques were explored by \citet{MR3843043} and \citet{MR3384258}, and adaptative Bayesian methods of bandwidth selection were presented by \citet{doi:10.1080/03610918.2020.1828921} and \citet{doi:10.1080/02664763.2021.1881456}.

        Below, we show how some of the asymptotic properties of $\hat{f}_{n,b}$ can be studied using the asymptotic expansion developed in Theorem~\ref{thm:p.k.expansion}.
        Assume that $f$ is Lipschitz continuous and bounded on $\mathcal{S}_{++}^{\hspace{0.3mm}d}$.
        (To make sense of this assumption, note that $\mathcal{S}_{++}^{\hspace{0.3mm}d}$ is an open and convex subset in the space of symmetric matrices of size $d\times d$, which itself is isomorphic to $\R^{d(d + 1)/2}$.)
        Then, straightforward calculations show that, for any given $\mat{S}\in \mathcal{S}_{++}^{\hspace{0.3mm}d}$,
        \begin{align}\label{eq:histogram.estimator.var.asymp}
            \VV(\hat{f}_{n,b}(\mat{S})) = n^{-1} \, \EE\left[\left(K_{1/b, b \hspace{0.3mm} \mat{S}}(\mat{X})\right)^2\right] - n^{-1} \left(\EE\left[K_{1/b, b \hspace{0.3mm} \mat{S}}(\mat{X})\right]\right)^2 = n^{-1} \, \EE\left[\left(K_{1/b, b \hspace{0.3mm} \mat{S}}(\mat{X})\right)^2\right] - \OO_{d,\mat{S}}(n^{-1}),
        \end{align}
        where
        \begin{align}\label{eq:histogram.estimator.var.asymp.next}
            \EE\left[\left(K_{1/b, b \hspace{0.3mm} \mat{S}}(\mat{X})\right)^2\right]
            &\stackrel{\eqref{eq:LLT.order.2}}{=} \int_{\mathcal{S}^{\hspace{0.3mm}d}} g_{1/b, b \hspace{0.3mm} \mat{S}}^2(\mat{X}) \, f(\mat{X}) \, \rd \mat{X} + \oo_{d,\mat{S}}(1) \stackrel{\phantom{\eqref{eq:LLT.order.2}}}{=} \frac{2^{-d(d + 1)/4} \, (f(\mat{S}) + \OO_{d,\mat{S}}(b^{1/2}))}{\sqrt{(2\pi)^{d(d + 1)/2} \, 2^{-d(d - 1)/2} |\sqrt{2 b} \, \mat{S}|^{d + 1}}} \underbrace{\int_{\mathcal{S}^{\hspace{0.3mm}d}} g_{1/b, \frac{b}{\sqrt{2}} \hspace{0.3mm} \mat{S}}(\mat{X}) \rd \mat{X}}_{=~1 + \oo_{d,\mat{S}}(1)} \, + \, \oo_{d,\mat{S}}(1) \notag \\[-4mm]
            &\stackrel{\phantom{\eqref{eq:LLT.order.2}}}{=} b^{-d(d + 1)/4} \, \frac{|\sqrt{\pi} \, \mat{S}|^{-\frac{d + 1}{2}}}{2^{d(d + 2)/2}} \, (f(\mat{S}) + \OO_{d,\mat{S}}(b^{1/2})).
        \end{align}
        By applying this last estimate in \eqref{eq:histogram.estimator.var.asymp}, we obtain the pointwise variance.

        \begin{proposition}[Pointwise variance]\label{prop:pointwise.variance}
            Assume that $f$ is Lipschitz continuous and bounded on $\mathcal{S}_{++}^{\hspace{0.3mm}d}$.
            For any given $\mat{S}\in \mathcal{S}_{++}^{\hspace{0.3mm}d}$, we have
            \begin{equation}
                \VV(\hat{f}_{n,b}(\mat{S})) = n^{-1} b^{-d(d + 1)/4} \, \frac{|\sqrt{\pi} \, \mat{S}|^{-\frac{d + 1}{2}}}{2^{d(d + 2)/2}} \, (f(\mat{S}) + \OO_{d,\mat{S}}(b^{1/2})), \quad n\to \infty.
            \end{equation}
        \end{proposition}

        From this, other asymptotic expressions can be derived such as the mean squared error and the mean integrated squared error, and we can also optimize the bandwidth parameter $b$ with respect these expressions to implement a plug-in selection method exactly as we would in the setting of traditional multivariate kernel estimators, see, e.g., \citet[Section 6.5]{MR3329609} or \citet[Section 3.6]{MR3822372}.
        The expressions for the mean squared error and the mean integrated squared error (away from the boundary) are provided below in Corollary~\ref{cor:bias.var.implies.MSE.density} and Theorem~\ref{thm:MISE.optimal.density}, respectively, together with the corresponding optimal choice of $b$.
        For the sake of completeness, we also provide results on the pointwise bias of our estimator (see Theorem~\ref{thm:bias}), its pointwise variance as we move towards the boundary of $\mathcal{S}_{++}^{\hspace{0.3mm}d}$ (see Theorem~\ref{thm:variance.near.boundary}) and its asymptotic normality (see Theorem~\ref{thm:Theorem.3.2.and.3.3.Babu.Canty.Chaubey}).

        For each result in the remainder of this section, one of the following two assumptions will be used:
        \begin{alignat}{3}
            &\bullet \quad &&\text{The density $f$ is Lipschitz continuous and bounded on $\mathcal{S}_{++}^{\hspace{0.3mm}d}$.} \label{eq:assump:f.density} \\
            &\bullet \quad &&\text{The density $f$  and its first order partial derivatives are continuous and bounded on $\mathcal{S}_{++}^{\hspace{0.3mm}d}$,} \\[-1mm]
            & &&\text{and the second order partial derivatives of $f$ are uniformly continuous and bounded on $\mathcal{S}_{++}^{\hspace{0.3mm}d}$.} \label{eq:assump:f.density.2}
        \end{alignat}

        \begin{remark}\label{rem:open.and.convex}
            {\normalfont
            Again, to make sense of the above assumptions, note that $\mathcal{S}_{++}^{\hspace{0.3mm}d}$ is an open and convex subset in the space of symmetric matrices of size $d\times d$, denoted by $\mathcal{S}^{\hspace{0.3mm}d}$, which itself is isomorphic to $\R^{d(d + 1)/2}$.
            }
        \end{remark}

        We denote the expectation of $\hat{f}_{n,b}(\mat{S})$ by
        \begin{equation}\label{eq:f.b.star.representation.1}
            \begin{aligned}
                f_b(\mat{S}) \leqdef \EE\left[\hat{f}_{n,b}(\mat{S})\right]
                &= \EE[K_{1/b, b \hspace{0.3mm} \mat{S}}(\mat{X}_1)] = \int_{\mathcal{S}_{++}^{\hspace{0.3mm}d}} K_{1/b, b \hspace{0.3mm} \mat{S}}(\mat{M}) f(\mat{M}) \rd \mat{M}.
            \end{aligned}
        \end{equation}
        Alternatively, notice that if $\mat{W}_{\mat{S}}\sim \mathrm{Wishart}_d(1/b, b \hspace{0.3mm} \mat{S})$, then we also have the representation
        \begin{equation}\label{eq:f.b.star.representation.2}
            f_b(\mat{S}) = \EE[f(\mat{W}_{\mat{S}})].
        \end{equation}

        The asymptotics of the pointwise bias and variance were first computed by \citet{MR1718494,MR1794247} for Beta and Gamma kernel estimators, by \citet{doi:10.1016/j.jmva.2021.104832} for the Dirichlet kernel estimator of \citet{doi:10.2307/2347365}, and by  \citet{MR3760293,Kokonendji_Some_2021} for multivariate associated kernel estimators.
        The next two theorems below extend the (unmodified) Gamma case to our multidimensional setting.

        \begin{theorem}[Pointwise bias]\label{thm:bias}
            Assume that \eqref{eq:assump:f.density.2} holds.
            Then as $b\to 0$, and for all $\mat{S}\in \mathcal{S}_{++}^{\hspace{0.3mm}d}$, we have
            \begin{equation}\label{eq:thm:bias.var.density.eq.bias}
                \BB[\hat{f}_{n,b}(\mat{S})] = f_b(\mat{S}) - f(\mat{S}) = b \, g(\mat{S}) + \oo_{d,\mat{S}}(b),
            \end{equation}
            where
            \begin{equation}\label{def:g}
                \begin{aligned}
                    g(\mat{S})
                    &\leqdef \frac{1}{2} \sum_{\substack{\bb{i},\bb{j}\in [d]^2 \\ i_1 \leq i_2, \, j_1 \leq j_2}} \left[2 \hspace{0.3mm} B_d^{\top} (\mat{S} \otimes \mat{S}) B_d\right]_{(\bb{i},\bb{j})} \, \frac{\partial^2}{\partial \mat{S}_{\bb{i}} \partial \mat{S}_{\bb{j}}} f(\mat{S}),
                \end{aligned}
            \end{equation}
            and where $\otimes$ denotes the Kronecker product, and $[ \, \cdot \, ]_{(\bb{i},\bb{j})}$ means that we select the entry $((i_2 - 1) i_2 / 2 + i_1, (j_2 - 1) j_2 / 2 + j_1)$ in the $(d(d+1)/2) \times (d(d+1)/2)$ matrix.
        \end{theorem}

        \begin{theorem}[Pointwise variance near and away from the boundary of $\mathcal{S}_{++}^{\hspace{0.3mm}d}$]\label{thm:variance.near.boundary}
            Assume that \eqref{eq:assump:f.density} holds.
            Furthermore, let $\mat{K}\in \mathcal{S}_{++}^{\hspace{0.3mm}d}$ be independent of $b$ and assume that it diagonalizes as $\mat{K} = V \, \mathrm{diag}(\bb{\lambda}(\mat{K})) \, V^{\top}$.
            Pick any subset $\emptyset \subseteq \mathcal{J}\subseteq [d]$ and assume that
            \vspace{-1mm}
            \begin{equation}
                \mat{S} = V \, \mathrm{diag}(\bb{\lambda}_b(\mat{K})) \, V^{\top}, \quad \text{where the vector $\bb{\lambda}_b(\mat{K})$ satisfies} \quad [\bb{\lambda}_b(\mat{K})]_i \leqdef
                \begin{cases}
                    \lambda_i(\mat{K}), &\mbox{if } i\in [d]\backslash \mathcal{J}, \\
                    b \lambda_i(\mat{K}), &\mbox{if } i\in \mathcal{J}.
                \end{cases}
            \end{equation}
            In particular, with this choice of $\mat{S}$, note that $|\mat{S}| = b^{|\mathcal{J}|} |\mat{K}|$.
            Then we have, as $n\to \infty$,
            \begin{equation}\label{eq:thm:bias.var.density.eq.var}
                \VV(\hat{f}_{n,b}(\mat{S})) = n^{-1} b^{-r(d) / 2 - |\mathcal{J}| \frac{d + 1}{2}} \cdot \psi(\mat{K}) \, (f(\mat{S}) + \OO_{d,\mat{S}}(b^{1/2})) + \OO_{d,\mat{S}}(n^{-1}),
            \end{equation}
            where
            \begin{equation}\label{eq:def.psi}
                r(d) \leqdef \frac{d(d + 1)}{2} \qquad \text{and} \qquad \psi(\mat{K}) \leqdef \frac{|\sqrt{\pi} \, \mat{K}|^{-\frac{d + 1}{2}}}{2^{d(d + 2)/2}}.
            \end{equation}
        \end{theorem}

        The above theorem means that the pointwise variance is $\OO_{d,\mat{S}}(n^{-1} b^{-r(d) / 2})$ away from the boundary of $\mathcal{S}_{++}^{\hspace{0.3mm}d}$ and it gets multiplied by a factor $b^{-(d + 1)/2}$ everytime one of the $d$ eigenvalues approaches zero at a linear rate with respect to $b$. If $|\mathcal{J}|$ eigenvalues approach zero as $b\to 0$, then the pointwise variance is $\OO_{d,\mat{S}}(n^{-1} b^{-r(d) / 2 - |\mathcal{J}| \frac{d + 1}{2}})$.

        By combining Theorem~\ref{thm:bias} and Proposition~\ref{prop:pointwise.variance} (equivalently, Theorem~\ref{thm:variance.near.boundary} for $\mathcal{J} = \emptyset$), we can compute the mean squared error of our estimator and optimize the choice of the bandwidth parameter $b$.

        \begin{corollary}[Mean squared error]\label{cor:bias.var.implies.MSE.density}
            Assume that \eqref{eq:assump:f.density.2} holds.
            Then, as $n\to \infty$ and $b = b(n)\to 0$, and for any given $\mat{S}\in \mathcal{S}_{++}^{\hspace{0.3mm}d}$, we have
            \vspace{-1mm}
            \begin{equation}
                \begin{aligned}
                    \mathrm{MSE}[\hat{f}_{n,b}(\mat{S})]
                    &\leqdef \EE\left[\left|\hat{f}_{n,b}(\mat{S}) - f(\mat{S})\right|^2\right] = \VV(\hat{f}_{n,b}(\mat{S})) + \left(\BB[\hat{f}_{n,b}(\mat{S})]\right)^2 \\
                    &= n^{-1} b^{-r(d) / 2} \cdot \psi(\mat{S}) f(\mat{S}) + b^2 g^2(\mat{S}) + \OO_{d,\mat{S}}(n^{-1} b^{-r(d) / 2 + 1/2}) + \oo_{d,\mat{S}}(b^2).
                \end{aligned}
            \end{equation}
            In particular, if $f(\mat{S}) \cdot g(\mat{S}) \neq 0$, the asymptotically optimal choice of $b$, with respect to $\mathrm{MSE}$, is
            \begin{equation}\label{eq:b.opt.MSE}
                b_{\mathrm{opt}}(\mat{S}) = n^{-2/(r(d) + 4)} \left[\frac{r(d)}{4} \cdot \frac{\psi(\mat{S}) f(\mat{S})}{g^2(\mat{S})}\right]^{2/(r(d) + 4)},
            \end{equation}
            with
            \begin{equation}
                \mathrm{MSE}[\hat{f}_{n,b_{\mathrm{opt}}(\mat{S})}(\mat{S})] = n^{-4/(r(d) + 4)} \left[\frac{1 + \frac{r(d)}{4}}{\left(\frac{r(d)}{4}\right)^{\frac{r(d)}{r(d) + 4}}}\right] \frac{\left(\psi(\mat{S}) f(\mat{S})\right)^{4 / (r(d) + 4)}}{\left(g^2(\mat{S})\right)^{-r(d) / (r(d) + 4)}} + \oo_{d,\mat{S}}(n^{-4/(r(d) + 4)}), \quad n\to \infty.
            \end{equation}
            More generally, if $n^{2/(r(d) + 4)} \, b\to \lambda$ as $n\to \infty$ and $b = b(n)\to 0$ for some $\lambda > 0$, then
            \begin{equation}
                \mathrm{MSE}[\hat{f}_{n,b}(\mat{S})] = n^{-4/(r(d) + 4)} \left[\lambda^{-r(d)/2} \psi(\mat{S}) f(\mat{S}) + \lambda^2 g^2(\mat{S})\right] + \oo_{d,\mat{S}}(n^{-4/(r(d) + 4)}).
            \end{equation}
        \end{corollary}

        By integrating the MSE on the following subset of $\mathcal{S}_{++}^{\hspace{0.3mm}d}$,
        \begin{equation}
            \mathcal{S}_{++}^{\hspace{0.3mm}d}(\delta) \leqdef \left\{\mat{M}\in \mathcal{S}_{++}^{\hspace{0.3mm}d} : \delta \leq \lambda_1(\mat{M}) \leq \dots \leq \lambda_d(\mat{M}) \leq \delta^{-1}\right\}, \qquad 0 < \delta < 1,
        \end{equation}
        we obtain the next result.

        \begin{theorem}[Mean integrated squared error on $\mathcal{S}_{++}^{\hspace{0.3mm}d}(\delta)$]\label{thm:MISE.optimal.density}
            Assume that \eqref{eq:assump:f.density.2} holds.
            For a given $\delta\in (0,1)$, assume also that
            \begin{equation}\label{eq:further.assumptions.MISE}
                \int_{\mathcal{S}_{++}^{\hspace{0.3mm}d}(\delta)} \psi(\mat{S}) f(\mat{S}) \rd \mat{S} < \infty, \qquad \int_{\mathcal{S}_{++}^{\hspace{0.3mm}d}(\delta)} g^2(\mat{S}) \rd \mat{S} < \infty,
            \end{equation}
            where recall $\psi$ and $g$ were defined in \eqref{eq:def.psi} and \eqref{def:g}, respectively.
            Then, as $n\to \infty$ and $b = b(n)\to 0$, we have
            \begin{equation}\label{def:MISE.density}
                \begin{aligned}
                    \mathrm{MISE}_{\delta}[\hat{f}_{n,b}]
                    &\leqdef \int_{\mathcal{S}_{++}^{\hspace{0.3mm}d}(\delta)} \EE\left[\left|\hat{f}_{n,b}(\mat{S}) - f(\mat{S})\right|^2\right] \rd \mat{S} \\
                    &= n^{-1} b^{-r(d)/2} \int_{\mathcal{S}_{++}^{\hspace{0.3mm}d}(\delta)} \psi(\mat{S}) f(\mat{S}) \rd \mat{S} + b^2 \int_{\mathcal{S}_{++}^{\hspace{0.3mm}d}(\delta)} g^2(\mat{S}) \rd \mat{S} + \oo_{d,\delta}(n^{-1} b^{-r(d)/2}) + \oo_{d,\delta}(b^2).
                \end{aligned}
            \end{equation}
            In particular, if $\int_{\mathcal{S}_{++}^{\hspace{0.3mm}d}(\delta)} g^2(\mat{S}) \rd \mat{S} > 0$, the asymptotically optimal choice of $b$, with respect to $\mathrm{MISE}_{\delta}$, is
            \begin{equation}\label{eq:b.opt.MISE}
                b_{\mathrm{opt}} = n^{-2/(r(d) + 4)} \left[\frac{r(d)}{4} \cdot \frac{\int_{\mathcal{S}_{++}^{\hspace{0.3mm}d}(\delta)} \psi(\mat{S}) f(\mat{S}) \rd \mat{S}}{\int_{\mathcal{S}_{++}^{\hspace{0.3mm}d}(\delta)} g^2(\mat{S}) \rd \mat{S}}\right]^{2/(r(d) + 4)},
            \end{equation}
            with
            \begin{equation}
                \begin{aligned}
                    \mathrm{MISE}_{\delta}[\hat{f}_{n,b_{\mathrm{opt}}}]
                    &= n^{-4 / (r(d) + 4)} \left[\frac{1 + \frac{r(d)}{4}}{\left(\frac{r(d)}{4}\right)^{\frac{r(d)}{r(d) + 4}}}\right] \frac{\left(\int_{\mathcal{S}_{++}^{\hspace{0.3mm}d}(\delta)} \psi(\mat{S}) f(\mat{S}) \rd \mat{S}\right)^{4 / (r(d) + 4)}}{\left(\int_{\mathcal{S}_{++}^{\hspace{0.3mm}d}(\delta)} g^2(\mat{S}) \rd \mat{S}\right)^{-r(d) / (r(d) + 4)}} + \oo_{d,\delta}(n^{-4/(r(d) + 4)}), \quad n\to \infty.
                \end{aligned}
            \end{equation}
            More generally, if $n^{2/(r(d) + 4)} \, b\to \lambda$ as $n\to \infty$ and $b = b(n)\to 0$ for some $\lambda > 0$, then
            \begin{equation}
                \begin{aligned}
                    \mathrm{MISE}_{\delta}[\hat{f}_{n,b}]
                    &= n^{-4 / (r(d) + 4)} \left[\lambda^{-r(d)/2} \int_{\mathcal{S}_{++}^{\hspace{0.3mm}d}(\delta)} \psi(\mat{S}) f(\mat{S}) \rd \mat{S} + \lambda^2 \int_{\mathcal{S}_{++}^{\hspace{0.3mm}d}(\delta)} g^2(\mat{S}) \rd \mat{S}\right] + \oo_{d,\delta}(n^{-4/(r(d) + 4)}).
                \end{aligned}
            \end{equation}
        \end{theorem}

        A straightforward verification of the Lindeberg condition for double arrays yields the asymptotic normality.

        \begin{theorem}[Asymptotic normality]\label{thm:Theorem.3.2.and.3.3.Babu.Canty.Chaubey}
            Assume that \eqref{eq:assump:f.density} holds.
            Let $\mat{S}\in \mathcal{S}_{++}^{\hspace{0.3mm}d}$ be such that $f(\mat{S}) > 0$.
            If $n^{1/2} b^{r(d)/4}\to \infty$ as $n\to \infty$ and $b = b(n)\to 0$, then
            \begin{equation}\label{eq:thm:Theorem.3.2.and.3.3.Babu.Canty.Chaubey.Prop.1}
                n^{1/2} b^{r(d)/4} (\hat{f}_{n,b}(\mat{S}) - f_b(\mat{S})) \stackrel{\mathscr{D}}{\longrightarrow} \mathcal{N}(0,\psi(\mat{S}) f(\mat{S})).
            \end{equation}
            If we also have $n^{1/2} b^{r(d)/4 + 1/2}\to 0$ as $n\to \infty$ and $b = b(n)\to 0$, then Theorem~\ref{thm:bias} implies
            \begin{equation}
                n^{1/2} b^{r(d)/4} (\hat{f}_{n,b}(\mat{S}) - f(\mat{S})) \stackrel{\mathscr{D}}{\longrightarrow} \mathcal{N}(0,\psi(\mat{S}) f(\mat{S})).
            \end{equation}
            Independently of the above rates for $n$ and $b$, if we assume \eqref{eq:assump:f.density.2} instead and $n^{2/(r(d) + 4)} \, b \to \lambda$ as $n\to \infty$ and $b = b(n)\to 0$ for some $\lambda > 0$, then Theorem~\ref{thm:bias} implies
            \begin{equation}\label{eq:thm:Theorem.3.2.and.3.3.Babu.Canty.Chaubey.Thm.3.3}
                n^{2 / (r(d) + 4)} (\hat{f}_{n,b}(\mat{S}) - f(\mat{S})) \stackrel{\mathscr{D}}{\longrightarrow} \mathcal{N}(\lambda \, g(\mat{S}), \lambda^{-r(d)/2} \psi(\mat{S}) f(\mat{S})).
            \end{equation}
        \end{theorem}

        \begin{remark}
            {\normalfont
            The rate of convergence for the traditional $d(d + 1)/2$-dimensional kernel density estimator with i.i.d.\ data and bandwidth $h$ is $\OO_d(n^{-1/2} h^{-r(d)/2})$ in Theorem 3.1.15 of \citet{MR0740865}, whereas $\hat{f}_{n,b}$ converges at a rate of $\OO_d(n^{-1/2} b^{-r(d)/4})$.
            Hence, the relation between the bandwidth of $\hat{f}_{n,b}$ and the bandwidth of the traditional multivariate kernel density estimator is $b \approx h^2$.
            }
        \end{remark}

    \subsection{Total variation and other probability metrics upper bounds between the Wishart and SMN distributions}\label{sec:total.variation}

    Our second application of Theorem~\ref{thm:p.k.expansion} is to compute an upper bound on the total variation between the probability measures induced by \eqref{eq:Wishart.density} and \eqref{eq:sym.matrix.normal.density}.
    Given the relation there is between the total variation and other probability metrics such as the Hellinger distance (see, e.g., \citet[p.421]{doi:10.2307/1403865}), we obtain several other upper bounds automatically. For the uninitiated reader, the utility of having total variation or Hellinger distance bounds between two measures is discussed by \citet{Pollard_2005_Asymptopia_chap_3}.

    \begin{theorem}\label{thm:total.variation}
        Let $\nu > d - 1$ and $\mat{S}\in \mathcal{S}_{++}^{\hspace{0.3mm}d}$ be given.
        Let $\QQ_{\nu,\mat{S}}$ be the law of the $\mathrm{SMN}_{d\times d}(\nu \hspace{0.3mm} \mat{S}, B_d^{\top} (\hspace{-0.5mm} \sqrt{2\nu} \, \mat{S} \otimes \hspace{-1mm} \sqrt{2\nu} \, \mat{S}) B_d)$ distribution defined in \eqref{eq:sym.matrix.normal.density}, and let $\PP_{\nu,\mat{S}}$ be the law of the $\mathrm{Wishart}_d(\nu, \mat{S})$ distribution defined in \eqref{eq:Wishart.density}.
        Then, as $\nu\to \infty$, we have
        \begin{equation}
            \mathrm{dist}\hspace{0.3mm}(\PP_{\nu,\mat{S}},\QQ_{\nu,\mat{S}}) \leq \frac{C \hspace{0.3mm} d^{\hspace{0.3mm}3/2}}{\sqrt{\nu}} \qquad \text{and} \qquad \mathcal{H}(\PP_{\nu,\mat{S}},\QQ_{\nu,\mat{S}}) \leq \sqrt{\frac{2C \hspace{0.3mm} d^{\hspace{0.3mm}3/2}}{\sqrt{\nu}}},
        \end{equation}
        where $C > 0$ is a universal constant, $\mathcal{H}(\cdot,\cdot)$ denotes the Hellinger distance, and $\mathrm{dist}\hspace{0.3mm}(\cdot,\cdot)$ can be replaced by any of the following probability metrics: Total variation, Kolmogorov (or Uniform) metric, L\'evy metric, Discrepancy metric, Prokhorov metric.
    \end{theorem}

\section{Proofs}\label{sec:proofs}

    \begin{proof}[\bf Proof of Theorem~\ref{thm:p.k.expansion}]
        First, note that
        \begin{equation}
            \mat{S}^{-1/2} \hspace{0.3mm} \mat{X} \, \mat{S}^{-1/2} = \nu \, (\nu \hspace{0.3mm} \mat{S})^{-1/2} (\nu \hspace{0.3mm} \mat{S} + \mat{X} - \nu \hspace{0.3mm} \mat{S}) (\nu \hspace{0.3mm} \mat{S})^{-1/2} = \nu \, (\mathrm{I}_d + \sqrt{2/\nu} \, \Delta_{\nu,\mat{S}}),
        \end{equation}
        so we can rewrite \eqref{eq:Wishart.density} as
        \begin{equation}\label{eq:Wishart.density.rewrite}
            K_{\nu,\mat{S}}(\mat{X})
            = \frac{|\mathrm{I}_d + \sqrt{2 / \nu} \, \Delta_{\nu,\mat{S}}|^{\nu/2 - (d + 1)/2}}{2^{d/2} \pi^{\hspace{0.3mm}d(d + 1)/4} |\hspace{-0.5mm} \sqrt{2\nu} \, \mat{S}|^{(d + 1)/2}} \cdot \frac{\exp\left(-\frac{\nu}{2}\mathrm{tr}(\mathrm{I}_d + \sqrt{2 / \nu} \, \Delta_{\nu,\mat{S}})\right) \, \exp(\frac{\nu d}{2})}{\prod_{i=1}^d \frac{(\nu - (i + 1))^{(\nu - i)/2}}{e^{-(i + 1)/2} \, \nu^{(\nu - i)/2}} \cdot \prod_{i=1}^d \frac{\Gamma(\frac{1}{2} (\nu - (i + 1)) + 1)}{\sqrt{2\pi} e^{-\frac{1}{2} (\nu - (i + 1))} [\frac{1}{2} (\nu - (i + 1))]^{(\nu - i)/2}}}.
        \end{equation}
        Using the Taylor expansion
        \begin{equation}\label{eq:Taylor.log.1.minus.y}
            \log(1 - y) = - y - \frac{y^2}{2} - \frac{y^3}{3} - \frac{y^4}{4} + \OO(y^5), \quad |y| < 1,
        \end{equation}
        and Stirling's formula,
        \begin{equation}
            \log \Gamma(z + 1) = \frac{1}{2} \log(2\pi) + (z + \tfrac{1}{2}) \log z - z + \frac{1}{12z} + \OO(z^{-3}), \quad z > 0,
        \end{equation}
        see, e.g., \citet[p.257]{MR0167642}, we have
        \begin{align}\label{eq:expr.1}
            \log \left(\prod_{i=1}^d \frac{(\nu - (i + 1))^{(\nu - i)/2}}{e^{-(i + 1)/2} \, \nu^{(\nu - i)/2}}\right)
            &= \sum_{i=1}^d \frac{i + 1}{2} + \sum_{i=1}^d \frac{\nu - i}{2} \log\left(1 - \frac{i + 1}{\nu}\right) \notag \\
            &= \sum_{i=1}^d \frac{i + 1}{2} - \sum_{i=1}^d \left\{\frac{(\nu - i) (i + 1)}{2 \nu} + \frac{(\nu - i) (i + 1)^2}{4 \nu^2} + \OO_d(\nu^{-2})\right\} \notag \\
            &= \sum_{i=1}^d \left\{\nu^{-1} \cdot \frac{i^2 - 1}{4} + \OO_d(\nu^{-2})\right\} = \nu^{-1} \cdot \frac{d \, (2 d^{\hspace{0.3mm}2} + 3 d - 5)}{24} + \OO_d(\nu^{-2}),
        \end{align}
        and
        \begin{equation}\label{eq:expr.2}
            \log \left(\prod_{i=1}^d \frac{\Gamma(\frac{1}{2} (\nu - (i + 1)) + 1)}{\sqrt{2\pi} e^{-\frac{1}{2} (\nu - (i + 1))} [\frac{1}{2} (\nu - (i + 1))]^{(\nu - i)/2}}\right) = \sum_{i=1}^d \left\{\frac{1}{6 (\nu - (i + 1))} + \OO_d(\nu^{-3})\right\} = \nu^{-1} \cdot \frac{d}{6} + \OO_d(\nu^{-2}).
        \end{equation}
        By taking the logarithm in \eqref{eq:Wishart.density.rewrite} and using the expressions found in \eqref{eq:expr.1} and \eqref{eq:expr.2}, we obtain (also using the fact that $\lambda_i(\mathrm{I}_d + \sqrt{2/\nu} \, \Delta_{\nu,\mat{S}}) = 1 + \sqrt{2/\nu} \, \lambda_i(\Delta_{\nu,\mat{S}})$ for all $1 \leq i \leq d$):
        \begin{equation}
            \begin{aligned}
                \log K_{\nu,\mat{S}}(\mat{X})
                &= \frac{1}{2} (\nu - (d + 1)) \sum_{i=1}^d \log\left(1 + \sqrt{2 / \nu} \, \lambda_i(\Delta_{\nu,\mat{S}})\right) - \frac{1}{2} \log\left(2^d \pi^{\hspace{0.3mm}d(d + 1)/2} |\hspace{-0.5mm} \sqrt{2\nu} \, \mat{S}|^{d + 1}\right) - \frac{\nu}{2} \sum_{i=1}^d \sqrt{\frac{2}{\nu}} \, \lambda_i(\Delta_{\nu,\mat{S}}) \\[0.5mm]
                &\quad- \nu^{-1} \cdot \left\{\frac{d \, (2 d^{\hspace{0.3mm}2} + 3 d - 5)}{24} + \frac{d}{6}\right\} + \OO_d(\nu^{-2}).
            \end{aligned}
        \end{equation}
        By the Taylor expansion in \eqref{eq:Taylor.log.1.minus.y} and the fact that $\sum_{i=1}^d \lambda_i^k(\Delta_{\nu,\mat{S}}) = \mathrm{tr}(\Delta_{\nu,\mat{S}}^k)$ for all $k\in \N$, we have, uniformly for $\mat{X}\in B_{\nu,\mat{S}}(\eta)$,
        \begin{equation}
            \sum_{i=1}^d \log\left(1 + \sqrt{2 / \nu} \, \lambda_i(\Delta_{\nu,\mat{S}})\right) = \sqrt{\frac{2}{\nu}} \, \mathrm{tr}(\Delta_{\nu,\mat{S}}) - \frac{1}{\nu} \, \mathrm{tr}(\Delta_{\nu,\mat{S}}^2) + \frac{2\sqrt{2}}{3 \nu^{3/2}} \, \mathrm{tr}(\Delta_{\nu,\mat{S}}^3) - \frac{1}{\nu^2} \, \mathrm{tr}(\Delta_{\nu,\mat{S}}^4) + \OO_{d,\eta}\left(\frac{\max_{1 \leq i \leq d} |\lambda_i(\Delta_{\nu,\mat{S}})|^5}{\nu^{5/2}}\right).
        \end{equation}
        Therefore,
        \begin{equation}
            \begin{aligned}
                \log K_{\nu,\mat{S}}(\mat{X})
                &= - \frac{1}{2} \log\left(2^d \pi^{\hspace{0.3mm}d(d + 1)/2} |\hspace{-0.5mm} \sqrt{2\nu} \, \mat{S}|^{d + 1}\right) - \frac{1}{2} \mathrm{tr}(\Delta_{\nu,\mat{S}}^2) + \nu^{-1/2} \cdot \Bigg\{\frac{\sqrt{2}}{3} \, \mathrm{tr}(\Delta_{\nu,\mat{S}}^3) - \frac{d + 1}{\sqrt{2}} \, \mathrm{tr}(\Delta_{\nu,\mat{S}})\Bigg\} \\
                &\quad+ \nu^{-1} \cdot \left\{- \frac{1}{2} \, \mathrm{tr}(\Delta_{\nu,\mat{S}}^4) + \frac{d + 1}{2} \, \mathrm{tr}(\Delta_{\nu,\mat{S}}^2) - \left(\frac{d \, (2 d^{\hspace{0.3mm}2} + 3 d - 5)}{24} + \frac{d}{6}\right)\right\} + \OO_{d,\eta}\left(\frac{1 + \max_{1 \leq i \leq d} |\lambda_i(\Delta_{\nu,\mat{S}})|^5}{\nu^{3/2}}\right).
            \end{aligned}
        \end{equation}
        With the expression for the symmetric matrix-variate normal density in \eqref{eq:sym.matrix.normal.density}, we can rewrite the above as
        \begin{equation}\label{eq:log.p.before.next}
            \begin{aligned}
                \log \left(\frac{K_{\nu,\mat{S}}(\mat{X})}{g_{\nu,\mat{S}}(\mat{X})}\right)
                &= \nu^{-1/2} \cdot \Bigg\{\frac{\sqrt{2}}{3} \, \mathrm{tr}(\Delta_{\nu,\mat{S}}^3) - \frac{d + 1}{\sqrt{2}} \, \mathrm{tr}(\Delta_{\nu,\mat{S}})\Bigg\} + \nu^{-1} \cdot \left\{- \frac{1}{2} \, \mathrm{tr}(\Delta_{\nu,\mat{S}}^4) + \frac{d + 1}{2} \, \mathrm{tr}(\Delta_{\nu,\mat{S}}^2) - \left(\frac{d \, (2 d^{\hspace{0.3mm}2} + 3 d - 5)}{24} + \frac{d}{6}\right)\right\} \\
                &\quad+ \OO_{d,\eta}\left(\frac{1 + \max_{1 \leq i \leq d} |\lambda_i(\Delta_{\nu,\mat{S}})|^5}{\nu^{3/2}}\right),
            \end{aligned}
        \end{equation}
        which proves \eqref{eq:LLT.order.2.log}.
        To obtain \eqref{eq:LLT.order.2} and conclude the proof, we take the exponential on both sides of the last equation and we expand the right-hand side with
        \begin{equation}\label{eq:Taylor.exponential}
            e^y = 1 + y + \frac{y^2}{2} + \OO(e^{\widetilde{\eta}} y^3), \quad \text{for } -\infty < y \leq \widetilde{\eta}.
        \end{equation}
        For $\nu$ large enough and uniformly for $\mat{X}\in B_{\nu, \mat{S}}(\eta)$, the right-hand side of \eqref{eq:log.p.before.next} is $\OO_d(1)$, so we get
        \begin{equation}
            \begin{aligned}
                \frac{K_{\nu,\mat{S}}(\mat{X})}{g_{\nu,\mat{S}}(\mat{X})} = 1
                &+ \nu^{-1/2} \cdot
                    \Bigg\{\frac{\sqrt{2}}{3} \, \mathrm{tr}(\Delta_{\nu,\mat{S}}^3) - \frac{d + 1}{\sqrt{2}} \, \mathrm{tr}(\Delta_{\nu,\mat{S}})\Bigg\} + \nu^{-1} \cdot
                    \left\{\frac{1}{9} \, \left(\mathrm{tr}(\Delta_{\nu,\mat{S}}^3)\right)^2 - \frac{d + 1}{3} \, \mathrm{tr}(\Delta_{\nu,\mat{S}}^3) \, \mathrm{tr}(\Delta_{\nu,\mat{S}}) + \frac{(d + 1)^2}{4} \, \left(\mathrm{tr}(\Delta_{\nu,\mat{S}})\right)^2 \right. \\
                &\quad\left.- \frac{1}{2} \, \mathrm{tr}(\Delta_{\nu,\mat{S}}^4) + \frac{d + 1}{2} \, \mathrm{tr}(\Delta_{\nu,\mat{S}}^2) - \left(\frac{d \, (2 d^{\hspace{0.3mm}2} + 3 d - 5)}{24} + \frac{d}{6}\right)\right\} + \OO_{d,\eta}\left(\frac{1 + \max_{1 \leq i \leq d} |\lambda_i(\Delta_{\nu,\mat{S}})|^9}{\nu^{3/2}}\right).
            \end{aligned}
        \end{equation}
        This ends the proof.
    \end{proof}

    \begin{proof}[\bf Proof of Theorem~\ref{thm:bias}]
        Assume that \eqref{eq:assump:f.density.2} holds, and let
        \begin{equation}\label{eq:def.xi.Wishart.r.v}
            \mat{W}_{\mat{S}} \leqdef (\mat{W}_{\bb{i}})_{\bb{i}\in [d]^2} \sim \mathrm{Wishart}_d(1/b, b \hspace{0.3mm} \mat{S}), \quad \mat{S}\in \mathcal{S}_{++}^{\hspace{0.3mm}d}, ~b > 0.
        \end{equation}
        (Recall the notation $[d]\leqdef \{1,\dots,d\}$.)
        By a second order mean value theorem, we have
        \begin{equation}\label{eq:second.order.MVT}
            \begin{aligned}
                f(\mat{W}_{\mat{S}}) - f(\mat{S})
                &= \sum_{\substack{\bb{i}\in [d]^2 \\ i_1 \leq i_2}} (\mat{W}_{\bb{i}} - \mat{S}_{\bb{i}}) \frac{\partial}{\partial \mat{S}_{\bb{i}}} f(\mat{S}) + \frac{1}{2} \sum_{\substack{\bb{i},\bb{j}\in [d]^2 \\ i_1 \leq i_2, \, j_1 \leq j_2}} (\mat{W}_{\bb{i}} - \mat{S}_{\bb{i}}) (\mat{W}_{\bb{j}} - \mat{S}_{\bb{j}}) \frac{\partial^2}{\partial \mat{S}_{\bb{i}} \partial \mat{S}_{\bb{j}}} f(\mat{S}) \\[-1mm]
                &\quad+ \frac{1}{2} \sum_{\substack{\bb{i},\bb{j}\in [d]^2 \\ i_1 \leq i_2, \, j_1 \leq j_2}} (\mat{W}_{\bb{i}} - \mat{S}_{\bb{i}}) (\mat{W}_{\bb{j}} - \mat{S}_{\bb{j}}) \left(\frac{\partial^2}{\partial \mat{S}_{\bb{i}} \partial \mat{S}_{\bb{j}}} f(\mat{M}_{\mat{S}}) - \frac{\partial^2}{\partial \mat{S}_{\bb{i}} \partial \mat{S}_{\bb{j}}} f(\mat{S})\right),
            \end{aligned}
        \end{equation}
        for some random matrix $\mat{M}_{\mat{S}}\in \mathcal{S}_{++}^{\hspace{0.3mm}d}$ on the line segment joining $\mat{W}_{\mat{S}}$ and $\mat{S}$ in $\mathcal{S}_{++}^{\hspace{0.3mm}d}$. (The mean value theorem is applicable because the subspace $\mathcal{S}_{\scriptscriptstyle ++}^{\scriptscriptstyle d}$ is open and convex in the space of symmetric matrices of size $d \times d$, recall Remark~\ref{rem:open.and.convex}.)
        If we take the expectation in the last equation, and then use the estimates in \eqref{eq:expectation.explicit.estimate} and \eqref{eq:covariance.explicit.estimate}, we get
        \begin{align}\label{eq:prop.Berntein.decomposition}
            &\Bigg|f_b(\mat{S}) - f(\mat{S}) - \frac{b}{2} \sum_{\substack{\bb{i},\bb{j}\in [d]^2 \\ i_1 \leq i_2, \, j_1 \leq j_2}} \left[2 \hspace{0.3mm} B_d^{\top} (\mat{S} \otimes \mat{S}) B_d\right]_{(\bb{i},\bb{j})} \frac{\partial^2}{\partial \mat{S}_{\bb{i}} \partial \mat{S}_{\bb{j}}} f(\mat{S})\Bigg| \notag \\
            &\quad\leq \frac{1}{2} \sum_{\substack{\bb{i},\bb{j}\in [d]^2 \\ i_1 \leq i_2, \, j_1 \leq j_2}} \EE\left[|\mat{W}_{\bb{i}} - \mat{S}_{\bb{i}}| |\mat{W}_{\bb{j}} - \mat{S}_{\bb{j}}|  \cdot \left|\frac{\partial^2}{\partial \mat{S}_{\bb{i}} \partial \mat{S}_{\bb{j}}} f(\mat{M}_{\mat{S}}) - \frac{\partial^2}{\partial \mat{S}_{\bb{i}} \partial \mat{S}_{\bb{j}}} f(\mat{S})\right| \cdot \ind_{\{\|\mathrm{vec}(\mat{W}_{\mat{S}} - \mat{S})\|_1 \leq \delta_{\e,d}\}}\right] \notag \\
            &\qquad+ \frac{1}{2} \sum_{\substack{\bb{i},\bb{j}\in [d]^2 \\ i_1 \leq i_2, \, j_1 \leq j_2}} \EE\left[|\mat{W}_{\bb{i}} - \mat{S}_{\bb{i}}| |\mat{W}_{\bb{j}} - \mat{S}_{\bb{j}}| \cdot \left|\frac{\partial^2}{\partial \mat{S}_{\bb{i}} \partial \mat{S}_{\bb{j}}} f(\mat{M}_{\mat{S}}) - \frac{\partial^2}{\partial \mat{S}_{\bb{i}} \partial \mat{S}_{\bb{j}}} f(\mat{S})\right| \cdot \ind_{\{\|\mathrm{vec}(\mat{W}_{\mat{S}} - \mat{S})\|_1 > \delta_{\e,d}\}}\right] \notag \\[1mm]
            &\quad\reqdef\Delta_1 + \Delta_2
        \end{align}
        where for any given $\e > 0$, the real number $\delta_{\e,d}\in (0,1]$ is such that
        \begin{equation}\label{eq:second.order.partial.derivatives.cont}
            \|\mathrm{vec}(\mat{S}' - \mat{S})\|_1 \leq \delta_{d,\e} \quad \text{implies} \quad \left|\frac{\partial^2}{\partial \mat{S}_{\bb{i}} \partial \mat{S}_{\bb{j}}} f(\mat{S}') - \frac{\partial^2}{\partial \mat{S}_{\bb{i}} \partial \mat{S}_{\bb{j}}} f(\mat{S})\right| < \e,
        \end{equation}
        uniformly for $\mat{S},\mat{S}'\in \mathcal{S}_{++}^{\hspace{0.3mm}d}$. (We know that such a number exists because the second order partial derivatives of $f$ are assumed to be uniformly continuous on $\mathcal{S}_{++}^{\hspace{0.3mm}d}$.)
        Equations~\eqref{eq:second.order.partial.derivatives.cont} and \eqref{eq:covariance.explicit.estimate} then yield, together with the Cauchy-Schwarz inequality,
        \begin{equation}\label{eq:prop.Berntein.end.1}
            \Delta_1 \leq \frac{1}{2} \sum_{\substack{\bb{i},\bb{j}\in [d]^2 \\ i_1 \leq i_2, \, j_1 \leq j_2}} \e \cdot \sqrt{\EE\left[|\mat{W}_{\bb{i}} - \mat{S}_{\bb{i}}|^2\right]} \sqrt{\EE\left[|\mat{W}_{\bb{j}} - \mat{S}_{\bb{j}}|^2\right]} \, = \e \cdot \OO_{d,\mat{S}}(b).
        \end{equation}
        The second order partial derivatives of $f$ are also assumed to be bounded, say by some constant $M_d > 0$.
        Furthermore, $\{\|\mathrm{vec}(\mat{W}_{\mat{S}} - \mat{S})\|_1 > \delta_{\e,d}\}$ implies that at least one component of $(\mat{W}_{\bb{k}} - \mat{S}_{\bb{k}})_{\bb{k}\in [d]^2}$ is larger than $\delta_{\e,d} / d^{\hspace{0.3mm}2}$, so a union bound over $\bb{k}$ followed by $d^{\hspace{0.3mm}2}$ concentration bounds for the marginals of the Wishart distribution (the diagonal entries of a Wishart random matrix are chi-square distributed while the off-diagonal entries are variance-gamma distributed) yield
        \begin{equation}\label{eq:prop.Berntein.end.2}
            \Delta_2 \leq \OO_{d,\mat{S}}\left(\frac{1}{2} \sum_{\substack{\bb{i},\bb{j}\in [d]^2 \\ i_1 \leq i_2, \, j_1 \leq j_2}} 2 M_d \cdot \sqrt{\sum_{\bb{k}\in [d]^2} \PP\left(|\mat{W}_{\bb{k}} - \mat{S}_{\bb{k}}| \geq \delta_{\e,d} / d^{\hspace{0.3mm}2}\right)}\right) \leq \OO_{d,\mat{S}}\left(d^{\hspace{0.3mm}5} \, M_d \cdot \sqrt{2 \exp\left(- \frac{(\delta_{\e,d} / d^{\hspace{0.3mm}2})^2}{2 \cdot b c_{d,\mat{S}}}\right)}\right),
        \end{equation}
        where $c_{d,\mat{S}} > 0$ is a large enough constant that depends only on $d$ and $\mat{S}$.
        If we choose a sequence $\e = \e(b)$ that goes to $0$ as $b\to 0$ slowly enough that $1 \geq \delta_{\e,d} > d^2 \, [100 \, b c_{d,\mat{S}} \, |\log b|]^{1 / 2}$, for example, then $\Delta_1 + \Delta_2$ in \eqref{eq:prop.Berntein.decomposition} is $\oo_{d,\mat{S}}(b)$ by \eqref{eq:prop.Berntein.end.1} and \eqref{eq:prop.Berntein.end.2}.
        This ends the proof.
    \end{proof}

    \begin{proof}[\bf Proof of Theorem~\ref{thm:variance.near.boundary}]
        Assume that \eqref{eq:assump:f.density} holds.
        First, note that we can write
        \begin{equation}\label{eq:thm:bias.var.density.begin.variance}
            \hat{f}_{n,b}(\mat{S}) - f_b(\mat{S}) = \frac{1}{n} \sum_{i=1}^n Y_{i,b}(\mat{S}),
        \end{equation}
        where the random variables
        \begin{equation}\label{eq:Y.i.b.random.field}
            Y_{i,b}(\mat{S}) \leqdef K_{1/b, b \hspace{0.3mm} \mat{S}}(\mat{X}_i) - f_b(\mat{S}), ~~1 \leq i \leq n, \quad \text{are i.i.d.}
        \end{equation}
        Hence, if $\widetilde{\mat{W}}_{\mat{S}} \sim \mathrm{Wishart}_d(2/b - (d + 1), b \hspace{0.3mm} \mat{S} / 2)$, then
        \begin{align}\label{eq:thm:bias.var.density.variance.expression}
            \VV(\hat{f}_{n,b}(\mat{S}))
            &= n^{-1} \, \EE\left[K_{1/b, b \hspace{0.3mm} \mat{S}}(\mat{X})^2\right] - n^{-1} \left(f_b(\mat{S})\right)^2 = n^{-1} A_b(\mat{S}) \, \EE[f(\widetilde{\mat{W}}_{\mat{S}})] - \OO_{d,\mat{S}}(n^{-1}) \\[0.5mm]
            &= n^{-1} A_b(\mat{S}) \, (f(\mat{S}) + \OO_{d,\mat{S}}(b^{1/2})) - \OO_{d,\mat{S}}(n^{-1}),
        \end{align}
        where
        \begin{align}\label{eq:def.A.b.s}
            A_b(\mat{S})
            &\leqdef |2b\sqrt{\pi} \, \mat{S}|^{-\frac{d + 1}{2}} \, \pi^{\frac{d}{2}} \prod_{i=1}^d \frac{\Gamma\left(\frac{1}{b} - \frac{d + i}{2}\right)}{2^{\frac{1}{b} - i} \, \Gamma^{\hspace{0.3mm}2}\left(\frac{1}{2b} - \frac{i + 1}{2} + 1\right)},
        \end{align}
        and where the last line in \eqref{eq:thm:bias.var.density.variance.expression} follows from the Lipschitz continuity of $f$, the Cauchy-Schwarz inequality and the analogue of \eqref{eq:covariance.explicit.estimate} for $\widetilde{\mat{W}}_{\mat{S}} = (\widetilde{\mat{W}}_{\bb{i}})_{\bb{i}\in [d]^2}$:
        \begin{equation}\label{eq:special.eq}
            \EE[f(\widetilde{\mat{W}}_{\mat{S}})] - f(\mat{S}) = \sum_{\substack{\bb{i}\in [d]^2 \\ i_1 \leq i_2}} \OO\left(\EE\left[|\widetilde{\mat{W}}_{\bb{i}} - \mat{S}_{\bb{i}}|\right]\right)  \leq \sum_{\substack{\bb{i}\in [d]^2 \\ i_1 \leq i_2}} \OO\left(\sqrt{\EE\left[|\widetilde{\mat{W}}_{\bb{i}} - \mat{S}_{\bb{i}}|^2\right]}\right) = \OO_{d,\mat{S}}(b^{1/2}).
        \end{equation}
        Now, by Stirling's formula,
        \begin{equation}
            \frac{\sqrt{2\pi} e^{-x} x^{x + 1/2}}{\Gamma(x + 1)} \longrightarrow 1, \quad x\to \infty.
        \end{equation}
        Therefore,
        \begin{equation}\label{eq:lem:A.b.x.asymptotics.begin}
            \begin{aligned}
                A_b(\mat{S})
                &= |2b\sqrt{\pi} \, \mat{S}|^{-\frac{d + 1}{2}} \, \pi^{\frac{d}{2}} \cdot \frac{1}{(2\pi)^{d/2}} \, \prod_{i=1}^d \frac{e^{(d - i)/2} \left(\frac{1}{b} - \frac{d + i + 2}{2}\right)^{\frac{1}{b} - \frac{d + i + 1}{2}}}{2^{\frac{1}{b} - i} \left(\frac{1}{2b} - \frac{i + 1}{2}\right)^{\frac{1}{b} - i}} \cdot (1 + \OO(b)) \\
                &= |2b\sqrt{\pi} \, \mat{S}|^{-\frac{d + 1}{2}} \, \pi^{\frac{d}{2}} \cdot \frac{b^{\frac{d(d + 1)}{4}}}{(2\pi)^{d/2}} \cdot \prod_{i=1}^d e^{(d - i)/2} \left(1 - \frac{(d - i)/2}{\frac{1}{b} - (i + 1)}\right)^{1/b}\cdot (1 + \OO(b))
                = \frac{|\sqrt{b\pi} \, \mat{S}|^{-\frac{d + 1}{2}}}{2^{d(d + 2)/2}} \cdot (1 + \OO_d(b)),
            \end{aligned}
        \end{equation}
        where the last equality follows from the fact that $\lim_{n\to \infty} e^x (1 - \frac{x}{n})^n \to 1$ for all $x\in \R$.
        By our assumption on $\mat{S}$, note that $|\mat{S}| = b^{|\mathcal{J}|} \, |\mat{K}|$, so we get the general expression in \eqref{eq:thm:bias.var.density.eq.var} by combining \eqref{eq:thm:bias.var.density.variance.expression} and \eqref{eq:lem:A.b.x.asymptotics.begin}.
        This ends the proof.
    \end{proof}

    \begin{proof}[\bf Proof of Theorem~\ref{thm:MISE.optimal.density}]
        Assume that \eqref{eq:assump:f.density.2} holds.
        By Theorem~\ref{thm:bias}, Theorem~\ref{thm:variance.near.boundary} for $\mathcal{J} = \emptyset$, our assumptions in \eqref{eq:further.assumptions.MISE}, and the dominated convergence theorem, it is possible to show that
        \begin{equation}
            \begin{aligned}
                \mathrm{MISE}_{\delta}[\hat{f}_{n,b}]
                &= \int_{\mathcal{S}_{++}^{\hspace{0.3mm}d}(\delta)} \VV(\hat{f}_{n,b}(\mat{S})) \rd \mat{S} + \int_{\mathcal{S}_{++}^{\hspace{0.3mm}d}(\delta)} \BB[\hat{f}_{n,b}(\mat{S})]^2 \rd \mat{S} \\
                &= n^{-1} b^{-r(d)/2} \int_{\mathcal{S}_{++}^{\hspace{0.3mm}d}(\delta)} \psi(\mat{S}) f(\mat{S}) \rd \mat{S} + b^2 \int_{\mathcal{S}_{++}^{\hspace{0.3mm}d}(\delta)} g^2(\mat{S}) \rd \mat{S} + \oo_{d,\delta}(n^{-1} b^{-r(d)/2}) + \oo_{d,\delta}(b^2).
            \end{aligned}
        \end{equation}
        This ends the proof.
    \end{proof}

    \begin{proof}[\bf Proof of Theorem~\ref{thm:Theorem.3.2.and.3.3.Babu.Canty.Chaubey}]
        Assume that \eqref{eq:assump:f.density.2} holds.
        By \eqref{eq:thm:bias.var.density.begin.variance}, the asymptotic normality of $n^{1/2} b^{r(d)/4} (\hat{f}_{n,b}(\mat{S}) - f_b(\mat{S}))$ will be proved if we verify the following Lindeberg condition for double arrays (see, e.g., Section 1.9.3 in \cite{MR0595165}):
        For every $\e > 0$,
        \begin{equation}\label{eq:prop:Proposition.1.Babu.Canty.Chaubey.Lindeberg.condition}
            s_b^{-2} \, \EE\left[|Y_{1,b}(\mat{S})|^2 \, \ind_{\{|Y_{1,b}(\mat{S})| > \e n^{1/2} s_b\}}\right] \longrightarrow 0, \quad n\to \infty,
        \end{equation}
        where $s_b^2 \leqdef \EE\left[|Y_{1,b}(\mat{S})|^2\right]$ and $b = b(n)\to 0$.
        From Lemma~\ref{lem:Wishart.density.bound} with $\nu = 1/b$ and $\mat{M} = b \, \mat{S}$, we know that
        \begin{equation}
            |Y_{1,b}(\mat{S})| = \OO\left(\psi(\mat{S}) b^{-r(d)/2}\right) = \OO_{d,\mat{S}}(b^{-r(d)/2}),
        \end{equation}
        and we also know that $s_b = b^{-r(d)/4} \sqrt{\psi(\mat{S}) f(\mat{S})} \, (1 + \oo_{d,\mat{S}}(1))$ when $f$ is Lipschitz continuous and bounded, by the proof of Theorem~\ref{thm:variance.near.boundary}.
        Therefore, whenever $n^{1/2} b^{r(d)/4}\to \infty$ as $n\to \infty$ (and $b\to 0$), we have
        \begin{equation}\label{eq:prop:Proposition.1.Babu.Canty.Chaubey.Lindeberg.condition.verify}
            \frac{|Y_{1,b}(\mat{S})|}{n^{1/2} s_b} = \OO_{d,\mat{S}}(n^{-1/2} \, b^{r(d)/4} \, b^{-r(d)/2}) = \OO_{d,\mat{S}}(n^{-1/2} b^{-r(d)/4}) \longrightarrow 0.
        \end{equation}
        Under this condition, Equation~\eqref{eq:prop:Proposition.1.Babu.Canty.Chaubey.Lindeberg.condition} holds (since for any given $\e > 0$, the indicator function is equal to $0$ for $n$ large enough, independently of $\omega$) and thus
        \begin{equation}
            \begin{aligned}
                n^{1/2} b^{r(d)/4} (\hat{f}_{n,b}(\mat{S}) - f_b(\mat{S}))
                &= n^{1/2} b^{r(d)/4} \cdot \frac{1}{n} \sum_{i=1}^n Y_{i,m} \stackrel{\mathscr{D}}{\longrightarrow} \mathcal{N}(0,\psi(\mat{S}) f(\mat{S})).
            \end{aligned}
        \end{equation}
        This ends the proof.
    \end{proof}

    \begin{proof}[\bf Proof of Theorem~\ref{thm:total.variation}]
        By the comparison of the total variation norm $\|\cdot\|$ with the Hellinger distance on page 726 of \citet{MR1922539}, we already know that
        \begin{equation}\label{eq:first.bound.total.variation}
            \|\PP_{\nu,\mat{S}} - \QQ_{\nu,\mat{S}}\| \leq \sqrt{2 \, \PP\left(\mat{X}\in B_{\nu,\mat{S}}^{\hspace{0.3mm}c}(1/2)\right) + \EE\left[\log\Bigg(\frac{\rd \PP_{\nu,\mat{S}}}{\rd \QQ_{\nu,\mat{S}}}(\mat{X})\Bigg) \, \ind_{\{\mat{X}\in B_{\nu,\mat{S}}(1/2)\}}\right]}.
        \end{equation}
        Then, by applying a union bound followed by large deviation bounds on the eigenvalues of the Wishart matrix, we get, for $\nu$ large enough,
        \begin{equation}\label{eq:concentration.bound}
            \PP(\mat{X}\in B_{\nu,\mat{S}}^{\hspace{0.3mm}c}(1/2)) \leq \sum_{i=1}^d \PP\left(|\lambda_i(\Delta_{\nu,\mat{S}})| > \frac{\nu^{1/6}}{2\sqrt{2}}\right) \leq d \cdot 2 \, \exp\left(-\frac{\nu^{1/3}}{100}\right).
        \end{equation}
        By Theorem~\ref{thm:p.k.expansion}, we have
        \begin{equation}\label{eq:estimate.I.begin}
            \begin{aligned}
                \EE\left[\log\Bigg(\frac{\rd \PP_{\nu,\mat{S}}}{\rd \QQ_{\nu,\mat{S}}}(\mat{X})\Bigg) \, \ind_{\{\mat{X}\in B_{\nu,\mat{S}}(1/2)\}}\right]
                &= \nu^{-1/2} \cdot \Bigg\{\frac{\sqrt{2}}{3} \cdot \EE\left[\mathrm{tr}(\Delta_{\nu,\mat{S}}^3)\right] - \frac{d + 1}{\sqrt{2}} \cdot \EE\left[\mathrm{tr}(\Delta_{\nu,\mat{S}})\right]\Bigg\} \\
                &\quad+ \nu^{-1/2} \cdot \OO\left(\left|\EE\left[\mathrm{tr}(\Delta_{\nu,\mat{S}}^3) \, \ind_{\{\bb{\lambda}(\Delta_{\nu,\mat{S}})\in B_{\nu,\mat{S}}^{\hspace{0.3mm}c}(1/2)\}}\right]\right| + d \cdot \left|\EE\left[\mathrm{tr}(\Delta_{\nu,\mat{S}}) \, \ind_{\{\bb{\lambda}(\Delta_{\nu,\mat{S}})\in B_{\nu,\mat{S}}^{\hspace{0.3mm}c}(1/2)\}}\right]\right|\right) \\[0.5mm]
                &\quad+ \nu^{-1} \cdot \OO\left(\left|\EE\left[\mathrm{tr}(\Delta_{\nu,\mat{S}}^4)\right]\right| + d \cdot \left|\EE\left[\mathrm{tr}(\Delta_{\nu,\mat{S}}^2)\right]\right| + d^{\hspace{0.3mm}3}\right).
            \end{aligned}
        \end{equation}
        On the right-hand side, the first and third lines are estimated using Lemma~\ref{lem:Leblanc.2012.boundary.Lemma.1}, and the second line is bounded using Lemma~\ref{lem:Leblanc.2012.boundary.Lemma.1.with.set.A}.
        We find
        \begin{equation}\label{eq:estimate.I.begin.next}
            \EE\left[\log\Bigg(\frac{\rd \PP_{\nu,\mat{S}}}{\rd \QQ_{\nu,\mat{S}}}(\mat{X})\Bigg) \, \ind_{\{\mat{X}\in B_{\nu,\mat{S}}(1/2)\}}\right] = \OO(\nu^{-1} d^{\hspace{0.3mm}3}).
        \end{equation}
        Putting \eqref{eq:concentration.bound} and \eqref{eq:estimate.I.begin} together in \eqref{eq:first.bound.total.variation} gives the conclusion.
    \end{proof}

\appendix

\section{Technical computations}

    Below, we compute the expectations for the trace of powers (up to $4$) of a normalized Wishart matrix.
    The lemma is used to estimate some trace moments and the $\asymp \nu^{-1}$ errors in \eqref{eq:estimate.I.begin} of the proof of Theorem~\ref{thm:total.variation}, and also as a preliminary result for the proof of Lemma~\ref{lem:Leblanc.2012.boundary.Lemma.1.with.set.A}.

    \begin{lemma}\label{lem:Leblanc.2012.boundary.Lemma.1}
        Let $\nu > d - 1$ and $\mat{S}\in \mathcal{S}_{++}^{\hspace{0.3mm}d}$ be given.
        If $\mat{X}\sim \mathrm{Wishart}_d(\nu,\mat{S})$ according to \eqref{eq:Wishart.density}, then
        \begin{align}
            \EE\left[\mathrm{tr}(\Delta_{\nu,\mathrm{I}_d})\right]
            &= 0, \qquad
            \EE\left[\mathrm{tr}(\Delta_{\nu,\mathrm{I}_d}^2)\right]
            = \frac{d (d + 1)}{2}, \qquad
            \EE\left[\mathrm{tr}(\Delta_{\nu,\mathrm{I}_d}^3)\right]
            = \nu^{-1/2} \cdot \frac{d (d^{\hspace{0.3mm}2} + 3 \hspace{0.3mm} d + 4)}{2 \sqrt{2}}, \\
            \EE\left[\mathrm{tr}(\Delta_{\nu,\mathrm{I}_d}^4)\right]
            &= \frac{d (2 \hspace{0.3mm} d^{\hspace{0.3mm}2} + 5 \hspace{0.3mm} d + 5)}{4} + \nu^{-1} \cdot \frac{d (d^{\hspace{0.3mm}3} + 6 \hspace{0.3mm} d^{\hspace{0.3mm}2} + 21 \hspace{0.3mm} d + 20)}{4},
        \end{align}
        where recall $\Delta_{\nu,\mat{S}} \leqdef (\hspace{-0.5mm} \sqrt{2\nu} \, \mat{S})^{-1/2} (\mat{X} - \nu \hspace{0.3mm} \mat{S}) (\hspace{-0.5mm} \sqrt{2\nu} \, \mat{S})^{-1/2}$.
    \end{lemma}

    \begin{proof}[\bf Proof of Lemma~\ref{lem:Leblanc.2012.boundary.Lemma.1}]
        Let $\mat{Y} \leqdef \mat{S}^{-1/2} \hspace{0.3mm} \mat{X} \, \mat{S}^{-1/2} \sim \mathrm{Wishart}_d(\nu,\mathrm{I}_d)$.
        It was shown by \citet[p.308-310]{MR2066255} (another source could be \citet[p.66]{MR347003}, or \citet[Theorem~3.2]{MR1868979}, although the latter is less explicit) that
        \begin{align}
            \EE[\mat{Y}]
            &= \nu \, \mathrm{I}_d, \\[1mm]
            \EE[\mat{Y}^2]
            &= \nu \, \mathrm{I}_d \, \mathrm{tr}(\mathrm{I}_d) + (\nu^2 + \nu) \, \mathrm{I}_d^{\hspace{0.3mm}2} = \nu \hspace{0.3mm} d \, \mathrm{I}_d + (\nu^2 + \nu) \, \mathrm{I}_d, \\[1mm]
            \EE[\mat{Y}^3]
            &= \nu \, \mathrm{I}_d \, (\mathrm{tr}(\mathrm{I}_d))^2 + (\nu^2 + \nu) \, \left(\mathrm{I}_d \, \mathrm{tr}(\mathrm{I}_d^{\hspace{0.3mm}2}) + 2 \, \mathrm{I}_d^{\hspace{0.3mm}2} \, \mathrm{tr}(\mathrm{I}_d)\right) + (\nu^3 + 3 \, \nu^2 + 4 \, \nu) \, \mathrm{I}_d^{\hspace{0.3mm}3} \notag \\
            &= \nu \hspace{0.3mm} d^{\hspace{0.3mm}2} \, \mathrm{I}_d + 3 \, (\nu^2 + \nu) \hspace{0.3mm} d \, \mathrm{I}_d + (\nu^3 + 3 \, \nu^2 + 4 \, \nu) \, \mathrm{I}_d, \\[1mm]
            \EE[\mat{Y}^4]
            &= \nu \, \mathrm{I}_d \, (\mathrm{tr}(\mathrm{I}_d))^3 + 3 \, (\nu^2 + \nu) \, \left(\mathrm{I}_d \, \mathrm{tr}(\mathrm{I}_d) \, \mathrm{tr}(\mathrm{I}_d^{\hspace{0.3mm}2}) + \mathrm{I}_d^{\hspace{0.3mm}2} \, (\mathrm{tr}(\mathrm{I}_d))^2\right) + (\nu^3 + 3 \, \nu^2 + 4 \, \nu) \, \left(\mathrm{I}_d \, \mathrm{tr}(\mathrm{I}_d^{\hspace{0.3mm}3}) + 3 \, \mathrm{I}_d^{\hspace{0.3mm}3} \, \mathrm{tr}(\mathrm{I}_d)\right) \notag \\
            &\quad+ (2 \, \nu^3 + 5 \, \nu^2 + 5 \, \nu) \, \mathrm{I}_d^{\hspace{0.3mm}2} \, \mathrm{tr}(\mathrm{I}_d^{\hspace{0.3mm}2}) + (\nu^4 + 6 \, \nu^3 + 21 \, \nu^2 + 20 \, \nu) \, \mathrm{I}_d^4 \notag \\
            &= \nu \hspace{0.3mm} d^{\hspace{0.3mm}3} \, \mathrm{I}_d + 6 \, (\nu^2 + \nu) \hspace{0.3mm} d^{\hspace{0.3mm}2} \, \mathrm{I}_d + (6 \, \nu^3 + 17 \, \nu^2 + 21 \, \nu) \hspace{0.3mm} d \, \mathrm{I}_d + (\nu^4 + 6 \, \nu^3 + 21 \, \nu^2 + 20 \, \nu) \, \mathrm{I}_d,
        \end{align}
        from which we deduce the following:
        \begin{align}
            \EE[\mat{Y} - \nu \, \mathrm{I}_d]
            &= 0, \\[1mm]
            \EE[(\mat{Y} - \nu \, \mathrm{I}_d)^2]
            &= \EE[\mat{Y}^2] - (\nu \, \mathrm{I}_d)^2 = \left\{\nu \hspace{0.3mm} d \, \mathrm{I}_d + (\nu^2 + \nu) \, \mathrm{I}_d\right\} - \nu^2 \, \mathrm{I}_d = \nu \, (d + 1) \, \mathrm{I}_d, \\[1mm]
            \EE[(\mat{Y} - \nu \, \mathrm{I}_d)^3]
            &= \EE[\mat{Y}^3] - 3 \, \nu \, \EE[\mat{Y}^2] + 2 \, (\nu \, \mathrm{I}_d)^3 \notag \\
            &= \left\{\nu \hspace{0.3mm} d^{\hspace{0.3mm}2} \, \mathrm{I}_d + 3 \, (\nu^2 + \nu) \hspace{0.3mm} d \, \mathrm{I}_d + (\nu^3 + 3 \, \nu^2 + 4 \, \nu) \, \mathrm{I}_d\right\} - 3 \, \nu \, \mathrm{I}_d \, \left\{\nu \hspace{0.3mm} d \, \mathrm{I}_d + (\nu^2 + \nu) \, \mathrm{I}_d\right\} + 2 \, \nu^3 \, \mathrm{I}_d^{\hspace{0.3mm}3} \notag \\
            &= \nu \, (d^{\hspace{0.3mm}2} + 3 \hspace{0.3mm} d + 4) \, \mathrm{I}_d \\[1mm]
            \EE[(\mat{Y} - \nu \, \mathrm{I}_d)^4]
            &= \EE[\mat{Y}^4] - 4 \, (\nu \, \mathrm{I}_d) \, \EE[\mat{Y}^3] + 6 \, (\nu \, \mathrm{I}_d)^2 \, \EE[\mat{Y}^2] - 3 \, (\nu \, \mathrm{I}_d)^4 \notag \\
            &= \nu \hspace{0.3mm} d^{\hspace{0.3mm}3} \, \mathrm{I}_d + 6 \, (\nu^2 + \nu) \hspace{0.3mm} d^{\hspace{0.3mm}2} \, \mathrm{I}_d + (6 \, \nu^3 + 17 \, \nu^2 + 21 \, \nu) \hspace{0.3mm} d \, \mathrm{I}_d  + (\nu^4 + 6 \, \nu^3 + 21 \, \nu^2 + 20 \, \nu) \, \mathrm{I}_d \notag \\
            &\quad- 4 \, \nu \, \mathrm{I}_d \, \left\{\nu \hspace{0.3mm} d^{\hspace{0.3mm}2} \, \mathrm{I}_d + 3 \, (\nu^2 + \nu) \hspace{0.3mm} d \, \mathrm{I}_d + (\nu^3 + 3 \, \nu^2 + 4 \, \nu) \, \mathrm{I}_d\right\} + 6 \, \nu^2 \, \mathrm{I}_d^{\hspace{0.3mm}2} \, \left\{\nu \hspace{0.3mm} d \, \mathrm{I}_d + (\nu^2 + \nu) \, \mathrm{I}_d\right\} - 3 \, \nu^4 \, \mathrm{I}_d^4 \notag \\
            &= \nu^2 \, (2 \hspace{0.3mm} d^{\hspace{0.3mm}2} + 5 \hspace{0.3mm} d + 5) \, \mathrm{I}_d + \nu \, (d^{\hspace{0.3mm}3} + 6 \hspace{0.3mm} d^{\hspace{0.3mm}2} + 21 \hspace{0.3mm} d + 20) \, \mathrm{I}_d.
        \end{align}
        By the linearity of expectations, we have
        \begin{equation}
            \EE\left[\mathrm{tr}(\Delta_{\nu,\mathrm{I}_d}^k)\right] = (2\nu)^{-k/2} \, \mathrm{tr}\left(\EE[(\mat{Y} - \nu \, \mathrm{I}_d)^k]\right),
            \quad \text{for any } k\in \N.
        \end{equation}
        The conclusion follows.
    \end{proof}

    We can also estimate the moments of Lemma~\ref{lem:Leblanc.2012.boundary.Lemma.1} on various events.
    The lemma below is used to estimate the $\asymp \nu^{-1/2}$ errors in \eqref{eq:estimate.I.begin} of the proof of Theorem~\ref{thm:total.variation}.

    \begin{lemma}\label{lem:Leblanc.2012.boundary.Lemma.1.with.set.A}
        Let $\nu > d - 1$ and $\mat{S}\in \mathcal{S}_{++}^{\hspace{0.3mm}d}$ be given, and let $A\in \mathscr{B}(\R^d)$ be a Borel set.
        If $\mat{X}\sim \mathrm{Wishart}_d(\nu,\mat{S})$ according to \eqref{eq:Wishart.density}, then, for $\nu$ large enough,
        \begin{align}
            &\left|\EE\left[\mathrm{tr}(\Delta_{\nu,\mat{S}}) \ind_{\{\bb{\lambda}(\Delta_{\nu,\mat{S}})\in A\}}\right]\right| \leq d^{\hspace{0.3mm}3/2} \, \left(\PP\left(\bb{\lambda}(\Delta_{\nu,\mat{S}})\in A^c\right)\right)^{1/2}\hspace{-0.5mm}, \label{eq:thm:central.moments.eq.1.set.A} \\[1mm]
            &\left|\EE\left[\mathrm{tr}(\Delta_{\nu,\mat{S}}^3) \ind_{\{\bb{\lambda}(\Delta_{\nu,\mat{S}})\in A\}}\right] - \nu^{-1/2} \cdot \frac{d (d^{\hspace{0.3mm}2} + 3 \hspace{0.3mm} d + 4)}{2 \sqrt{2}}\right| \leq 3 \hspace{0.3mm} d^{\hspace{0.3mm}5/2} \, \left(\PP\left(\bb{\lambda}(\Delta_{\nu,\mat{S}})\in A^c\right)\right)^{1/4}\hspace{-0.5mm}, \label{eq:thm:central.moments.eq.3.set.A}
        \end{align}
        where recall $\Delta_{\nu,\mat{S}} \leqdef (\hspace{-0.5mm} \sqrt{2\nu} \, \mat{S})^{-1/2} (\mat{X} - \nu \hspace{0.3mm} \mat{S}) (\hspace{-0.5mm} \sqrt{2\nu} \, \mat{S})^{-1/2}$.
    \end{lemma}

    \begin{proof}[\bf Proof of Lemma~\ref{lem:Leblanc.2012.boundary.Lemma.1.with.set.A}]
        By Lemma~\ref{lem:Leblanc.2012.boundary.Lemma.1}, the Cauchy-Schwarz inequality, and Jensen's inequality ($(\mathrm{tr}(\Delta_{\nu,\mat{S}}))^2 \leq d \cdot \mathrm{tr}(\Delta_{\nu,\mat{S}}^2)$), we have
        \begin{equation}
            \begin{aligned}
                \left|\EE\left[\mathrm{tr}(\Delta_{\nu,\mat{S}}) \ind_{\{\bb{\lambda}(\Delta_{\nu,\mat{S}})\in A\}}\right]\right|
                &= \left|\EE\left[\mathrm{tr}(\Delta_{\nu,\mat{S}}) \ind_{\{\bb{\lambda}(\Delta_{\nu,\mat{S}})\in A^c\}}\right]\right| \leq \left(\EE\left[(\mathrm{tr}(\Delta_{\nu,\mat{S}}))^2\right]\right)^{1/2} \left(\PP\left(\bb{\lambda}(\Delta_{\nu,\mat{S}})\in A^c\right)\right)^{1/2} \\[1mm]
                &\leq \left(d \cdot \EE\left[\mathrm{tr}(\Delta_{\nu,\mat{S}}^2)\right]\right)^{1/2} \left(\PP\left(\bb{\lambda}(\Delta_{\nu,\mat{S}})\in A^c\right)\right)^{1/2} \leq d^{\hspace{0.3mm}3/2} \, \left(\PP\left(\bb{\lambda}(\Delta_{\nu,\mat{S}})\in A^c\right)\right)^{1/2}\hspace{-0.5mm},
            \end{aligned}
        \end{equation}
        which proves \eqref{eq:thm:central.moments.eq.1.set.A}.
        Similarly, by Lemma~\ref{lem:Leblanc.2012.boundary.Lemma.1}, Holder's inequality, and Jensen's inequality ($(\mathrm{tr}(\Delta_{\nu,\mat{S}}^3))^{4/3} \leq d^{\hspace{0.3mm}1/3} \mathrm{tr}(\Delta_{\nu,\mat{S}}^4)$), we have, for $\nu$ large enough,
        \begin{equation}
            \begin{aligned}
                \left|\EE\left[\mathrm{tr}(\Delta_{\nu,\mat{S}}^3) \ind_{\{\bb{\lambda}(\Delta_{\nu,\mat{S}})\in A\}}\right] - \nu^{-1/2} \cdot \frac{d (d^{\hspace{0.3mm}2} + 3 \hspace{0.3mm} d + 4)}{2 \sqrt{2}}\right|
                &= \left|\EE\left[\mathrm{tr}(\Delta_{\nu,\mat{S}}^3) \ind_{\{\bb{\lambda}(\Delta_{\nu,\mat{S}})\in A^c\}}\right]\right| \leq \left(\EE\left[(\mathrm{tr}(\Delta_{\nu,\mat{S}}^3))^{4/3}\right]\right)^{3/4} \left(\PP\left(\bb{\lambda}(\Delta_{\nu,\mat{S}})\in A^c\right)\right)^{1/4} \\[1mm]
                &\leq \left(d^{\hspace{0.3mm}1/3} \, \EE\left[\mathrm{tr}(\Delta_{\nu,\mat{S}}^4)\right]\right)^{3/4} \left(\PP\left(\bb{\lambda}(\Delta_{\nu,\mat{S}})\in A^c\right)\right)^{1/4} \leq 3 \hspace{0.3mm} d^{\hspace{0.3mm}5/2} \, \left(\PP\left(\bb{\lambda}(\Delta_{\nu,\mat{S}})\in A^c\right)\right)^{1/4}\hspace{-0.5mm},
            \end{aligned}
        \end{equation}
        which proves \eqref{eq:thm:central.moments.eq.3.set.A}.
        This ends the proof.
    \end{proof}

    In the next lemma, we bound the density of the $\mathrm{Wishart}_d(\nu,\mat{M})$ distribution from \eqref{eq:Wishart.density} when $\nu > d + 1$.

    \begin{lemma}\label{lem:Wishart.density.bound}
        If $\nu > d + 1$ and $\mat{M}\in \mathcal{S}_{++}^{\hspace{0.3mm}d}$, then
        \begin{equation}\label{eq:lem:Wishart.density.bound}
            \sup_{\mat{X}\in \mathcal{S}_{++}^{\hspace{0.3mm}d}} K_{\nu,\mat{M}}(\mat{X}) \leq \frac{(2\pi / e)^{-\frac{d(d + 1)}{4}} |\mat{M}|^{-(d + 1)/2}}{(2 e)^{d/2} (\nu - (d + 1))^{d(d + 1)/4}}.
        \end{equation}
    \end{lemma}

    \begin{proof}[\bf Proof of Lemma~\ref{lem:Wishart.density.bound}]
        When $\nu > d + 1$, it is easily verified that the mode of the $\mathrm{Wishart}_d(\nu,\mat{M})$ distribution is $(\nu - (d + 1)) \, \mat{M}$, so the expression for the Wishart density in \eqref{eq:Wishart.density} yields
        \begin{equation}
            \sup_{\mat{X}\in \mathcal{S}_{++}^{\hspace{0.3mm}d}} K_{\nu,\mat{M}}(\mat{X}) = \frac{(\nu - (d + 1))^{\frac{\nu d}{2} - \frac{d(d + 1)}{2}} |\mat{M}|^{-(d + 1)/2} \exp\big(-\frac{d}{2} (\nu - (d + 1))\big)}{2^{\nu d/2} \pi^{\hspace{0.3mm}d(d-1)/4} \prod_{i=1}^d \Gamma(\frac{1}{2} (\nu - (i + 1)) + 1)}.
        \end{equation}
        From Lemma~1 in \cite{MR162751}, we know that, for all $y > 1$, $\sqrt{2\pi e} \, e^{-(y - \frac{1}{2})} (y - 1)^{y - \frac{1}{2}} \leq \Gamma(y)$, so we get
        \begin{equation}
            \sup_{\mat{X}\in \mathcal{S}_{++}^{\hspace{0.3mm}d}} K_{\nu,\mat{M}}(\mat{X}) \leq \frac{(\nu - (d + 1))^{\frac{\nu d}{2} - \frac{d(d + 1)}{2}} |\mat{M}|^{-(d + 1)/2} \, e^{-\frac{\nu d}{2} + \frac{d}{2}(d + 1)}}{2^{\nu d/2} \pi^{\hspace{0.3mm}d(d-1)/4} \cdot (2\pi e)^{d/2} e^{-\frac{\nu d}{2} + \frac{d}{4}(d + 1)} \big[\frac{1}{2} (\nu - (d + 1))\big]^{\frac{\nu d}{2} - \frac{d(d + 1)}{4}}} = \frac{(2\pi / e)^{-\frac{d(d + 1)}{4}} |\mat{M}|^{-(d + 1)/2}}{(2 e)^{d/2} (\nu - (d + 1))^{d(d + 1)/4}}.
        \end{equation}
        This ends the proof.
    \end{proof}

\section{Acronyms}\label{sec:acronyms}

\begin{tabular}{ll}
    CMB & cosmic microwave background \\
    i.i.d. & independent and identically distributed \\
    SMN & symmetric matrix-variate normal \\
    SPD & symmetric positive definite
\end{tabular}

\section{Simulation code}\label{sec:R.code}

Supplementary material related to this article can be found online at \href{https://doi.org/10.1016/j.jmva.2021.104}{https://doi.org/10.1016/j.jmva.2021.104}.

\section*{Acknowledgments}

First, I would like to thank Donald Richards for his indications on how to calculate the moments in Lemma~\ref{lem:Leblanc.2012.boundary.Lemma.1}.
I also thank the Editor, the Associate Editor and the referees for their insightful remarks which led to improvements in the presentation of this paper.
The author is supported by postdoctoral fellowships from the NSERC (PDF) and the FRQNT (B3X supplement and B3XR).

%
%

\phantomsection
\addcontentsline{toc}{chapter}{References}

\bibliographystyle{myjmva}
\bibliography{Ouimet_2021_LLT_Wishart_bib}

\end{document}